\documentclass[10pt,a4paper]{article}

\usepackage[left=3cm, right=3cm, bottom=3cm]{geometry} 

\usepackage[T1]{fontenc} 
\usepackage{ngerman} 

\usepackage[english]{babel}

\newcommand{\footremember}[2]{%
	\footnote{#2}
	\newcounter{#1}
	\setcounter{#1}{\value{footnote}}%
}
 
\usepackage{palatino}

\usepackage{amsmath}
\usepackage{amsthm}
\usepackage{amssymb}
\usepackage{url}
\usepackage[square,sort,comma,numbers]{natbib}

\usepackage{hyperref}
\usepackage{graphicx}
\usepackage{url}
\usepackage{listings}

\makeindex

\allowdisplaybreaks

\newtheorem{Theorem}{Theorem}[section]
\newtheorem{Definition}[Theorem]{Definition} 
\newtheorem{Lemma}[Theorem]{Lemma}	
\newtheorem{Proposition}[Theorem]{Proposition}
\newtheorem{Corollary}[Theorem]{Corollary}

\numberwithin{equation}{section} 
\newcommand{\C}{\ensuremath{\mathbb{C}}}
\newcommand{\R}{\ensuremath{\mathbb{R}}}
\newcommand{\Q}{\ensuremath{\mathbb{Q}}}
\newcommand{\E}{\ensuremath{\mathbb{E}}}
\newcommand{\PP}{\ensuremath{\mathbb{P}}}
\newcommand{\F}{\ensuremath{\mathbb{F}}}

\newcommand{\N}{\ensuremath{\mathbb{N}}}

\newcommand{\Div}{\ensuremath{\operatorname{div}}}

\newcommand{\ft}{\ensuremath{\mathcal{F}_t}}

\newcommand{\ind}{\ensuremath{\mathbf{1}}}

\newcommand{\re}{\operatorname{Re}}

\DeclareMathOperator*{\esssup}{\operatorname{ess\,sup}}
\DeclareMathOperator*{\essinf}{\operatorname{ess\,inf}}

\begin{document}
\title{Quasilinear parabolic stochastic evolution equations via maximal \texorpdfstring{$ L^{p} $}{Lp}-regularity}

\author{Luca Hornung\footremember{address}{{Institute for Analysis, Karlsruhe Institute of Technology, D-76128 Karlsruhe, Germany (E-mail adress: luca.hornung@kit.edu)}} }

\date{\today}
        \maketitle
     
    \begin{abstract}
    	\noindent
We study the Cauchy problem for an abstract quasilinear stochastic parabolic evolution equation on a Banach space driven by a cylindrical Brownian motion. We prove existence and uniqueness of a local strong solution up to a maximal stopping time, that is characterised by a blow-up alternative. The key idea is an iterative application of the theory about maximal $ L^p $- regularity for semilinear stochastic evolution equations by Van Neerven, Veraar and Weis. We apply our local well-posedness result to a convection-diffusion equation on a bounded domain with Dirichlet, Neumann or mixed boudary conditions and to a generalized Navier-Stokes equation describing non-Newtonian fluids. In the first example, we can even show that the solution exists globally.

\medskip


    \end{abstract}

    \section{Introduction}

In this article, we develop an abstract framework to deal with quasilinear stochastic evolution equations driven by a cylindrical Brownian motion, like the convection-diffusion equation
\begin{equation}\label{diffusion_convection_equation}
\begin{cases}
du(t)=\big[\Div(a(u(t))\nabla u(t))-\Div(G(u))\big]dt+\sum_{n=1}^{\infty}B_n(u(t))d\beta_n(t), t\in (0,T]\\
u(0)=u_0,
\end{cases}
\end{equation} 
on a domain $ D\subset \R^{d} $ with Dirichlet, Neumann or even mixed boundary conditions, a uniformly positive definite diffusion matrix $ a $ and a sequence of independent Brow\-nian motions $ (\beta_n)_{n\in\N}. $ \eqref{diffusion_convection_equation} describes certain phenomena where particles, energy, or other physical quantities are transferred inside a physical system due to diffusion and convection. Here the stochastic perturbation represents random external perturbation or a lack of knowledge of certain parameters.	\\ \\
We consider the abstract quasilinear stochastic evolution equations 
\begin{equation*}
(\operatorname{QSEE})\begin{cases}
du(t)=\big[-A(u(t))u(t)+F(u(t))\big]dt+B(u(t))dW(t),\ t\in (0,T],\\
u(0)=u_0,
\end{cases}
\end{equation*} 
on a Banach space $ E $ in a setting that covers \eqref{diffusion_convection_equation} and we aim to establish existence and uniqueness of local strong solutions of $ (\operatorname{QSEE}) $ up to a maximal stopping time $ \tau $.\\ \\
In a deterministic setting, these equations have been studied for more than 30 years using strong linearisation techniques relying on the solvability of non-autonomous equations under certain H\"older continuity assumptions (see e.g. \cite{lunardi_abstract_quasilinear_parabolic_equation}, \cite{Yagi_abstract_parabolic_evolution_equations_and_their_applications}) or relying on maximal $ L^p $-regularity  (see e.g.  \cite{amann_linear_and_quasilinear_parabolic_problems}, \cite{amann_quasilinear_parabolic_via_maximal_regularity}, \cite{clement_li_abstract_parabolic_quasilinear_equations_and_application}, \cite{pruess_simonett_moving_interfaces_and_quasilinear}). In a stochastic setting, concrete quasilinear equations have been extensively studied in the literature in case of monotone coefficients (see e.g. \cite{krylov_stochastic_evolution_equation}, \cite{prevot_roeckner_a_concise_course_on_spde}, \cite{barbu_prato_roeckner_existence_uniqueness_nonnegative_stochastic_porous_media} and \cite{barbu_prato_roeckner_existence_uniqueness_stochastic_porous_media} and the references therein). In the same spirit is \cite{liu_roeckner_spde_in_hilbert_space_locally_monotone_coefficients}, where the authors extend the known results to the case of locally monotone coefficients. Existence and uniqueness of \eqref{diffusion_convection_equation} on the torus $ \mathbb{T}^d $ with a globally Lipschitz continuous quasilinearity $ a $ was proven in \cite{hofmanova_zhang_existence_uniqueness_quasilinear} by approximating $ a $ and $ G $ with locally monotone coefficients, deriving uniform estimates for the solutions of the approximating problem and passing to the limit afterwards. However, as far as we know, there are no abstract results comparable to the state of knowledge in the deterministic case. One reason might be, that the theory of maximal regularity for the stochastic integral in Banach spaces has not been developed until 2012. Then, Van Neerven, Veraar and Weis found out (see \cite{nerven_veraar_weis_maximal_lp_regularity_stochastic_convolution}), that if $ E=L^{q}(O,\mu) $ for some $ q\geq 2 $ (or more general, $ E $ is a UMD Banach space of type $ 2 $) and the operator $ B $ has a bounded $ H^{\infty}(\Sigma_{\theta}) $-calculus on a sector $ \Sigma_{\theta}  $ with $ \theta\in (0,\pi/2), $ one has
\[ \big\|t\mapsto\sum_{k=1}^{\infty}\int_{0}^{t}B^{1/2}e^{-B(t-s)}g_k(s)d\beta_s\big\|_{L^{p}(\Omega\times\R_{\geq 0};L^{q}(O))} \leq C_{\operatorname{MRS}}\|g\|_{L^{p}(\Omega\times\R_{\geq 0};L^{q}(O;l^2))}. \]
for $ p>2 $. Together with the well-known deterministic maximal regularity result
\[ \big\|t\mapsto\int_{0}^{t}Be^{-B(t-s)}f(s)\operatorname{ds}\big\|_{L^{p}(\R_{\geq 0};L^{q}(O))} \leq C_{\operatorname{MRD}}\|f\|_{L^{p}(\R_{\geq 0};L^{q}(O))}, \]
this led to a new theory for semilinear stochastic evolution equations of the form
\begin{equation*}
\operatorname{(SEE)}\begin{cases}
du(t)=\big[-Bu(t)+G(u(t))\big]dt+B(u(t))dW(t),\ t\in (0,T],\\
u(0)=u_0,
\end{cases}
\end{equation*} 
where nonlinearites $ G:[0,T]\times D(B)\to L^{q}(O) $ and $ B:[0,T]\times D(B)\to D(B^{1/2}) $ were allowed, as long as they were Lipschitz continuous with a small Lipschitz constant (see \cite{nerven_veraar_weis_maximal_lp_stochastic_evolution}). \\ \\
Now, we briefly describe our main assumptions and our strategy. We also work on UMD Banach spaces $ E $ of type $ 2 $, e.g. $ E=L^{q}(O;\mu) $ for $ q\geq 2 $ and we choose $ p>2. $  We assume, that the domain of the operators $ A(z) $ is constant, i.e. there exists a Banach space $ E^1 $ such that $ D(A(z))=E^1 $ for every $ z\in  \big(E,E^1\big)_{1-1/p,p}. $ and we demand $ A $ to be at least locally Lipschitz continuous, i.e. for every $ R>0 $ there exists $ L(R)>0 $ such that
$$ \|A(z)-A(y)\|_{B(E^1,E)}\leq L(R)\|z-y\|_{(E,E^1)_{1-1/p,p}} $$ 
for every $ y,z $ with $ \|y\|_{(E,E^1)_{1-1/p,p}},\|z\|_{(E,E^1)_{1-1/p,p}}\leq R $. As a first step, we consider 
\[\widetilde{F}_1(u(t))=\theta_{\lambda}\big(\sup\limits_{s\in[0,t]}\|u(s)-u_0\|_{(E,E^1)_{1-1/p,p}}+\|u\|_{L^{p}(0,t;E^1)}\big)\big(A(u(t))u(t)-A(u_0)u(t)\big),\]
where $ \theta_{\lambda}:\R_{\geq 0}\to [0,1] $ is a Lipschitz continuous cut-off function, that equals one on $ [0,\lambda] $ and vanishes on $ [2\lambda,\infty) $. This means, that as long as $ u(t) $ is close enough to $ u_0 $ and $ \|u\|_{L^{p}(0,t;E^1)} $ is small, we have $ A(u(t))u(t)=A(u_0)u(t)+\widetilde{F}_1(u(t)). $
We prove that $ \widetilde{F}_1 $ has a Lipschitz constant proportional to $ \lambda $ and thus, choosing $ \lambda $ small enough, satisfies the assumptions needed to solve $ (\operatorname{SEE}) $ with $ B=A(u_0) $ and $ G(u(t))=-\widetilde{F}_1(u(t))+F(u(t)). $ 
A solution of $ (\operatorname{SEE}) $ exists on $ [0,T], $ but it just solves $ (\operatorname{QSEE}) $ on the random interval $ [0,\tau_1], $ where $ \tau_1 $ is a stopping time given by
\[ \tau_1=\inf\big\{t\in [0,T]:\|u(t)-u_0\|_{(E,E^1)_{1-1/p,p}}+\|u\|_{L^{p}(0,t;E^1)}>\lambda\big \}. \]
Then, we iterate this procedure with the new initial value $ u(\tau_1) $ and set
\[\widetilde{F}_2(u(t))=\theta_{\lambda}\big(\sup_{s\in[\tau_1,t]}\|u(s)-u(\tau_1)\|_{(E,E^1)_{1-1/p,p}}+\|u\|_{L^{p}(\tau_1,t;E^1)}\big)\big(A(u(t))u(t)-A(u_{\tau_1})u(t)\big)\]
for $ t\in [\tau_1,T] $ and get a stopping time $ \tau_2 $ and a solution of $ (\operatorname{QSEE}) $ on $ [\tau_1,\tau_2]. $ Inductively, we construct a sequence of stopping times $ (\tau_n)_n $ such that $ u $ solves $ (\operatorname{QSEE}) $ on $ [0,\tau_n] $ for every $ n\in\N. $ Finally, we define $ \tau=\sup_{n\in\N}\tau_n $ and show that the solution of $(\operatorname{QSEE})$ is unique and satisfies 
\[ \mathbb{P}\Big\{\tau<T, \|u\|_{L^{p}(0,\tau;E^{1})}<\infty, u:[0,\tau)\to (E,E^1)_{1-1/p,p} \text{ is uniformly continuous}\Big \}= 0.  \]
This means, that if one wants to prove that a local solution exists in the whole interval $ [0,T], $ it is sufficient to show uniform continuity with values in $ (E,E^1)_{1-1/p,p} $ and to control the quantity $ \|u\|_{L^{p}(0,\tau;E^{1})} $ pathwise.\\ \\
All in all, the main advantage of our abstract operator theoretic approach is, that we can cover a wide range of applications. For example, our assumption on the $ H^{\infty} $-calculus is satisfied in many situations. As a rule of thumb, one might say, that almost all elliptic differential operators of even order have such a calculus (see e.g. \cite{amann_hieber_simonet_bounded_functional_calculus_elliptic}, \cite{denk_hieber_pruss_new_thoughts_on_old_results}, \cite{denk_hieber_pruss_r_boundedness_fourier_multiplier}, \cite{duong_mcintosh_functional_calculus_of_second_order_measurable},
\cite{duong_robinson_semigroup_kernels_holomorphic_functional_calculus},
\cite{duong_simonett_functional_calculus_elliptic_operators}, \cite{duong_yan_bounded_holomorphic_functional_calculus_for_non-divergence}, \cite{kalton_kunstmann_weis_pertubation_and_interpolation_for_functional_calculus}, \cite{kalton_weis_functional_calculus_and_sum_of_closed_operators},
\cite{Kunstmann_Weis_Lecture_Notes} and \cite{mcintosh_operators_which_have_a_functional_calculus}).

The full power of our setting can be seen in our examples. First, we discuss a quasilinear elliptic equation on $ \R^{d} $ of the form
\begin{equation*}
\begin{cases}
du(t)=\big[\sum_{i,j=1}^{d}a_{ij}(\cdot,u(t),\nabla u(t))\partial_i\partial_ju(t)+ f(t)\big]dt+B(t,u(t))dW(t),\\
u(0)=u_0.
\end{cases}
\end{equation*}
As far as we know, this equation has not been studied in a stochastic setting, since the usual monotonicity methods are not available and it is difficult to derive energy estimates. 

Next, we give an application to fluid dynamics. The example of a generalized Navier-Stokes equation for non-Newtonian fluids is inspired by the deterministic work of Bothe and Pr\"uss in \cite{bothe_pruss_navier_stokes}. The stochastic noise perturbation occur in the context of turbulences, for example in the Kreichnan model. 

Last but not least, we discuss the above mentioned convection-diffusion equation $ \eqref{diffusion_convection_equation} $ with mixed boundary conditions. We prove existence and uniqueness of local weak solutions in the sense of partial differential equations, i.e. we consider it as an equation in $ W^{-1,q}(D). $ In this application, we make use of the great progress within the last years concerning mixed boundary problems in $ W^{-1,q}(D) $ for $ q>2. $ Exemplary, we need the square-root-property of operators in divergence form in $ L^{q}(D) $ (see \cite{auschzer_badr_haller_rehberg_square_root_divergence_form_lp},
\cite{egert_haller_tolksdorf_kato_mixed_boundary} and the references therein) and the so called isomorphism property between $ W^{1,q}(D) $ and $ W^{-1,q}(D) $ (see \cite{disser_kaiser_rehberg_optimal_sobolev_regularity_for_divergence_elliptic_operators}, \cite{elschner_rehberg_gunther_optimal_regularity_c1_interace}). Afterwards, we restrict us to Dirichlet boundary conditions and show that under a global Lipschitz assumption on the diffusion matrix, the solution does not explode and exists on the whole interval $ [0,T]. $ This generalises the work of Hofmanova and Zhang (\cite{hofmanova_zhang_existence_uniqueness_quasilinear}) from the torus to arbitrary bounded $ C^1 $-domains. Moreover, our method does not need initial data in the space $ C^{1+\varepsilon}(\overline{D}) $, but only $ u_0\in (W^{-1,q}(D),W^{1,q}_0(D))_{1-1/p,p}, $ which seems to be natural, if one expects solutions that are pathwise in $ L^{p}(0,T;W^{1,q}_0(D)). $

    \section{Preliminaries}
The purpose of this section is to provide a short overview over the basic tools used in this paper. For most of the proofs and further details, we give references to the literature.

Throughout this paper, let $ (\Omega,\mathfrak{F},\F=(\ft)_{t\geq 0},\PP) $ be a filtered probability space, that satisfies the usual conditions, i.e. $ \mathcal{F}_0 $ contains all $ \PP $-null sets and the filtration is right-continuous. Moreover, given normed spaces $ X $ and $ Y $, $ B(X,Y) $ denotes the set of all linear and bounded operators from $ X $ to $ Y $. Last but not least, we write $ C(a,b;X) $ for the space of uniformly continuous functions on $ [a,b] $ with values in $ X $ equipped with its usual norm.

\subsection{Stopping times and the \texorpdfstring{$\sigma$-algebra of $\tau$-past}{sigma-algebra of tau-past}}
It will be necessary to stop a stochastic process when it leaves certain balls around the initial value. However, this time will differ from path to path and therefore we introduce stopping times. $ \tau:\Omega\to [0,T] $ is called $ \F $-stopping time, if $ \{\tau\leq t \}\in\ft $ for all $ t\in [0,T]. $ By the right-continuity, this is equivalent to $ \{\tau<t \}\in\ft. $ The $ \sigma $-algebra
\[ \mathcal{F}_\tau =\big\{A\in\mathfrak{F}: A\cap\{t\leq \tau \}\in\mathcal{F}_t\ \forall t\in [0,T]\big\}\]
is called $ \sigma $-algebra of $ \tau $-past and can be interpreted as the knowledge of an observer at the random moment $ \tau. $ 

The following well-known results will be used frequently. The proof can be found e.g. in \cite{Klenke}, Lemma $ 9.21 $ and Lemma $ 9.23. $
\begin{Proposition}\label{appendix_probability_tau_past}
$ \mathcal{F}_\tau $ is a $ \sigma $-algebra and satisfies the following properties.
\begin{itemize}
\item [a)] If $ \tau=t $ almost surely for some $ t\in [0,T], $ we have $ \mathcal{F}_\tau=\ft. $
\item [b)] Given another $ \F$-stopping time $ \sigma, $ we have $ \mathcal{F}_{\tau\wedge\sigma}=\mathcal{F}_\tau\cap\mathcal{F}_\sigma. $ In particular, if $ \tau\leq\sigma $ almost surely, we have the inclusion $ \mathcal{F}_\tau\subset \mathcal{F}_\sigma. $
\item [c)] If $ (X(\cdot,t))_{t\in[0,T]} $ is a progressively measurable process with respect to $ \F, $ then the random variable $ X_\tau(\omega):=X(\omega,\tau(\omega)) $ is $ \mathcal{F}_\tau $-measurable.
\end{itemize}
\end{Proposition}
Throughout this article, we will use the notation 
\[ \Lambda\times [\tau,\mu):=\big\{(\omega,t)\in \Omega\times[0,T]:\ t\in [\tau(\omega),\mu(\omega)) \big\} \]
for some $ \Lambda\subset\Omega $ and stopping times $ \tau,\mu $ with $ \tau\leq\mu $. Closed and open random intervals are defined similarly. If we call a process $ u$ defined on $[\Omega\times [\tau,\mu] $ strongly measurable or adapted, we mean that $ u\ind_{\Omega\times [\tau,\mu]} $ is strongly measurable or adapted.

 Since we did not find any reference in the literature for the following Lemma, we give a short proof.
\begin{Lemma}\label{quasi_stochastic_parabolic_stopping_time}
	Let $ X_t:\Omega\times [0,T]\to\R_{\geq 0} $, $ t\in [0,T], $ be an $ \F $-adapted process with almost surely continuous paths, $ \sigma $ an $ \F $-stopping time with values in $ [0,T] $ and $ \lambda>0. $ If we define
	\[ \widetilde{\sigma}=\inf\big\{t\in [0,T-\sigma]:X_{t+\sigma}>\lambda \big\}\wedge T, \]
	then $ \sigma+\widetilde{\sigma} $ is also an $ \F $-stopping time.
\end{Lemma}
\begin{proof}
	Since $ \F $ is right-continuous, it is sufficient to prove $ \{\sigma+\widetilde{\sigma}<t \}\in\mathcal{F}_t $ for given $ t\in [0,T]. $ We start with
	\begin{equation}\label{quasi_stochastic_parabolic_stopping_time_eq}
	\big\{\sigma+\widetilde{\sigma}<t\big\}=\bigcup_{q_1,q_2\in\Q_{\geq 0}, q_1+q_2<t}\{\sigma<q_1,\widetilde{\sigma}<q_2 \}
	\end{equation}
	and prove that the sets $ \{\sigma<q_1,\widetilde{\sigma}<q_2 \} $ are contained in $ \mathcal{F}_t. $ For fixed $ q_1,q_2\in\Q_{\geq 0} $ with $ q_1+q_2<t $, the definition of $ \widetilde{\sigma} $ and the pathwise continuity of $ t\mapsto X_t $ yield
	\begin{align*}
	\{\widetilde{\sigma}<q_2\}&
	=\bigcup_{s\in [0,q_2)}\{X_{\sigma+s}>\lambda \}
	=\bigcup_{q\in [0,q_2)\cap\Q}\{X_{\sigma+q}>\lambda\}.
	\end{align*}
	Thus, we have
	\begin{align*}
	\big\{\sigma<q_1,\widetilde{\sigma}<q_2\big\}=\bigcup_{q\in[0,q_2)\cap\Q }\Big(\{\sigma<q_1\}\cap\{X_{\sigma+q}>\lambda \}\Big).
	\end{align*}
	Moreover, Proposition \ref{appendix_probability_tau_past} implies $ \{X_{\sigma+q}>\lambda \}\in \mathcal{F}_{\sigma+q} $ and since $ \{\sigma<q_1\}\in\mathcal{F}_{q_1} $ in any case by definition of stopping times, we conclude
	\begin{align*}
	\big\{\sigma<q_1,\widetilde{\sigma}<q_2\big\}\in \bigcup_{q\in[0,q_2)\cap\Q } \big(\mathcal{F}_{q_1}\cap\mathcal{F}_{\tau+q}\big)\subset\mathcal{F}_{q_1+q_2}\cap\mathcal{F}_{\sigma+q_2}\subset \mathcal{F}_{\min(q_1+q_2,\sigma+q_2)}\subset\mathcal{F}_{q_1+q_2}.
	\end{align*}
	Hence the claimed result follows by \eqref{quasi_stochastic_parabolic_stopping_time_eq}.
\end{proof}
\subsection{\texorpdfstring{$\gamma $-radonifying operators and stochastic integration}{gamma-radonifying operators and stochastic integration}}
Let $ \widetilde{H} $ a separable Hilbert space with orthonormal basis $ (h_n)_{n\in\N} $, $ E $ a Banach space and $ (\gamma_n)_{n\in\N} $ a sequence of independent standard-Gaussian distributed random variables. The Banach space $ \gamma(\widetilde{H};E) $ of $ \gamma $-radonifying operators is the closure of $ \{T:\widetilde{H}\to E\text{ linear and of finite rank}\} $ with respect to the norm
\[ \|T\|_{\gamma(\widetilde{H};E)}=\big(\E\|\sum_{n=1}^{\infty}\gamma_nTh_n\|_{E}^2\big)^{1/2} \]
Here, the norm is independent of the choice of the orthonormal basis. 

In the special case, that $ E=L^{q}(O,\mu) $ for some $ q\in (1,\infty) $, $ \gamma(\widetilde{H};E) $ is isomorph to $ L^{q}(O;\widetilde{H}) $ via  the isomorphism $ L^{q}(O;\widetilde{H})\ni f\mapsto T_f\in \gamma(\widetilde{H};E), $ where $ T_f $ is defined by
\[ T_f(h)(x):=\langle f(x),h\rangle_{H} \]
for $ h\in \widetilde{H} $ and $ x\in O. $ The equivalence of $ \|T_f\|_{\gamma(\widetilde{H};E)}\simeq \|f\|_{L^{q}(O;\widetilde{H})} $ can be shown easily by the Kahane-Khintchine inequality $  \Big(\E\|\sum_{n=1}^{\infty}\gamma_nf_n\|_{E}^q\Big)^{1/q}\simeq_{q} \E\|\sum_{n=1}^{\infty}\gamma_nf_n\|_{E} $
for $ q\in [1,\infty) $. For further details about $ \gamma $-radonifying operators, we refer to the survey paper of Van Neerven (see \cite{nerven_gamma_a_survey}).\\ \\
We now sketch the construction of the stochastic integral with respect to an $ L^{2}(0,T;H) $-cylin\-dri\-cal Brownian motion, that is a bounded linear operator $ W:L^{2}(0,T;H)\to L^{2}(\Omega) $ with the following properties.
\begin{itemize}
	\item	[a)]For all $ f\in L^{2}(0,T;H), $ the random variable $ W(f) $ is centred Gaussian.
	\item[b)] For all $ t\in [0,T] $ and $ f\in L^{2}(0,T;H) $ supported in $ [0,t], $ $ W(f) $ is $ \ft $-measurable.
	\item[c)] For all $ t\in [0,T] $ and $ f\in L^{2}(0,T;H) $ supported in $ (t,T], $ $W(f)$ is independent of $ \ft. $ 
\item[d)] We have $ \E(W(f)\cdot W(g))= \langle f,g\rangle_{L^{2}(0,T;H)} $ for all $ f,g\in L^{2}(0,T;H). $
\end{itemize}
An example for an $ L^{2}(0,T;H) $-cylindrical Brownian motion is a family $ (\beta_n)_{n\in\N} $ of independent real valued Brownian motions together with $ H=l^2(\N) $ and $ W $ uniquely determined by the formula $ W(\ind_{(0,t]}e_n)=\beta_n(t), $ $ n\in\N $, where $ (e_n)_n $ is the sequence of the standard unit vectors in $ l^2(\N) $. 

For a stochastic processes $ G:\Omega\times \R_{\geq 0}\times H\to E $ of the form
$$ G=\ind_{(s,t]\times B}\langle\cdot,y\rangle_{H} x $$
for $ B\in\mathcal{F}_s $, $ y\in H $ and $ x\in E, $ we can define the stochastic integral via
\[ I(G):=\int_{0}^{T}G dW:=\ind_{B}W(\ind_{(s,t]}h)x\in E \]
and we can extend it to $ \F $-adapted step processes, that are finite linear combinations of such processes. Van Nerven, Veraar and Weis proved in \cite{nerven_veraar_weis_stochastic_integration_in_umd_spaces} the following two sided estimate for this stochastic integral.
\begin{Theorem}
Let $ E $ be a UMD Banach space and $ G $ be an $ \F $-adapted step processes in $ \gamma(H;X). $ Then, for all $ p\in(1,\infty) $ one has the It\^o-isomorphism
\[ \E\|I(G)\|_{L^{p}(\Omega;E)}\simeq_{p}\|G\|_{L^{p}(\Omega;\gamma(L^2(0,T;H);E))}. \]
In particular, the stochastic integral can be continued to a linear and bounded operator $$ I:L^{p}(\Omega;\gamma(L^2(0,T;H);E))\to L^{p}(\Omega;E). $$
\end{Theorem}
In this article, we focus on integrands in $ L^{p}(\Omega\times[0,T];\gamma(H;E)) $ for $ p>2. $ Therefore, we restrict ourselves to UMD Banach spaces of type 2 (details about type and cotype of Banach spaces can be found in \cite{pisier_probabalistic_methods_in_the_geometry_of_Banach_spaces}), for which the embeddings
\[ L^{p}(0,T;\gamma(H;E))\hookrightarrow L^{2}(0,T;\gamma(H;E))\hookrightarrow\gamma(L^2(0,T;H);E)\]
are bounded. Consequently, the stochastic integral $ I(G) $ is also defined for functions $ G\in L^{p}(\Omega\times [0,T];\gamma(H;E)).  $ Note, that Hilbert spaces or Banach spaces that are isomorph to closed subspaces of $ L^{q}(O;\mu) $, $ q>2 $, are of type 2 and have the UMD property.
\subsection{\texorpdfstring{R-Boundedness and $ H^{\infty} $-calculus}{boundedness of the h-infty-calculus}}
Let $ (r_n)_{n\in\N} $ be a sequence of Rademacher random variables on a probability space $ (\widetilde{\Omega},\mathcal{A},\widetilde{\mathbb{P}}) $, i.e. $ \widetilde{\mathbb{P}}(r_n=1)=\widetilde{\mathbb{P}}(r_n=-1)=\tfrac{1}{2}. $ Given the Banach spaces $ X $ and $ Y $, a family  $ \mathcal{T}\subset B(X,Y) $ is called R-bounded, if there exists $ C>0 $, such that
\[ \|\sum_{j=1}^{N}r_jT_jx_j\|_{Y}\leq C\|\sum_{j=1}^{N}r_jx_j\|_{X} \]
for all $ (T_j)_{j=1}^{N}\subset\tau $ and $ (x_j)_{j=1}^{N}\subset X. $ The least possible constant $ C $ will be called R-bound of $ \mathcal{T}. $ Note, that every R-bounded family is uniformly bounded in $ B(X,Y), $ whereas the converse holds only if $ X,Y $ are Hilbert spaces. For details, we refer to \cite{clement_sukochev_schauder_decomposition}, \cite{denk_hieber_pruss_r_boundedness_fourier_multiplier} and \cite{Kunstmann_Weis_Lecture_Notes}.

An operator $ (A,D(A)) $ is called sectorial on a Banach space $ E $ of angle $ \theta\in (0,\pi/2) $, if it is closed, densely defined, injective and it has a dense range. Moreover, we require, that its spectrum is contained in the sector $ \Sigma_{\theta}=\{z\in\C:|\operatorname{arg}(z)|<\theta\} $ and that the set
\begin{equation}\label{preliminaries_sectorial}
\big \{ \lambda R(\lambda,A): \lambda\notin \Sigma_{\phi} \big\} 
\end{equation}
is for all $ \phi\in (\theta,\pi) $ bounded in $ B(E) $ and the bound only depends on $ \phi $.
In this case, $ -A $ generates a holomorphic semigroup on $ E. $ If the set in \eqref{preliminaries_sectorial} is even R-bounded, one says that $ A $ is R-sectorial.

 For any holomorphic function $ f $ on $ \Sigma_{\phi},\ \phi>\theta, $ satisfying the growth estimate $ |f(z)|\leq C\frac{|z|^{\delta}}{1+|z|^{2\delta}} $ for some $ \delta>0 $, the integral
\[ f(A)=\frac{1}{2\pi i}\int_{\Sigma_{\phi}}f(z)R(z,A)\operatorname{dz} \]
exists. This integral defines a functional calculus for functions with the above growth estimate. We say that $ A $ has a bounded $ H^{\infty}(\Sigma_{\phi}) $-calculus, if there exists $ C>0 $ such that
\begin{equation}\label{prelim_bla}
\|f(A)\|_{B(E)}\leq C \|f\|_{\infty}
\end{equation}
is satisfied for all these functions. The least constant $ C>0 $ will be called bound of the $ H^{\infty} $-calculus. In this case, the functional calculus $ f\mapsto f(A) $ can be extended to any bounded holomorphic function on $ \Sigma_{\phi} $ and \eqref{prelim_bla} remains true. Details on sectorial operators, R-sectorial operators and the functional calculus can be found in \cite{Kunstmann_Weis_Lecture_Notes},  \cite{Haase_Functional_calculus} and in the references we mentioned in the introduction. Note, that the boundedness of the $ H^{\infty} $-calculus of $ A $ particularly implies that $ A $ is R-sectorial, if $ E $ is UMD. A nice list of operators having such a functional calculus can be found in \cite{nerven_veraar_weis_maximal_lp_stochastic_evolution}, Example $ 3.2. $ 

    \section{\texorpdfstring{Maximal $ L^{p} $-regularity for stochastic evolution equations}{Maximal l-p-regularity for stochastic evolution equations}}

In this section, we present results on semilinear parabolic stochastic evolution equations based on the work of Van Neerven, Veraar and Weis on maximal $  L^{p} $-regularity in \cite{nerven_veraar_weis_maximal_lp_regularity_stochastic_convolution} and \cite{nerven_veraar_weis_maximal_lp_stochastic_evolution}. Throughout this section, $ \tau $ denotes an $ \F $-stopping time with $ 0\leq\tau\leq T $ almost surely. We give conditions, that guarantee the existence and uniqueness of strong solutions of
\begin{equation*}
\operatorname{(SEE)}\begin{cases}
du(t)=\big[-Au(t)+F(u(t))+f(t)\big]dt+\big[B(u(t))+b(t)\big]dW(t),\ t\in (\tau,T],\\
u(\tau)=u_\tau.
\end{cases}
\end{equation*}
Before we specify our assumptions, we point out that they contain slight generalizations of the above papers. Firstly, as described in the introduction, we will need to start the equation not only at time zero, but at an $ \F $-stopping time $ \tau $ with an $ \mathcal{F}_\tau $-measurable initial value $ u_\tau. $ Moreover, in \cite{nerven_veraar_weis_maximal_lp_stochastic_evolution}, the authors demanded the following Lipschitz condition on the deterministic nonlinearity. They assumed the existence of $ L_F,\widetilde{L_F},C_F>0 $ such that
\begin{align*}
\|F(t,\omega,x)\|_{E}\leq C_F\left (1+\|x\|_{E^1}\right )
\end{align*}
and 
\begin{align*}
\|F(t,\omega,x)-F(t,\omega,x)\|_{E}&\leq L_{F}\|x-y\|_{ E^1}+\widetilde{L}_{F}\|x-y\|_{E}
\end{align*}
for all $ t\in [0,T],x\in E^1$ and almost all $\omega\in\Omega $ and $ L_F $ was demanded to be small enough. It turns out, that this assumption is not suitable to deal with quasilinear equation. We need a weaker analogue that is not pointwise in time, but with respect to $ \|\cdot\|_{L^{p}(0,T;E)} $. Therefore, we replaced the above assumption by $ [\operatorname{S4}] $ and $ [\operatorname{S5}]. $ However, the maximal $ L^{p} $-regularity technique allows this generalization and the proofs are essentially unchanged. We make the following hypothesis.
\begin{itemize}
	\item[\lbrack\text{S1}\rbrack] Let $ p\in (2,\infty) $ and $ E,E^1 $ be UMD Banach spaces with type $ 2 $ or alternatively $ p=2 $ and $ E,E^1 $ Hilbert spaces. We assume the embedding $ E^{1}\hookrightarrow E $ to be dense and write $ E^{1/2} $ for the complex interpolation space $ [E,E^{1}]_{1/2} $ and $ E_p $ for the real interpolation space $ (E,E^{1})_{1-1/p,p} $. Moreover, let the family 
	$$ \{J_\delta:\delta>0 \}\subset B\big(L^{p}(\Omega\times(0,\infty);\gamma(H;E)),L^{p}(\Omega\times(0,\infty);E)\big) $$
	defined by
	\[ J_{\delta}g(t):=\delta^{-1/2}\int_{(t-\delta)\vee 0}^{t}g(s)dW(s) \]
	be $ R $-bounded. 
	\item[\lbrack\text{S2}\rbrack]The mapping $ A:\Omega\to B(E^{1},E) $ is such that $ \omega\mapsto A(\omega)x $ is for all $ x\in E^{1} $ strongly $ \mathcal{F}_\tau $-measurable and such that $ D(A(\omega))=E^{1} $, i.e we have 
	\[ \|A(\omega)x\|_{E}\simeq \|x\|_{E^{1}} \]
	for almost all $ \omega\in\Omega $ and all $ x\in E^{1} $ with estimates independent of $ \omega\in\Omega $. Moreover, we assume $ 0\in\rho(A(\omega)) $ and that for almost all $ \omega\in\Omega $ the operator $ A(\omega) $ has a bounded $ H^{\infty} $-calculus of angle $ \eta\in [0,\pi/2) $ with
	\[ \|\Psi(A(\omega))\|_{B(E)}\leq M\|\Psi\|_{H^{\infty}(\Sigma_{\eta})} \]
	for all $ \Psi\in H^{\infty}(\Sigma_{\eta}) $ with constants $ M,\eta $ independent of $ \omega. $
	\item[\lbrack\text{S3}\rbrack] The initial value $ u_\tau:\Omega\to E_p $ is a strongly $ \mathcal{F}_\tau $-measurable random variable.
\item[\lbrack\text{S4}\rbrack] The mapping	$ F:L_{\mathbb{F}}^0(\Omega; L^{p}(\tau,T;E^1)\cap C(\tau,T;E_p)) \to L_{\mathbb{F}}^0(\Omega; L^{p}(\tau,T;E))$ is such that for a given $ \F $-stopping time $ \widetilde{\tau} $ with $ \tau\leq\widetilde{\tau}\leq T $ almost surely, we have $ F(u)\ind_{[\tau,\widetilde{\tau}]}=F(v)\ind_{[\tau,\widetilde{\tau}]} $ almost surely, whenever $ u\ind_{[\tau,\widetilde{\tau}]}=v\ind_{[\tau,\widetilde{\tau}]}. $  Moreover, there exist $ L_F^{(i)},\widetilde{L}_{F},C_{F}^{(i)}\geq 0 $, $ i=1,2,  $ such that $ F $ is for almost all $ \omega\in\Omega $ of linear growth, i.e. 
	\begin{align*}
	\|F(\cdot,\omega,\phi_1)&\|_{ L^{p}(\tau,T;E)}\leq C_F^{(1)}\left (1+\|\phi_1\|_{ L^{p}(\tau,T;E^1)}\right )+C_F^{(2)}\left (1+\|\phi_1\|_{C(\tau,T;E_p)}\right )
	\end{align*}
	and Lipschitz continuous, i.e.
	\begin{align*}
	\|F(&\cdot,\omega,\phi_1)-F(\cdot,\omega,\phi_2)\|_{ L^{p}(\tau,T;E)}\\
	&\leq L_{F}^{(1)}\|\phi_1-\phi_2\|_{ L^{p}(\tau,T;E^1)}+\widetilde{L}_{F}\|\phi_1-\phi_2\|_{_{ L^{p}(\tau,T;E)}}
	 +L_{F}^{(2)}\|\phi_1-\phi_2\|_{C(\tau,T;E_p)}
	\end{align*}
	for all $ \phi_1,\phi_2\in  L^{p}(\tau,T;E^1)\cap C(\tau,T;E_p) $ with constants independent of $ \omega $. 
	\item[\lbrack\text{S5}\rbrack] The mapping	$ B:L_{\mathbb{F}}^0(\Omega; L^{p}(\tau,T;E^1)\cap C(0,T;E_p)) \to L_{\mathbb{F}}^0(\Omega; L^{p}(\tau,T;\gamma(H;E^{1/2})))$ is such that for a given $ \F $-stopping time $ \widetilde{\tau} $ with $ \tau\leq\widetilde{\tau}\leq T $ almost surely, we have $ B(u)\ind_{[\tau,\widetilde{\tau}]}=B(v)\ind_{[\tau,\widetilde{\tau}]} $ almost surely, whenever $ u\ind_{[\tau,\widetilde{\tau}]}=v\ind_{[\tau,\widetilde{\tau}]}. $ Moreover, there exist $ L_B^{(i)},\widetilde{L}_{B},C_{B}^{(i)}\geq 0 $, $ i=1,2,  $ such that $ B $ is for almost all $ \omega\in\Omega $ of linear growth, i.e.
	\begin{align*}
	\|B(\cdot,\omega,\phi_1)&\|_{L^{p}(\tau,T;\gamma(H;E^{1/2}))}\leq C_B^{(1)}\left (1+\|\phi_1\|_{ L^{p}(\tau,T;E^{1})}\right )+C_B^{(2)}\left (1+\|\phi_1\|_{C(\tau,T;E_p)}\right )
	\end{align*}
	and Lipschitz, i.e.
	\begin{align*}
	\|B(&\cdot, \omega,\phi_1)-B(\cdot,\omega,\phi_2)\|_{L^{p}(\tau,T;\gamma(H;E^{1/2}))}\\
	&\leq L_{B}^{(1)}\|\phi_1-\phi_2\|_{ L^{p}(\tau,T,E^{1})}+\widetilde{L}_{B}\|\phi_1-\phi_2\|_{ L^{p}(\tau,T;E)}+L_{B}^{(2)}\|\phi_1-\phi_2\|_{C(\tau,T;E_p)}
	\end{align*}
	for all  $ \phi_1,\phi_2\in  L^{p}(\tau,T;E^{1})\cap C(\tau,T;E_p) $ with constants independent of $ \omega $. 
	\item[\lbrack\text{S6}\rbrack]The functions $ f:\Omega\times [0,T]\to E $ and $ b:\Omega\times [0,T]\to\gamma(H;E^{1/2}) $ are strongly measurable and adapted to $\F. $ Moreover, we assume $ f\ind_{[\tau,T]}\in L^{p}(\Omega\times[0,T];E) $ and $ b\ind_{[\tau,T]}\in L^{p}(\Omega\times[0,T];\gamma(H;E^{1/2})) $.
\end{itemize}

Before we proceed, we want to refer to the very helpful comments on these assumptions by Van Neerven, Veraar and Weis in \cite{nerven_veraar_weis_maximal_lp_stochastic_evolution}, Remark $ 4.1. $ Additionally, we point out, that every Banach space $ E $ isomorphic to $ L^{q}(O,\mu) $ with $ q>2 $ satisfies $ [\operatorname{S1}]. $ In particular, this is true for all Sobolev and Besov spaces with integrability index $ q. $ Last but not least, we want to mention that the constants $L_F^{(1)},L_{F}^{(2)}  $ and $L_B^{(1)},L_{B}^{(2)}  $ in $ [\operatorname{S4}] $ and $ [\operatorname{S5}] $ have to be small enough. To be precise, they depend on the constants appearing in the maximal $ L^{p}$-regularity estimates 
\begin{equation}\label{lp_max_reg_evolution_equation_stochastic_max_reg}
 \big\|t\mapsto\int_{0}^{t}e^{-A(t-s)}g(s)dW(s)\big\|_{L^{p}(\Omega\times\R_{\geq 0};E^1)\cap L^{p}(\Omega;C(0,T;E_p)} \leq C_{\operatorname{MRS}}\|g\|_{L^{p}(\Omega\times\R_{\geq 0};\gamma(H;E^{1/2}))}
\end{equation}
\begin{equation}\label{lp_max_reg_evolution_equation_deterministic_max_reg}
\big\|t\mapsto\int_{0}^{t}e^{-A(t-s)}f(s)\operatorname{ds}\big\|_{L^{p}(\Omega\times\R_{\geq 0};E^1)\cap L^{p}(\Omega;C(0,T;E_p))} \leq C_{\operatorname{MRD}}\|f\|_{L^{p}(\Omega\times\R_{\geq 0};E)}
\end{equation}
and have to satisfy the analogue of the condition in Theorem $ 4.5 $ in \cite{nerven_veraar_weis_maximal_lp_stochastic_evolution}. 
\begin{itemize}
	\item[\lbrack\text{S7}\rbrack] Let the constants in $ [\operatorname{S4}] $ and $ [\operatorname{S5}]  $ be small enough to ensure $$ C_{ \operatorname{MRD}}(L_F^{(1)}+L_F^{(2)})+C_{\operatorname{MRS}}(L_B^{(1)}+L_B^{(2)})<1. $$
\end{itemize}
Before we state the main result about solutions of $ (\operatorname{SEE}) $, we recall the concept of strong solutions. 
\begin{Definition}
Let $ \mu$ be another $ \F $-stopping time with $ \tau\leq\mu\leq T $ almost surely. A process $ u:\Omega\times [\tau,\mu]\to E $ is called a strong solution of $ (\operatorname{SEE}) $, if it is strongly measurable, adapted and 
\begin{itemize}
\item [a)] $ u\in L^{1}(\tau,\mu;E^{1}) $, $ F(\cdot,u)\in L^{1}(\tau,\mu;E) $ almost surely and the stochastic process $ B(\cdot,u)\ind_{[\tau,\mu]}:\Omega\times[0,T]\to \gamma(H;E) $ is stochastically integrable.
\item [b)] The identity 
\begin{align*}
u(t)-u_\tau=-\int_{\tau}^{t}Au(s)\operatorname{ds}+\int_{\tau}^{t}F(s,u(s))+f(s)\operatorname{ds}+\int_{\tau}^{t}B(s,u(s))+b(s)dW(s)
\end{align*}
	holds almost surely for all $ t\in [\tau,\mu] $ in $ E. $
\end{itemize}
\end{Definition}
Note, that any strong solution is also a mild solution in the sense of evolution equations (see \cite{nerven_veraar_weis_maximal_lp_stochastic_evolution}, Proposition $ 4.4 $). We now present the main results of this section. The proof is almost the same as the proof of Theorem $ 4.5. $ in \cite{nerven_veraar_weis_maximal_lp_stochastic_evolution} and and we just discuss the differences. For sake of clearness, we split the statements into a Theorem and some Corollaries.
\begin{Theorem}\label{lp_auto_stoch_evolution_equation_global_strong_solution_lp_initial_data}
Additionally to $ [\operatorname{S1}]-[\operatorname{S7}] $, we assume $ u_\tau\in L^p(\Omega,E_p) $ for some $ p\in (2,\infty) $. Then, there exists a unique strong solution $ u $ of $ (\operatorname{SEE}) $
 on $ [\tau, T] $ 
and $ u $ has almost surely continuous paths on $ [\tau,T]$ in $ E_p $. Moreover, we have the estimates
\begin{equation}\label{lp_auto_stoch_evolution_equation_global_strong_solution_lp_initial_data_estimate11}
\left (\E \|u\ind_{[\tau,T]}\|_{ L^{p}(0,T;E^{1})}^{p}\right )^{1/p}\leq C\big(1+\|u_\tau\|_{L^{p}(\Omega,E_p)}\big)
\end{equation}
and
\begin{equation}\label{lp_auto_stoch_evolution_equation_global_strong_solution_lp_initial_data_estimate22}
\left (\E \sup_{t\in[\tau,T]}\|u(t)\|_{E_p}^{p}\right )^{1/p}\leq C\big(1+\|u_\tau\|_{L^{p}(\Omega,E_p)}\big).
\end{equation}
for some $ C>0 $ independent of $ u_\tau $.
\end{Theorem}
\begin{proof}
The main idea is to apply the contraction mapping theorem to prove existence and uniqueness on small intervals and to put these solutions together to a solution on $ [\tau,T] $. Let $ \kappa>0 $. We consider the operator
\begin{align}\label{lp_auto_stoch_evolution_equation_fixed_point_operator}
K\phi(t)=&e^{-(t-\tau)A}u_\tau+  \int_{\tau}^{t}(e^{-(t-s)A} (F(s,\phi(s))+f(s))\operatorname{ds}\notag\\
&+ \int_{0}^{t}e^{-(t-s)A} (B(s,\phi(s))+b(s))\ind_{(\tau,T]}(s)dW(s)
\end{align}
on the metric space 
\begin{align*}
\mathcal{E}_p(\tau,\kappa):= \big\{&u:\Omega\times[\tau,(\tau+\kappa)\wedge T]\to E_p\big|\ u \text{ is adapted },\ u\ind_{[\tau,(\tau+\kappa)\wedge T]}\in L^{p}(\Omega\times [0,T];E^{1}),\\
& u\text{ is a.s. cont. on } [\tau,(\tau+\kappa)\wedge T]\text{ with }u(\tau)=u_{\tau} \text{ a.s., } \E\sup_{t\in[\tau,(\tau+\kappa)\wedge T]}\|u(t)\|_{E_p}^p<\infty \big\},
\end{align*}
 endowed with the metric induced by the norm
\begin{align*}
\|u\|_{\mathcal{E}_p(\tau,\kappa),\mu}:=&\big\|u\ind_{[\tau,(\tau+\kappa)\wedge T]}\big\|_{ L^{p}(\Omega\times [0,T];E^{1}))}+\mu\big\|u\ind_{[\tau,(\tau+\kappa)\wedge T]}\big\|_{L^{p}(\Omega\times [0,T];E))}\\
&+\Big (\E\|u(t)\|_{C(\tau,(\tau+\kappa)\wedge T;E_p)}^{p}\Big)^{1/p},
\end{align*}
where $ \mu>0 $ will be chosen later on. We have to show that the stochastic integrand in \eqref{lp_auto_stoch_evolution_equation_fixed_point_operator} is adapted and thus the integral is well-defined. We have $ (0,\infty)\subset\rho(-A(\omega)) $ for almost all $ \omega\in\Omega $ and the identity 
\[ e^{-tA(\omega)}x=\lim_{n\to\infty}\left (\tfrac{n}{t}R(\tfrac{n}{t},-A(\omega))\right )^{n}x \]
holds true. Thus,  $ \omega\to e^{-tA(\omega)}x $ is for all $ x\in E $ and $ t\geq 0 $ strongly $ \mathcal{F}_{\tau}$-measurable as pointwise limit of strongly measurable functions. 

As a consequence, for fixed $ s\leq t $, the mapping $ \omega\mapsto e^{-(t-s)A(\omega)}x\ind_{ s>\tau(\omega)} $ is strongly $ \mathcal{F}_{s}$-measurable. Indeed, for every Borel set $ B\subset E $ and $ x\in E $ we have
\begin{align*}
\{ e^{-(t-s)A}x\ind_{ s>\tau}\in B\}=\{0\in B,s\leq\tau \}\cup\{e^{-(t-s)A}x\in B,  s>\tau\}.
\end{align*}
Since the filtration $ \F $ is right-continuous, we both have $ \{s\leq\tau \}\in \mathcal{F}_s $ and $ \{s>\tau \}\in \mathcal{F}_s $. Thus we obtain $ \{0\in B,s\leq\tau \}\in\mathcal{F}_s $ and $$ \{e^{-(t-s)A}x\in B, s>\tau\}\in\mathcal{F}_\tau\cap\mathcal{F}_s=\mathcal{F}_{\tau\wedge s}\subset \mathcal{F}_s $$
which yields the desired result. We conclude, that 
\[ \omega \mapsto e^{-(t-s)A(\omega)}(B(\omega,\phi(s))+b(\omega,s))\ind_{ s>\tau(\omega)} \]
is strongly $ \mathcal{F}_s $-measurable as composition of strongly $ \mathcal{F}_s $-measurable functions. 

In the same way as Van Neerven, Veraar and Weis, we check that $ K $ is a contraction if $ \kappa $ is small enough and $ \mu(\kappa) $ large enough. This yields the existence and uniqueness of a strong solution on $ [\tau,(\tau+ \kappa)\wedge T]. $ Iterating this procedure yields the claimed result.
\end{proof}

If two different initial values coincide on some subset of $ \Omega $ of positive measure, the corresponding solutions of $ (\operatorname{SEE}) $ also coincide on this subset. The proof is a slight variation of step 7 in the proof of Theorem $ 4.5 $ in \cite{nerven_veraar_weis_maximal_lp_stochastic_evolution}.

\begin{Lemma}\label{lp_auto_stoch_evolution_equation_global_strong_solution_local_uniqueness}
	We assume the conditions $ [\operatorname{S1}]-[\operatorname{S7}] $ to be fulfilled.
	Let $ u $ and $ v $ be the strong solutions of $ (\operatorname{SEE}) $ corresponding to the strongly  $ \mathcal{F}_{\tau} $-measurabe initial values $ u_\tau,v_{\tau}\in L^{p}(\Omega,E_p) $. Then, we have $ u(\omega,t)=v(\omega,t) $ for almost all $ \omega\in\{u_{\tau}=v_{\tau} \} $ and all $ t\in [\tau(\omega),T]. $
\end{Lemma}

As an easy application of this Lemma, one can prove existence and uniqueness of strong solutions of $ (\operatorname{SEE}) $ with initial data $ u_{\tau}, $ that is only measurable with respect to $ \mathcal{F}_{\tau}. $ The strategy of the proof was already used in \cite{nerven_veraar_weis_stochastic_evolution_equations_umd_spaces}, Theorem $ 7.1 $ and \cite{seidler_maximal_inequality_revisited}, Proposition $ 5.4, $ in a nonmaximal regularity situation. In step 3 and 4 of the proof of Theorem $ 4.5 $ in \cite{nerven_veraar_weis_maximal_lp_stochastic_evolution}, there is a version that fits to our situation.
\begin{Corollary}\label{lp_auto_stoch_evolution_equation_global_strong_solution_measurable_initial_data}
We assume the conditions $ [\operatorname{S1}]-[\operatorname{S7}] $ to be fulfilled.
If we do not demand anything on $ u_\tau $ but to be a strongly $ \mathcal{F}_{\tau} $-measurable, $ E_p $-valued random variable, then $ (\operatorname{SEE}) $
still has a unique strong solution $ u $ on $ [\tau,T] $ with $ u\in L^{p}(\tau,T;E^{1}) \cap C(\tau,T;E_p)$ almost surely, but $ u $ has not necessarily any integrability properties with respect to $ \omega. $
\end{Corollary}

Finally, we give an analogous result to Lemma \ref{lp_auto_stoch_evolution_equation_global_strong_solution_local_uniqueness} in case that the initial values $ u_{\tau}$ and $v_{\tau} $ are not in $ L^{p}(\Omega,E_p) $, but only $ \mathcal{F}_\tau $-measurable. The statement is a direct consequence of Lemma \ref{lp_auto_stoch_evolution_equation_global_strong_solution_local_uniqueness} and Corollary \ref{lp_auto_stoch_evolution_equation_global_strong_solution_measurable_initial_data}.
\begin{Corollary}\label{lp_auto_stoch_evolution_equation_global_strong_solution_measurable_initial_data_uniqueness}
We assume the conditions $ [\operatorname{S1}]-[\operatorname{S7}] $ to be fulfilled. 
Let $ u $ and $ v $ be the strong solutions of $ (\operatorname{SEE}) $ to the strongly  $ \mathcal{F}_{\tau} $-measurabe initial random variables $ u_\tau,v_{\tau}:\Omega\to E_p $. Then we have $ u(\omega,t)=v(\omega,t) $ for almost all $ \omega\in\{u_{\tau}=v_{\tau} \} $ and all $ t\in [\tau(\omega),T]. $
\end{Corollary}

    \section{Abstract quasilinear evolution equations}
In this chapter, we consider a quasilinear stochastic evolution equation of the form
\begin{equation*}
(\operatorname{QSEE})\begin{cases}
du(t)&=\left [-A(u(t))u(t)+F(t,u(t))+f(t)\right]\operatorname{dt}+[B(t,u(t))+b(t)]dW(t),\\
u(0)&=u_0
\end{cases}
\end{equation*}
for $ t\in [0,T]. $ Our main result will be the existence and uniqueness of strong solutions of this equation. At first, we assume $ z\mapsto A(z): (E,E^1)_{1-1/p,p}\to B(E^1,E) $ to be globally Lipschitz continuous. This assumption will be dropped later on.

\subsection{Globally Lipschitz continuous quasilinearity}
Before, we start, present our setting in detail. We assume
\begin{itemize}
	\item[\lbrack\text{Q1}\rbrack] Let $ p\in (2,\infty) $ and $ E,E^1 $ be UMD Banach spaces with type $ 2 $ or $ p=2 $ and $ E,E^1 $ Hilbert spaces. We assume the embedding $ E^{1}\hookrightarrow E $ to be dense and write $ E^{1/2} $ for the complex interpolation space $ [E,E^{1}]_{1/2} $ and $ E_p $ for the real interpolation space $ (E,E^{1})_{1-1/p,p} $. Moreover, let the family 
	$$ \{J_\delta:\delta>0 \}\subset B\big(L^{p}(\Omega\times(0,\infty);\gamma(H;E)),L^{p}(\Omega\times(0,\infty);E)\big) $$
	defined by
	\[ J_{\delta}b(t):=\delta^{-1/2}\int_{(t-\delta)\vee 0}^{t}b(s)dW(s) \]
	be $ R $-bounded. 
	\item[\lbrack\text{Q2}\rbrack] The initial value $ u_0:\Omega\to E_{p} $ is strongly $ \mathcal{F}_0 $-measurable.
	\item[\lbrack\text{Q3}\rbrack] The mapping $ A:\Omega\times E_p\to B(E^{1},E) $ is such that $ (y,\omega)\mapsto A(\omega,y)x $ is for every $ x\in E^1 $ strongly measurable and $ \omega\mapsto A(\omega,y)x $ is for fixed $ y\in E_p $ and $ x\in E^{1} $ strongly $ \mathcal{F}_{0} $-measurable. 
	\item[\lbrack\text{Q4}\rbrack]For all $ y\in E_p $ and almost all $ \omega\in\Omega, $ the operators $ A(\omega,y) $ are sectorial with $ 0\in\rho(A(\omega,y) ) $ and have a bounded $ H^{\infty}(\Sigma_{\eta}) $-calculus of angle $ \eta\in (0,\pi/2) $, i.e.
	\[ \|\phi(A(\omega,y))\|_{B(E)}\leq C\|\phi\|_{H^{\infty}(\Sigma_{\eta})} \]
	with a constant $ C>0 $ independent of $ \omega $ and $ y. $ 
	\item[\lbrack\text{Q5}\rbrack] There exists $ C_{Q}>0 $, such that for all $ z,y\in E_p $ and almost all $ \omega\in\Omega $, we have
	\[ \|A(\omega,z)-A(\omega,y)\|_{B(E^{1},E)}\leq C_{Q}\|z-y\|_{E_p}. \]
	\item[\lbrack\text{Q6}\rbrack] The mapping	$ F:L_{\mathbb{F}}^0(\Omega; L^{p}(0,T;E^1)\cap C(0,T;E_p)) \to L_{\mathbb{F}}^0(\Omega; L^{p}(0,T;E))$ is such that for a given $ \F $-stopping time $ \widetilde{\tau} $ with $ 0\leq\widetilde{\tau}\leq T $ almost surely, we have $ F(u)\ind_{[0,\widetilde{\tau}]}=F(v)\ind_{[0,\widetilde{\tau}]} $ almost surely, whenever $ u\ind_{[0,\widetilde{\tau}]}=v\ind_{[0,\widetilde{\tau}]}. $  Moreover, there exist $ L_F^{(i)},\widetilde{L}_{F},C_{F}^{(i)}\geq 0 $, $ i=1,2,  $ such that $ F $ is for almost all $ \omega\in\Omega $ of linear growth, i.e. 
	\begin{align*}
	\|F(\cdot,\omega,\phi_1)&\|_{ L^{p}(0,T;E)}\leq C_F^{(1)}\left (1+\|\phi_1\|_{ L^{p}(0,T;E^1)}\right )+C_F^{(2)}\left (1+\|\phi_1\|_{C(0,T;E_p)}\right )
	\end{align*}
	and Lipschitz continuous, i.e.
	\begin{align*}
	\|F(&\cdot,\omega,\phi_1)-F(\cdot,\omega,\phi_2)\|_{ L^{p}(0,T;E)}\\
	&\leq L_{F}^{(1)}\|\phi_1-\phi_2\|_{ L^{p}(0,T;E^1)}+\widetilde{L}_{F}\|\phi_1-\phi_2\|_{_{ L^{p}(0,T;E)}}
	+L_{F}^{(2)}\|\phi_1-\phi_2\|_{C(0,T;E_p)}
	\end{align*}
	for all $ \phi_1,\phi_2\in  L^{p}(0,T;E^1)\cap C(0,T;E_p) $ with constants independent of $ \omega $. 
	\item[\lbrack\text{Q7}\rbrack] The mapping	$ B:L_{\mathbb{F}}^0(\Omega; L^{p}(0,T;E^1)\cap C(0,T;E_p)) \to L_{\mathbb{F}}^0(\Omega; L^{p}(0,T;\gamma(H;E^{1/2})))$ is such that for a given $ \F $-stopping time $ \widetilde{\tau} $ with $ 0\leq\widetilde{\tau}\leq T $ almost surely, we have $ B(u)\ind_{[0,\widetilde{\tau}]}=B(v)\ind_{[0,\widetilde{\tau}]} $ almost surely, whenever $ u\ind_{[0,\widetilde{\tau}]}=v\ind_{[0,\widetilde{\tau}]}. $ Moreover, there exist $ L_B^{(i)},\widetilde{L}_{B},C_{B}^{(i)}\geq 0 $, $ i=1,2,  $ such that $ B $ is for almost all $ \omega\in\Omega $ of linear growth, i.e.
	\begin{align*}
	\|B(\cdot,\omega,\phi_1)&\|_{L^{p}(0,T;\gamma(H;E^{1/2}))}\leq C_B^{(1)}\left (1+\|\phi_1\|_{ L^{p}(0,T;E^{1})}\right )+C_B^{(2)}\left (1+\|\phi_1\|_{C(0,T;E_p)}\right )
	\end{align*}
	and Lipschitz, i.e.
	\begin{align*}
	\|B(&\cdot, \omega,\phi_1)-B(\cdot,\omega,\phi_2)\|_{L^{p}(0,T;\gamma(H;E^{1/2}))}\\
	&\leq L_{B}^{(1)}\|\phi_1-\phi_2\|_{ L^{p}(0,T,E^{1})}+\widetilde{L}_{B}\|\phi_1-\phi_2\|_{ L^{p}(0,T;E)}+L_{B}^{(2)}\|\phi_1-\phi_2\|_{C(0,T;E_p)}
	\end{align*}
	for all  $ \phi_1,\phi_2\in  L^{p}(0,T;E^{1})\cap C(0,T;E_p) $ with constants independent of $ \omega $. 
	\item[\lbrack\text{Q8}\rbrack] Let the constants of $ [\operatorname{Q6}] $ and $ [\operatorname{Q7}] $ be small enough to ensure $$ C_{ \operatorname{MRD}}(L_F^{(1)}+L_F^{(2)})+C_{\operatorname{MRS}}(L_B^{(1)}+L_B^{(2)})<1, $$
 where $ C_{\operatorname{MRD}} $ and $ C_{\operatorname{MRS}} $ are the constants we introduced in \eqref{lp_max_reg_evolution_equation_stochastic_max_reg} and \eqref{lp_max_reg_evolution_equation_deterministic_max_reg}.
	\item[\lbrack\text{Q9}\rbrack] The functions $ f:\Omega\times [0,T]\to E $ and $ b:\Omega\times [0,T]\to\gamma(H;E^{1/2}) $ are strongly measurable and adapted to $\F. $ Moreover, we assume $ f\in L^{p}(\Omega\times[0,T];E) $ and $ b\in L^{p}(\Omega\times[0,T];\gamma(H;E^{1/2})) $.
\end{itemize}
We define strong solutions of $ (\operatorname{QSEE}) $ in the same way as in the corresponding chapter for the linear stochastic evolution equation $ (\operatorname{SEE}) $. The only difference is, that we replace the autonomous operator by $ A(\omega,u(t)). $
\begin{Definition}
	Let $ \mu$ be another $ \F $-stopping times with $ \tau\leq\mu $ almost surely. $ u:\Omega\times [\tau,\mu]\to E $ is called a strong solution of $ (\operatorname{QSEE}) $, if it is strongly measurable, adapted and we have
	\begin{itemize}
		\item [a)] $ u\in L^{1}(\tau,\mu;E^{1}) $, $ F(\cdot,u)\in L^{1}(\tau,\mu;\gamma(H;E)) $ almost surely and $ B(\cdot,u)\ind_{[\tau,\mu]} $ is stochastically integrable.
		\item[b)] The identity 
		\begin{align*}
		u(t)-u_\tau=&-\int_{\tau}^{t}A(u(s))u(s)\operatorname{ds}+\int_{\tau}^{t}F(s,u(s))+f(s)\operatorname{ds}\\
		&+\int_{\tau}^{t}B(s,u(s))+b(s)dW(s)
		\end{align*}
		holds almost surely for all $ t\in [\tau,\mu] $ in $ E $. 
	\end{itemize}
\end{Definition}

Even in the deterministic case, quasilinear evolution equations do not have global solutions without further structural assumptions. Therefore, we now explain the concept of local solutions. The following definition adapts the terms Van Neerven, Veraar and Weis introduced in \cite{nerven_veraar_weis_stochastic_evolution_equations_umd_spaces} to our situation.

\begin{Definition}\label{lp_quasi_stochastic_parabolic_solution_concept}
	Let $ \sigma,\sigma_n $, $ n\in\N, $ be $ \F $-stopping times with $ 0\leq \sigma,\sigma_n\leq T $ almost surely for all $ n\in\N $ and let $ u:\Omega\times [0,\sigma)\to E $ be a stochastic process. 
	\begin{itemize}
		\item [a)] We say that $ \big(u,(\sigma_n)_n,\sigma\big) $ is a local solution of $ (\operatorname{QSEE}), $ if $ (\sigma_n)_{n\in\N} $ is an increasing sequence with $ \lim_{n\to\infty}\sigma_n=\sigma $ pointwise almost surely, such that 
		$$ u (\omega,\cdot)\in L^{p}(0,\sigma_n(\omega);E^{1})\cap C(0,\sigma_n(\omega);E_p) $$ 
		for almost all $ \omega\in\Omega $ and $ u $ is for all $ n\in\N $ a strong solution of $ (\operatorname{QSEE}) $ on $ [0,\sigma_n] $.
		\item [b)]We call a local solution $ \big(u,(\sigma_n)_n,\sigma\big) $ of $ (\operatorname{QSEE}) $ unique, if every local solution $ \big(\widetilde{u},(\widetilde{\sigma_n})_n,\widetilde{\sigma}\big) $ satisfies $ \widetilde{u}(\omega,t)=u(\omega,t)$ for almost all $ \omega\in\Omega $ and for all $ t\in [0,\sigma\wedge\widetilde{\sigma}). $
		\item [c)] We call a local solution $ \big(u,(\sigma_n)_n,\sigma\big) $ of $ (\operatorname{QSEE}) $ a maximal unique local solution, if for any other local solution $ \big(\widetilde{u},(\widetilde{\sigma_n})_n,\widetilde{\sigma}\big) $, we almost surely have $ \widetilde{\sigma}\leq\sigma $ and $ \widetilde{u}(\omega,t)= u(\omega,t) $ for almost all $ \omega\in\Omega $ and all $ t\in [0,\widetilde{\sigma}) $.
	\end{itemize}
\end{Definition}
If the approximating sequence $ \tau_n $ is not important for a result, we shortly write $ (u,\tau). $ In the following we establish a well-posedness result for the quasilinear evolution equation $ (\operatorname{QSEE}) $ up to a maximal stopping time. The next theorem is the main result of this section and will be proved later on.

\begin{Theorem}\label{lp_quasi_stochastic_parabolic_full_result_global_lipschitz}
	If the assumptions $ [\operatorname{Q1}]-[\operatorname{Q9}] $ are satisfied, the quasilinear stochastic evolution equation $ (\operatorname{QSEE}) $
	has a maximal unique local solution $ \big(u,(\tau_n)_n,\tau\big). $ Moreover, we have 
	\[ \mathbb{P}\Big\{\tau<T, \|u\|_{L^{p}(0,\tau;E^{1})}<\infty, u:[0,\tau)\to E_p \text{ is uniformly continuous}\Big \}= 0.  \]
	If we additionally assume $ u_0\in L^{p}(\Omega;E_p), $ the estimate $$ \left (\E \|u\|_{L^{p}(0,\tau_n;E^{1})}^{p}\right )^{1/p}\leq C^{(n)}(1+\|u_0\|_{L^{p}(\Omega,E_p)})$$
	$$ \left (\E \sup_{t\in[0,\tau_n]}\|u(t)\|_{E_p}^{p}\right )^{1/p}\leq C^{(n)}(1+\|u_0\|_{L^{p}(\Omega,E_p)}) $$
	holds true for all $ n\in\N $ and for some $ C^{(n)}>0 $ independent of $ u_0 $.
\end{Theorem}

At first, we prove existence of a strong solution in small balls around the initial value up to a stopping time $ \tau_1. $ Then, iterating this procedure, we construct a local solution $ \big(u,(\tau_n)_n,\tau\big) $ of $ (\operatorname{QSEE}) $. Afterwards we derive a suitable blow-up alternative, which helps us to prove that  $ \big(u,(\tau_n)_n,\tau\big) $ is indeed a maximal unique solution.  

We begin with the definition of a cut-off function $ \theta_{\lambda} $, that helps us to enclose the processes in a suitable ball around the initial value. 
Let 
\begin{equation*}
\Phi(t)=
\begin{cases}
1 &\text{ for }t\in [0,1]\\
-t+2 &\text{ for }t\in [1,2]\\
0 &\text{ for }t\in [2,\infty)
\end{cases}
\end{equation*}
and define $ \Phi_{\lambda}(t):=\Phi(\tfrac{t}{\lambda}) $ which gives us a monotonously decreasing function bounded by $ 1 $, that equals $ 1 $ on $ [0,\lambda] $ and that vanishes on $ [2\lambda,\infty) $. Moreover, $ \Phi_{\lambda} $ is Lipschitz continuous with 
\[ |\Phi_{\lambda}(t)-\Phi_{\lambda}(s)|\leq \lambda^{-1}|t-s| \]
for all $ t,s\geq 0. $ Now we can define the desired cut-off function. For given $ u_a\in E_p $, $ u\in C(a,b;E_p)\cap  L^{p}(a,b;E^{1}) $ and $ t\in [a,b] $ let 
$$ \theta_{\lambda}(a,t,u,u_a):=\Phi_{\lambda}\left (\|u\|_{ L^{p}(a,t;E^{1})}+\sup\limits_{s\in [a,t]}\|u(s)-u_a\|_{E_p}\right ) $$
Clearly, we have $ \theta_{\lambda}(a,t,u,u_a)=0 $ if $\|u\|_{ L^{p}(a,t;E^{1})}+\sup\limits_{s\in [a,t]}\|u(s)-u_a\|_{E_p}\geq 2\lambda$ and if $\|u\|_{ L^{p}(a,t;E^{1})}+\sup\limits_{s\in [a,t]}\|u(s)-u_a\|_{E_p}\leq \lambda,$ we have
\begin{equation*}\label{lp_quasi_stochastic_parabolic_cutoff_identity}
A(u(t))u(t)=A(u_a)u(t)+\theta_{\lambda}(a,t,u,u_a)\big(A(u(t))-A(u_a)\big)u(t).
\end{equation*}
With this fact in mind, it is quite natural to start with the stochastic evolution equation 
\begin{equation}\label{lp_quasi_stochastic_parabolic_linearized_problem}
\begin{cases}
du(t)&=\left [-A(u_\sigma)u(t)+\widetilde{F}_{\lambda}(t,u(t))+f(t)\right]\operatorname{dt}+[B(t,u(t))+b(t)]dW(t)\\
u(\sigma)&=u_\sigma
\end{cases}
\end{equation}
where $ \widetilde{F}_{\lambda} $ is given by
\[ \widetilde{F}_{\lambda}(t,u(t))= \theta_{\lambda}(\sigma,t,u,u_{\sigma})\big(A(u_\sigma)-A(u(t))\big)u(t)+F(t,u(t)). \]
Since we want to sustain local solutions to a maximal time interval, it will be necessary to consider not only the initial time zero but also, as in the previous chapter, an initial $ \F $-stopping time $ \sigma. $ 
First we prove existence and uniqueness of a strong solution of \eqref{lp_quasi_stochastic_parabolic_linearized_problem} using the methods described in the previous chapter and in a second step, we then restrict the solution to an interval $ [\sigma,\sigma+\nu] $, on which we have $ \theta_\lambda(\sigma,t,u(t),u_{\sigma})=1 $. Again $ \nu $ is not a fixed number, but an $ \F $-stopping time, because stochastic processes behave differently from path to path. 

The following Lemma makes sure that the nonlinearity $ \widetilde{F}_\lambda $ satisfies the assumptions of Theorem \ref{lp_auto_stoch_evolution_equation_global_strong_solution_lp_initial_data}, if one chooses $ \lambda $ small enough.

\begin{Lemma}\label{lp_quasi_stochastic_parabolic_quasilinearity_contractive}
	Let $ \sigma $ be a $ \F $-stopping time with $ 0\leq \sigma\leq T $ almost surely and let $ u_\sigma:\Omega\to E_p $ be strongly $ \F_{\sigma} $-measurable. For $ t\in [0,T] $, $ \lambda>0 $, $ \omega\in\Omega $ and $ y\in E^1 $ 
	we define
	\begin{equation*}
	Q_{\lambda,\sigma}(\omega,t,y,u_{\sigma}(\omega)):=\begin{cases}
	\theta_{\lambda}(\sigma(\omega),t,y,u_{\sigma}(\omega))\left (A(\omega,u_{\sigma}(\omega))-A(\omega,y)\right )y&, \text{if  }t\geq\sigma(\omega),\\
	0&, \text{if }t<\sigma(\omega).
	\end{cases}
	\end{equation*}
	If we additionally assume $ [\operatorname{Q1}]-[\operatorname{Q5}] $, $ Q_{\lambda,\sigma} $ is strongly measurable and $ \omega\mapsto Q_{\lambda,\sigma}(\omega,t,y,u_{\sigma}(\omega)) $ is strongly $ \mathcal{F}_{t} $-measurable. Moreover, $ Q_{\lambda,\sigma} $ is of linear growth, i.e
	$$ \|Q_{\lambda,\sigma}(\omega,\cdot,u,u_{\sigma}(\omega)) \|_{ L^{p}(\sigma,T;E)}\leq 4C_{Q}\lambda^{2} $$
	and Lipschitz, i.e
	\begin{align*}
	\|Q_{\lambda,\sigma}(\omega,\cdot,u,u_{\sigma}(\omega))&-Q_{\lambda,\sigma}(\omega,\cdot,v,u_{\sigma}(\omega))\|_{ L^{p}(\sigma,T;E)}\\
	&\leq 6C_{Q}\lambda \left (\|u-v\|_{ L^{p}(\sigma,T;E^{1})}+\|u-v\|_{C(\sigma,T;E_p)}\right )
	\end{align*}
	for almost all $ \omega\in\Omega $ and for all $ u,v\in  L^{p}(\sigma,T;E^{1})\cap C(\sigma,T;E_p) $.
\end{Lemma}
\begin{proof}
The claimed measurability properties of $ Q_{\lambda,\sigma} $ are immediate, since the random variable $ \omega\mapsto Q_{\lambda,\sigma}(\omega,t,y,u_{\sigma})\ind_{t\geq\sigma(\omega)}  $ is by $ [\operatorname{Q3}] $ strongly $ \mathcal{F_{\sigma}} $-measurable. 

To prove the Lipschitz and the growth estimate we argue pathwise for fixed $ \omega\in\Omega $ with $ \sigma(\omega)<T $. In order to keep the notation simple, we suppress the explicit dependence on $ \omega. $ Let $ u,v\in L^{p}(\sigma,T;E^{1})\cap C(\sigma,T;E_p) $ and define
	\[ \sigma_u=\inf\big\{s\in[\sigma,T]: \|u\|_{ L^{p}(\sigma,s;E^{1})}+\|u-u_{\sigma}\|_{C(\sigma,s;E_p)}\geq 2\lambda\big\}\wedge T \]
	and similarly 
	\[ \sigma_v=\inf\big\{s\in[\sigma,T]: \|v\|_{ L^{p}(\sigma,s;E^{1})}+\|v-u_{\sigma}\|_{C(\sigma,s;E_p)}\geq 2\lambda\big\}\wedge T. \]
	Note, that the definition of $ \theta_{\lambda}(\sigma,t,u,u_{\sigma}) $ ensures $ Q_{\lambda,\sigma}(t,u(t),u_{\sigma})=0 $ for $ t\geq \sigma_u $ and $ Q_{\lambda,\sigma}(t,v(t),u_{\sigma})=0 $ for $ t\geq \sigma_v. $ In the following we assume without restriction that $ \sigma_u\geq\sigma_v $. First we prove the growth estimate. $ \theta_{\lambda}\leq 1 $, $ [\operatorname{Q5}] $ and the definition of $ \sigma_u $ yield
	\begin{align*}
	\|Q_{\lambda,\sigma}(\cdot,u,u_{\sigma})\|_{ L^{p}(\sigma,T;E)}&=\|Q_{\lambda,\sigma}(\cdot,u,u_{\sigma})\|_{ L^{p}(\sigma,\sigma_u;E)}\\
	&\leq C_{Q}\sup\limits_{t\in[\sigma,\sigma_u]}\|u(t)-u_{\sigma}\|_{E_p}\|u\|_{ L^{p}(\sigma,\sigma_u;E^{1})}\leq 4C_{Q}\lambda^{2}.
	\end{align*}
	For the Lipschitz estimate, we start with
	\begin{align*}
	\|Q_{\lambda,\sigma}(\cdot,u,u_{\sigma})&-Q_{\lambda,\sigma}(\cdot,v,u_{\sigma})\|_{ L^{p}(\sigma,T;E)}\\
	&\leq \|\big(\theta_{\lambda}(\sigma,\cdot,u,u_{\sigma})-\theta_{\lambda}(\sigma,\cdot,v,u_{\sigma})\big)\big(A(u)-A(u_{\sigma})\big)u\|_{ L^{p}(\sigma,\sigma_u;E)}\\
	&\hspace{0,3cm}+\|\theta_{\lambda}(\sigma,\cdot,v,u_{\sigma})\big(A(u)-A(v)\big)u\|_{ L^{p}(\sigma,\sigma_u;E)}\\
	&\hspace{0,3cm}+\|\theta_{\lambda}(\sigma,\cdot,v,u_{\sigma})\big(A(v)-A(u_\sigma)\big)(u-v)\|_{ L^{p}(\sigma,\sigma_v;E)}.
	\end{align*}
	Note that in the last step we used $ \theta_{\lambda}(\sigma,t,v,u_{\sigma})= 0 $ for $ t\geq\sigma_{v}. $ The assumed Lipschitz continuity of $ \theta_\lambda $ yields
	\begin{align*}
	&\|\big(\theta_{\lambda}(\sigma,\cdot,u,u_{\sigma})-\theta_{\lambda}(\sigma,\cdot,v,u_{\sigma})\big)\big(A(u)-A(u_{\sigma})\big)u\|_{ L^{p}(\sigma,\sigma_u;E)}\\
	&\leq \sup\limits_{t\in[\sigma,\sigma_u]}\big|\theta_{\lambda}(\sigma,t,u,u_{\sigma})-\theta_{\lambda}(\sigma,t,v,u_{\sigma})\big| \sup_{t\in [\sigma,\sigma_u]}\|A(u(t))-A(u_{\sigma})\|_{B(E^{1},E)} \|u\|_{ L^{p}(\sigma,\sigma_u;E^{1})}\\
	&\leq \lambda^{-1}\sup\limits_{s\in[\sigma,T]}\Big|\|u\|_{ L^{p}(a,s;E^{1})}+\|u-u_{\sigma}\|_{C(\sigma,s;E_p)}-\|v\|_{ L^{p}(a,s;E^{1})}-\|v-u_{\sigma}\|_{C(\sigma,s;E_p)}\Big|\\
	&\hspace{1cm}C_{Q}\sup\limits_{t\in[a,\sigma_u]}\|u(t)-u_{\sigma}\|_{E_p}\|u\|_{ L^{p}(\sigma,\sigma_u;E^{1})}\\
	&\leq 4C_{Q}\lambda\left (\|u-v\|_{ L^{p}(\sigma,T;E^{1})}+\|u-v\|_{C(\sigma,T;E_p)}\right ).
	\end{align*}
	In the last step we used the definition of $ \sigma_u $  to estimate the terms not depending on the difference $ u-v. $ Accordingly, we derive
	\begin{align*}
	\|\theta_{\lambda}&(\sigma,\cdot,v,u_{\sigma})\big(A(u)-A(v\big)u\|_{ L^{p}(\sigma,\sigma_u;E)}\leq 2C_{Q}\lambda\|u-v\|_{C(\sigma,T;E_p)}
	\end{align*}
	and
	\begin{align*}
	\|\theta_{\lambda}&(\sigma,\cdot,v,u_{\sigma})\big(A(v)-A(u_\sigma\big)(u-v)\|_{ L^{p}(\sigma,\sigma_v;E)}\leq 2C_{Q}\lambda \|u-v\|_{ L^{p}(\sigma,T;E^{1})}
	\end{align*}
	respectively. After all, we proved 
	\begin{align*}
	\|Q_{\lambda,\sigma}(\cdot,u,u_{\sigma})&-Q_{\lambda,\sigma}(\cdot,v,u_{\sigma})\|_{ L^{p}(\sigma,T;E)}\leq 6C_{Q}\lambda \left (\|u-v\|_{ L^{p}(\sigma,T;E^{1})}+\|u-v\|_{C(\sigma,T;E_p)}\right ),
	\end{align*}
	which is the claimed result.
\end{proof}
Next, we construct a local solution of $ (\operatorname{QSEE}) $ starting from a random initial time. We do this by solving \eqref{lp_quasi_stochastic_parabolic_linearized_problem} and restricting the solution to a random interval on which the solution also satisfies $ (\operatorname{QSEE}) $.
\begin{Proposition}\label{lp_quasi_stochastic_parabolic_local_solution}
	Let $ \sigma $ be an $ \F $-stopping time with $ 0\leq\sigma\leq T $ almost surely, $ u_{\sigma}:\Omega\to E_p $ be strongly $ \mathcal{F}_{\sigma} $-measurable and $ \lambda>0 $ small enough to ensure 
	$$ C_{\operatorname{MRD}}(6C_Q\lambda+L_F^{(1)}+L_F^{(2)})+C_{\operatorname{MRS}}(L_B^{(1)}+L_B^{(2)})<1. $$ If $ [\operatorname{Q1}]-[\operatorname{Q9}] $ are satisfied, the equation 	
	\begin{equation}
	\begin{cases}
	du(t)&=\left [A(u(t))u(t)+F(t,u(t))+f(t)\right]\operatorname{dt}+[B(t,u(t))+b(t)]dW(t)\\
	u(\sigma)&=u_{\sigma}
	\end{cases}
	\end{equation}
	has a unique solution on $ [\sigma,\widetilde{\sigma}] $ with $ u\in L^{p}(\sigma,\widetilde{\sigma};E^1)\cap C(\sigma,\widetilde{\sigma};E_p) $ almost surely. Here the $ \F $-stopping time $ \widetilde{\sigma} $ is given by
	\[ \widetilde{\sigma}=\inf\Big\{t\in [\sigma,T]: \|u-u_{\sigma}\|_{C(\sigma,t;E_p)}+\|u\|_{L^{p}(\sigma,t;E^{1})}>\lambda\Big \}\wedge T. \]
If we additionally assume $ u_\sigma\in L^{p}(\Omega;E_p), $ we also have 
	$$ \big (\E \|u\|_{ L^{p}(\sigma,\widetilde{\sigma};E^{1})}^{p}\big )^{1/p}\leq C\big(1+\|u_\sigma\|_{L^{p}(\Omega,E_p)}\big)$$
	$$ \big (\E \sup_{t\in[\sigma,\widetilde{\sigma}]}\|u(t)\|_{E_p}^{p}\big )^{1/p}\leq C\big(1+\|u_\sigma\|_{L^{p}(\Omega,E_p)}\big) $$
  for some constant $ C>0 $ independent of $ u_\sigma$.
\end{Proposition}
\begin{proof}
	To construct a local solution, we first consider the equation
	\begin{equation}\label{lp_quasi_stochastic_parabolic_helpfful equation}
	\begin{cases}
	du(t)&=\left [-A(u_\sigma)u(t)+F^{(1)}(t,u(t))+f(t)\right]\operatorname{dt}+[B(t,u(t))+b(t)]dW(t),\\
	u(\sigma)&=u_\sigma,
	\end{cases}
	\end{equation}
	where $ F^{(1)} $ is given by
	\[ F^{(1)}(\omega,t,y)= Q_{\lambda,\sigma}(\omega,t,y,u_\sigma)+F(\omega,t,y). \]
	Here we use the mapping $ Q_{\lambda,\sigma} $ we defined in Lemma \ref{lp_quasi_stochastic_parabolic_quasilinearity_contractive}.
	The same Lemma, together with $ [\operatorname{Q6}] $, shows that $ F^{(1)} $ is of linear growth, adapted and we have
	\begin{align*}
	\|F^{(1)}&(\omega,\cdot,u)-F^{(1)}(\omega,\cdot,v)\|_{ L^{p}(\sigma,T;E)}\\
	&\leq (6C_{Q}\lambda+L_{F}^{(1)})\|u-v\|_{ L^{p}(\sigma,T;E^{1})}+ (6C_{Q}\lambda+L_{F}^{(2)})\|u-v\|_{C(\sigma,T;E_p)}+\widetilde{L}_{F}\|u-v\|_{ L^{p}(\sigma,T;E)}
	\end{align*}
	for all $ u,v\in L^{p}(\sigma,T;E^{1})\cap C(\sigma,T;E_p) $. By choice of $ \lambda, $
	 we can apply Corollary \ref{lp_auto_stoch_evolution_equation_global_strong_solution_measurable_initial_data} and obtain a unique strong solution $ u $ on $ [\sigma,T] $ with $ u\in L^{p}(\sigma,T;E^{1})\cap C(\sigma,T;E_p) $ almost surely. Since both $ t\mapsto \|u^{(1)}(\omega,\cdot)\|_{ L^{p}(\sigma,t;E^{1})} $ and $ t\mapsto \|u^{(1)}(\omega,\cdot)-u_0\|_{C(\sigma,t;E_p)} $ are adapted and pathwise almost surely continuous, $ \widetilde{\sigma} $ is an $ \F $-stopping time by Lemma \ref{quasi_stochastic_parabolic_stopping_time}. Moreover, for $ \sigma(\omega)\leq t\leq\widetilde{\sigma}(\omega) $ the identity
	\[ Q_{\lambda,\sigma}(\omega,t,u(t),u_\sigma(\omega))=\big(A(\omega,u_\sigma(\omega))-A(\omega,u(\omega,t))\big)u(\omega,t) \]
	holds and thus $ u $ is a strong solution of 
	\begin{equation*}
	\begin{cases}
	du(t)&=\left [A(u(t))u(t)+F(t,u(t))+f(t)\right]\operatorname{dt}+[B(t,u(t))+b(t)]dW(t)\\
	u(\sigma)&=u_\sigma
	\end{cases}
	\end{equation*}
	on $ [\sigma,\widetilde{\sigma}]. $ In case that $ u_\sigma\in L^{p}(\Omega,E_p) $, we additionally get 
	$$ \big (\E \|u\|_{ L^{p}(\sigma,\widetilde{\sigma};E^{1})}^{p}\big )^{1/p}\leq C(1+\|u_\sigma\|_{L^{p}(\Omega,E_p)})$$
	$$ \big (\E \sup_{t\in[\sigma,\widetilde{\sigma}]}\|u(t)\|_{E_p}^{p}\big )^{1/p}\leq C\big(1+\|u_\sigma\|_{L^{p}(\Omega,E_p)}\big) $$
 as an immediate consequence of \eqref{lp_auto_stoch_evolution_equation_global_strong_solution_lp_initial_data_estimate11} and \eqref{lp_auto_stoch_evolution_equation_global_strong_solution_lp_initial_data_estimate22}.
\end{proof}

Now, we are in the position to prove the main theorem with the following strategy. We already showed, that the set $ \Gamma $ of stopping times $ \widetilde{\tau} $, such that $ (\operatorname{QSEE}) $ has a unique solution $ u $ on $ [0,\widetilde{\tau}] $ is non-empty. Hence, the essential supremum $ \tau $ of $ \Gamma $ (for the definiton of the $ \esssup $ of random variables see e.g. \cite{karatzas_shreve_methods_of_mathematical_finance}, Appendix $ \operatorname{A} $) exists. This $ \tau $ will be our maximal stopping time, that also satisfies the blow-up criterion. However, at first it is unclear, whether $ \tau $ is a stopping time or not. It turns out, that this can be shown, if $ \Gamma $ is closed under pairwise maximization. This property is consequence of the following Lemma.

\begin{Lemma}\label{lp_quasi_stochastic_parabolic_pairwise_maximum_uniqueness}
If $ (u_1,\tau_1) $ and $ (u_2,\tau_2) $ are unique local solutions of $ (\operatorname{QSEE}) $ with $  u_i\in L^{p}(0,\tau_i;E^1)\cap C(0,\tau_i;E_p)  $ almost surely for $ i=1,2 $, then the equation $ (\operatorname{QSEE}) $ has a unique solution $ (u,\tau_1\vee\tau_2) $ with $  u\in L^{p}(0,\tau_1\vee\tau_2;E^1)\cap C(0,\tau_1\vee\tau_2;E_p)  $ almost surely. 

If we additionally assume $  u_i\ind_{[0,\tau_i]}\in L^{p}(\Omega\times [0,T];E^1)  $ and $ \E\sup_{t\in [0,\tau_i)}\|u(t)\|_{E_p}^p<\infty $ for $ i=1,2 $, we also get $  u\ind_{[0,\tau_1\vee\tau_2]}\in L^{p}(\Omega\times [0,T];E^1)  $ and $ \E\sup_{t\in [0,\tau_1\vee\tau_2)}\|u(t)\|_{E_p}^p<\infty $.
\end{Lemma}
\begin{proof}
Define $ u $ by
\begin{align*}
	u(t)=u_1(t\wedge\tau_1)+u_2(t\wedge\tau_2)-u_1(t\wedge\tau_1\wedge\tau_2)
\end{align*}
for $ t\in [0,\tau_1\vee\tau_2]. $ Clearly, $ u $ is adapted as a composition of stopped adapted processes. By uniqueness, we have $ u_1(t)=u_2(t) $ almost surly for every $ t\in [0,\tau_1\wedge\tau_2]. $ Hence, $ u=u_1 $ on $ \{\tau_1>\tau_2\}\times [0,\tau_1) $ and $ u=u_2 $ on $ \{\tau_1\leq \tau_2\}\times [0,\tau_2) $. This proves that $ (u,\tau_1\vee\tau_2) $ is a unique solution of $ (\operatorname{QSEE}) $, that inherits all the regularity properties from $ u_1 $ and $ u_2 $. 
\end{proof}

\begin{proof}[Proof of Theorem \ref{lp_quasi_stochastic_parabolic_full_result_global_lipschitz}]
We define the $ \Gamma$ as the set of all $ \F $-stopping times $ \widetilde{\tau}:\Omega\to [0,T] $ such that there exists a unique solution $ (\widetilde{u},\widetilde{\tau}) $ with $ \widetilde{u}\in L^{p}(0,\widetilde{\tau};E^1)\cap C(0,\widetilde{\tau};E_p)  $ almost surely.

By Proposition \ref{lp_quasi_stochastic_parabolic_local_solution}, this set is non-empty ( start with $ \sigma=0 $, then the corresponding $ \widetilde{\sigma} $ is in $ \Gamma $). Moreover, by Lemma \ref{lp_quasi_stochastic_parabolic_pairwise_maximum_uniqueness}, $ \Gamma $ is closed under pairwise maximization, i.e. if $ \tau_1,\tau_2\in \Gamma$, we also have $ \tau_1\vee\tau_2\in\Gamma. $ Consequently, Theorem $ A.3 $ in \cite{karatzas_shreve_methods_of_mathematical_finance} yields the existence of $ \tau:=\esssup \Gamma $ and of an increasing sequence of stopping times $ (\tau_n)_n $ in $ \Gamma $ with $ \tau=\lim_{n\to\infty}\tau_n $ almost surely. In particular, $ \tau $ is an $ \F $-stopping time as almost sure limit of $ \F $-stopping times. 

Each $ \tau_n $ belongs to a unique solution $ (u_n,\tau_n) $. This can be used to ultimately define the solution of $ (\operatorname{QSEE}) $ on $ [0,\tau). $ We set $ u=u_n $ on $ \Omega\times [0,\tau_n) $. Then, $ u $ is a well-defined process on $ \Omega\times [0,\tau) $ and $ \big(u,(\tau_n)_n,\tau\big ) $ is a unique solution in the sense of Definition \ref{lp_quasi_stochastic_parabolic_solution_concept}. 

Next, we show, that
\[ \widetilde{\Omega}=\{\tau<T,\ \|u\|_{ L^{p}(0,\tau;E^{1})}<\infty,\ u:[0,\tau)\to E_p\text{ is uniformly continuous} \} \]
is a set of measure zero. Assume $ \mathbb{P}(\widetilde{\Omega})>0. $ Since $ u $ is pathwise uniformly continuous on $ \widetilde{\Omega} $, we can extend $ u $ on $ \widetilde{\Omega} $ to the closed interval $ [0,\tau]. $ Moreover, since we have  $ \tau_n\to\tau $ almost surely, we also have  $ \sup_{s\in[\tau_{n},\tau)}\|u(\tau_n)-u(s)\|_{E_p}\to 0 $ and $ \|u\|_{ L^{p}(\tau_n,\tau;E^{1})}\to 0 $ almost surely for $ n\to\infty $ on $ \widetilde{\Omega}. $ 

By Egorov's theorem there exists a subset $ \Lambda\subset\widetilde{\Omega} $ of positive measure such that the above limits are uniform on $ \Lambda $. In particular, there exists $ N\in\N $ such that 
$$ \sup\limits_{s\in[\tau_{N}(\omega),t]}\|u(\omega,\tau_N(\omega))-u(\omega,s)\|_{E_p}+\|u(\omega,\cdot)\|_{ L^{p}(\tau_N(\omega),t;E^{1})}<\lambda/2 $$
for all $ \omega\in\Lambda $ and $ t\in[\tau_N(\omega),\tau(\omega)], $ where $ \lambda>0 $ is chosen in the same way as in Proposition \ref{lp_quasi_stochastic_parabolic_local_solution}.
Moreover, there exists a unique solution $ v $ of 
	\begin{equation*}
	\begin{cases}
	dv(t)&=\left [A(v(t))v(t)+F(t,v(t))+f(t)\right]\operatorname{dt}+[B(t,v(t))+b(t)]dW(t)\\
	v(\tau_N) &=u(\tau_N) 
	\end{cases}
	\end{equation*}
 on $ [\tau_N,\widetilde{\tau}_N] $ with \[ \widetilde{\tau}_N=\inf\Big\{t\in [\tau_N,T]: \|v-u_{\tau_N}\|_{C(\tau_N,t;E_p)}+\|v\|_{L^{p}(\tau_N,t;E^{1})}>\lambda\Big \}\wedge T. \]
The process $ \widetilde{u}:=u\ind_{\ind[0,\tau_N)}+v\ind_{[\tau_N,\widetilde{\tau_N}]} $ solves $ (\operatorname{QSEE}) $ on $ [0,\widetilde{\tau_N}] $. By uniqueness, $ u $ and $ \widetilde{u} $ coincide on $ [0,\tau\wedge\widetilde{\tau}_N). $ However, on $ \Lambda $
we have $ \widetilde{\tau}_N>\tau $ which contradicts the definition of $ \tau $ as $ \esssup $ of $ \Gamma $. All in all, we proved $ \mathbb{P}(\widetilde{\Omega})=0$.
 
 If $ u_0\in L^{p}(\Omega;E_p), $ we replace $ \Gamma $ by 
 \begin{align*}
\widetilde{\Gamma}=\Big\{\sigma\in \Gamma: \text{ the solution }&u^{(\sigma)}\text{ corresponding to }\sigma \text{ satisfies }\\
&u^{(\sigma)}\ind_{[0,\sigma)}\in L^{p}(\Omega\times [0,T];E^1),\E\sup_{t\in [0,\sigma)}\|u(t)\|_{E_p}^p<\infty \Big\}. 
 \end{align*}
 and repeat the argument step by step.
 
 It remains to prove maximality of the solution. Let $ (z,(\mu_n)_n,\mu) $ be another local solution of $ (\operatorname{QSEE}) $. By uniqueness of $ u, $ we get $ z=u $ on $ [0,\tau\wedge\mu). $ Assume, that there is a set of positive measure $ \Lambda\subset\Omega $ with $ \mu>\tau $ on $ \Lambda. $ Then, for almost all $ \omega\in\Lambda $ there exists $ n=n(\omega)\in\N $ with $ \mu_n(\omega)>\tau(\omega). $ In particular, by definition of a local solution, $ u:\Lambda\times[0,\tau]\to E_p $ is pathwise almost surely uniformly continuous and we have $ \|u\|_{L^{p}(0,\tau;E^{1})}<\infty $ on $ \Lambda. $ Thus the blow-up criterion we derived above implies $ \tau=T $ almost surely on $ \Lambda $. But this contradicts $ \mu>\tau $ on $ \Lambda $, since $ \mu $ is also bounded by $ T. $ Hence, we established $ \mu\leq \tau $ almost surely, which is the claimed result.
 \end{proof} 

 We prove, that if two different initial values coincide on a set of positive measure, the corresponding solutions also coincide on this set. 

\begin{Lemma}\label{lp_quasi_stochastic_parabolic_local_solution_uniqueness_different_initial_value}
	Let $ \sigma $ be an $ \F $-stopping time with $ 0\leq\sigma\leq T $ almost surely, $ u_{\sigma}:\Omega\to E_p $, $ v_{\sigma}:\Omega\to E_p $ be strongly $ \mathcal{F}_{\sigma} $-measurable and $ \lambda>0 $ as in Proposition \ref{lp_quasi_stochastic_parabolic_local_solution}.  The unique strong solutions $ u $ and $ v $ of 	
	\begin{equation}
	dz(t)=\left [A(z(t))z(t)+F(t,z(t))+f(t)\right]\operatorname{dt}+[B(t,z(t))+b(t)]dW(t)
	\end{equation}
with initial data $ u_{\sigma} $ and $ v_{\sigma} $ respectively satisfy $ u(\omega,t)=v(\omega,t) $ for almost all $ \omega\in \{u_\sigma=v_\sigma \} $ and all $ t\in [\sigma,\widetilde{\sigma}]. $ Again $ \widetilde{\sigma} $ is given by
\[ \widetilde{\sigma}=\inf\Big\{t\in [\sigma,T]: \|u-u_{\sigma}\|_{C(\sigma,t;E_p)}+\|u\|_{L^{p}(\sigma,t;E^{1})}>\lambda\Big \}\wedge T. \]
\end{Lemma}
\begin{proof}
Writing the equation in the form 
\begin{equation*}
dz(t)=\left [-A(z(\sigma))z(t)+F^{(1)}(t,z(t))+f(t)\right]\operatorname{dt}+[B(t,z(t))+b(t)]dW(t)
\end{equation*}
as in Proposition \ref{lp_quasi_stochastic_parabolic_local_solution} with 
\[ F^{(1)}(\omega,t,y)= Q_{\lambda,\sigma}(\omega,t,y,u_\sigma)+F(\omega,t,y), \]
the result follows from Corollary \ref{lp_auto_stoch_evolution_equation_global_strong_solution_measurable_initial_data_uniqueness}.
\end{proof}

\begin{Corollary}\label{lp_quasi_stochastic_parabolic_local_uniqueness_initial_value}
	Let $ \big(u,\tau\big) $ and $ \big(v,\mu\big) $ be the maximal unique strong solutions of $ (\operatorname{QSEE}) $ to the initial values $ u_0\in E_p $ and $ v_0\in E_p $ respectively. Then, we have $ \tau(\omega)=\mu(\omega) $ and $ u(\omega,t)= v(\omega,t) $ for almost all $ \omega\in\{u_0=v_0 \} $ and all $ t\in [0,\tau(\omega)). $
\end{Corollary}
\begin{proof}
We define the $ \Gamma$ as the set of all $ \F $-stopping times $ \widetilde{\tau}:\Omega\to [0,T] $ such that the maximal unique solutions $ (u,\tau) $ and $ (v,\mu) $ of $ (\operatorname{QSEE}) $ to the initial values $ u_0 $ and $ v_0 $ satisfy $ u(\omega,t)=v(\omega,t) $ for almost all $ \omega\in\Omega $ and for all $ t\in [0,\widetilde{\tau}]. $ Clearly, every stopping time in $ \Gamma $ vanishes on $ \{u_0\neq v_0 \}. $
However, by Lemma \eqref{lp_quasi_stochastic_parabolic_local_solution_uniqueness_different_initial_value}, $ \Gamma $ contains a stopping time that is almost surely strictly positve on $ \{u_0=v_0 \}. $ Moreover, as in the proof of Theorem \ref{lp_quasi_stochastic_parabolic_full_result_global_lipschitz}, we can see that $ \Gamma $ is closed under pairwise maximization. Thus, Theorem $ A.3 $ in \cite{karatzas_shreve_methods_of_mathematical_finance} yields the existence of $ \eta:=\esssup \Gamma $ and of an increasing sequence of stopping times $ (\eta_n)_n $ in $ \Gamma $ with $ \lim_{n\to\infty}\eta_n=\eta $ almost surely. In particular, $ \eta $ is an $ \F $-stopping time, that is also almost surely strictly positive on $ \{u_0=v_0 \} $ and vanishes on $ \{u_0\neq v_0 \}. $

It remains to show $ \eta=\tau=\mu $ almost surely on $ \{u_0=v_0 \}. $ Assume $ \eta<\tau\wedge\mu $ on $ \{u_0=v_0 \}. $ Then, by pathwise uniform continuity of $ u $ and $ v, $ we have $ u(\eta),v(\eta)\in E_p $ almost surely. Consequently, Lemma \ref{lp_quasi_stochastic_parabolic_local_solution_uniqueness_different_initial_value} implies, that the strong solutions $ z_1,z_2 $ of 
	\begin{equation}
dz(t)=\left [A(z(t))z(t)+F(t,z(t))+f(t)\right]\operatorname{dt}+[B(t,z(t))+b(t)]dW(t)
	\end{equation}
with initial value $ u(\eta) $ and $ v(\eta) $ respectively coincide on  $ \{u(\eta)=v(\eta) \}\times [\eta,\widetilde{\eta}] $ for some stopping time $ \widetilde{\eta}>\eta $ almost surely. We define $ \widetilde{z_1}=u\ind_{[0,\eta)}+z_1\ind_{[\eta,\widetilde{\eta}]} $ and $ \widetilde{z_2}=v\ind_{[0,\eta)}+z_2\ind_{[\eta,\widetilde{\eta}]}. $ By construction, $ \widetilde{z_1} $ and $ \widetilde{z_2} $ are solutions of $ (\operatorname{QSEE}) $ with initial data $ \widetilde{z_1}(0)=u_0 $ and $ \widetilde{z_2}(0)=v_0 $ respectively. Hence, by uniqueness, we have $ u=\widetilde{z_1} $ on $ [0,\widetilde{\eta}] $ and $ v=\widetilde{z_2} $ on $ [0,\widetilde{\eta}]. $ However, by almost sure coincidence of $ u $ and $ v $ on $ \{u_0=v_0 \}\times[0,\eta] $ and of $ z_1 $ and $ z_2 $ on $ \{u(\eta)=v(\eta) \}\times [\eta,\widetilde{\eta}] $, we get $ u=v $ almost surely on $ \{u_0=v_0 \}\times[0,\widetilde{\eta}]. $ This contradicts the definition of $ \eta $ as the $ \esssup $ of $ \Gamma $ and hence we proved $ \eta=\tau\wedge\mu $ almost surely on $ \{u_0=v_0 \} $. 

Last but not least, we show that $ \tau=\mu $ almost surely on $ \{u_0=v_0 \} $. Assume $ \mu<\tau $ on $ \Lambda \subset \{u_0=v_0 \}$ with $ \PP(\Lambda)>0. $ Since $ u $ and $ v $ coincide on $ \{u_0=v_0 \}\times [0,\mu) $ and $ u $ has a larger time of existence on $ \Lambda $, we know that $ v:[0,\mu)\to E_p $ is uniformly continuous and $ \|v\|_{L^{p}(0,\mu;E^1)}<\infty $ on $ \Lambda.$ Hence, the blow-up criterion from Theorem \ref{lp_quasi_stochastic_parabolic_full_result_global_lipschitz} implies $ \mu=T $ almost surely on $ \Lambda, $ which contradicts $ \mu<\tau $ on $ \Lambda. $ Consequently, $ \PP(\mu<\tau,u_0=v_0)=0. $ In the same way, we show $ \PP(\tau<\mu,u_0=v_0)=0,$ which closes the proof.
\end{proof}
\subsection{Local Lipschitz continuous quasilinearity}
In $ [\operatorname{Q4}] $ and $ [\operatorname{Q5}] $, we assumed a uniform boundedness of the $ H^{\infty} $-calculus of $ A(\omega,u(t)) $ and a global Lipschitz condition on $ A. $ However, we established a local well-posedness theorie only using local methods. Therefore we can generalize our result in the next section and  allow local Lipschitz continuous nonlinearities and local boundedness of the $ H^{\infty} $-calculus. We replace $ [\operatorname{Q4}]-[\operatorname{Q7}] $ by the following assumptions.
\begin{itemize}
	\item[\lbrack\text{Q4*}\rbrack]For all $ n\in\N, $ there exists $  \mu(n), C(n)>0 $ such that the operators $ \mu(n)+A(\omega,y) $ have a bounded $ H^{\infty}(\Sigma_{\eta(n)}) $-calculus of angle $ \eta(n)\in (0,\pi/2) $ with
	\[ \|\phi(\mu(n)+A(\omega,y))\|_{B(E)}\leq C(n)\|\phi\|_{H^{\infty}(\Sigma_{\eta})} \]
	for all $ \phi\in H^{\infty}(\Sigma_{\eta(n)}) $, $ y\in E_p $ with $ \|y\|_{E_p}\leq n $ and almost all $ \omega\in\Omega. $
	\item[\lbrack\text{Q5*}\rbrack] For all $ n\in\N $ there exist $ C_{Q}(n)>0 $ such that
	\[ \|A(\omega,z)-A(\omega,y)\|_{B(E^{1},E)}\leq C_{Q}(n)\|z-y\|_{E_p} \]
	for all $ y,z\in E_p $ with $ \|y\|_{E_p},\|z\|_{E_p}\leq n $ and almost all $ \omega\in\Omega. $
	\item[\lbrack\text{Q6*}\rbrack] $ F=F_1+F_2:\Omega\times [0,T]\times E_1\to E $ is strongly measurable and $ \omega\mapsto F(\omega,t,x) $ is for all $ t\in[0,T] $ and $ x\in E_1 $ strongly $ \ft $-measurable. $ F_2 $ satisfies the estimates of $ [\operatorname{Q7}]. $ Moreover, given $ n\in\N $, there exist $ L_{F_1}(n),C_{F_1}(n)\geq 0 $, such that $ F_1 $ is locally of linear growth, i.e.
	\begin{align*}
		\|F_1(\omega,t,\phi_1)&\|_{E}\leq C_{F_1}(n)\left (1+\|\phi_1\|_{E_p}\right )
	\end{align*}
	and locally Lipschitz continuous, i.e.
	\begin{align*}
		\|F_1(\omega,t,\phi_1)-F_1(\omega,t,\phi_2)\|_{E}&\leq L_{F_1}(n)\|\phi_1-\phi_2\|_{E_p}
	\end{align*}
	for all $ \phi_1,\phi_2\in E_p $ with $ \|\phi_1\|_{E_p},\|\phi_2\|_{E_p}\leq n $, for all $ t\in [0,T] $ and for almost all $ \omega\in\Omega $.
	\item[\lbrack\text{Q7*}\rbrack] The function $ B=B_1+B_2:\Omega\times [0,T]\times E_1\to \gamma (H,E^{1/2}) $ is strongly measurable and $ \omega\mapsto B(\omega,t,x) $ is for all $ t\in[0,T] $ and $ x\in E_1 $ strongly $ \ft $-measurable. $ B_2 $ satisfies the estimates of $ [\operatorname{Q7}] $. Moreover, there exist $ L_{B_1}(n),C_{B_1}(n)\geq 0 $ such that $ B_1 $ is locally of linear growth, i.e.
	\begin{align*}
		\|B_1(\omega,t,\phi_1)\|_{\gamma(H;E^{1/2})}\leq C_{B_1}(n)\left (1+\|\phi_1\|_{E_p}\right )
	\end{align*}
	and locally Lipschitz, i.e.
	\begin{align*}
\|B_1(\omega,t,\phi_1)-B_1(\omega,t,\phi_2)\|_{\gamma(H;E^{1/2})}\leq L_{B_1}(n)\|\phi_1-\phi_2\|_{E_p}
	\end{align*}
	for all $ \phi_1,\phi_2\in E_p $ with $ \|\phi_1\|_{E_p},\|\phi_2\|_{E_p}\leq n $, for all $ t\in [0,T] $ and for almost all $ \omega\in\Omega $.
\end{itemize}
 To construct a solution of $ (\operatorname{QSEE}) $ for given $ \mathcal{F}_0 $-measurable $ u_0:\Omega\to E_p $, we first investigate the truncated equation 
\begin{equation}\label{lp_quasi_stochastic_parabolic_construction_local_lipschitz_bla} 
\begin{cases}
du(t)&=\left [-A_n(u(t))u(t)+F_n(t,u(t))+f(t)\right]\operatorname{dt}+[B_n(t,u(t))+b(t)]dW(t),\\
u(0)&=u_0\ind_{\Gamma_n},
\end{cases}
\end{equation}
where $ A_n(\omega,y):=A(\omega,R_ny) $, $ F_n(\omega,t,y):=F_1(\omega,t,R_ny)+F_2(\omega,t,y) $, $ B_n(\omega,t,y):=B_1(\omega,t,R_ny)+B_2(\omega,t,y) $ and $ \Gamma_n:=\{\|u_0\|_{E_p}\leq \frac{n}{2} \}. $ The cut-off mapping $ R_n:E_p\to E_p $ is defined by
\begin{equation}\label{lp_quasi_stochastic_parabolic_local_truncated}
R_ny=\begin{cases}
y, &\text{ if }\|y\|_{E_p}\leq n\\
\frac{ny}{\|y\|_{E_p}},&\text{ if }\|y\|_{E_p}> n.
\end{cases}
\end{equation}
  The idea to use such a truncation to extend global Lipschitz nonlinearities to local ones was used several time in case of semilinear equations (see e.g. \cite{brzeniak_on_stochastic_convolution_in_banach_space}, Theorem $ 4.10 $, \cite{seidler_maximal_inequality_revisited}, Proposition $ 5.4 $, \cite{nerven_veraar_weis_stochastic_evolution_equations_umd_spaces}, Theorem $8.1$). The following Lemma can be checked easily.
\begin{Lemma}
	Given $ n\in\N, $ the mapping $ R_n:E_p\to E_p $ defined in \eqref{lp_quasi_stochastic_parabolic_local_truncated} is Lipschitz, i.e.
	\[ \|R_nx-R_ny\|_{E_p}\leq 2\|x-y\|_{E_p}. \]
	In particular, $ A_n $ satisfies $ [\operatorname{Q4}],[\operatorname{Q5}] $ and $ F_n $ and $ B_n $ satisfy $ [\operatorname{Q6}]$ and $ [\operatorname{Q7}]$ respectively.
\end{Lemma}

We can apply Theorem \ref{lp_quasi_stochastic_parabolic_full_result_global_lipschitz} to the truncated equation \eqref{lp_quasi_stochastic_parabolic_construction_local_lipschitz_bla} and obtain for every $ n\in\N $ a unique maximal local solution  $ \big(u_n,(\tau_{nk})_{k},\tau_n\big) $. To do this, note that one can infix the spectral shift from $ [\operatorname{Q4*}] $, i.e. we solve actually solve 
\begin{equation*}
\begin{cases}
du(t)&=\left [-\widetilde{A}_n(u(t))u(t)+\widetilde{F}_n(t,u(t))+f(t)\right]\operatorname{dt}+[B_n(t,u(t))+b(t)]dW(t),\\
u(0)&=u_0\ind_{\Gamma_n},
\end{cases}
\end{equation*}
with $ \widetilde{A}_n(u(t))u(t)=\mu(n)+A_n(u(t))u(t) $ and $ \widetilde{F}_n(t,u(t))=F_n(t,u(t))+\mu(n). $
 In each case, $ u_n $ satisfies $ (\operatorname{QSEE}) $ on $ \Gamma_n\times [0,\sigma_n), $ where $ \sigma_n $ is defined by
\begin{align}\label{lp_quasi_stochastic_parabolic_construction_local_lipschitz_approx_stop}
\sigma_n:=\tau_n\wedge\inf\big\{t\in[0,\tau_n): \|u_n(t)\|_{E_p}>n\big \}.
\end{align}
Note, that $ \sigma_n $ is indeed an $ \F $-stopping time, since $ \tau_n $ is one and entrance times of $ \F $-adapted processes into open sets are also stopping times. In the following Lemma, we show that the sequence $ (\sigma_n)_n $ increases pathwise starting from a large enough $ n\in\N. $

\begin{Lemma}\label{lp_quasi_stochastic_parabolic_local_lipschitz_truncated_uniqueness}
	There is a set $ N\subset\Omega $ with $ \mathbb{P}(N)=0 $ such that the sequence $ \big(\sigma_n(\omega)\big)_{n\in\N} $ is for all $ \omega\in\Omega\setminus N $ monotonously increasing beginning from some $ n=n(\omega)\in\N $. Moreover, we have $ u_k(\omega,t)=u_l(\omega,t) $ for all $ l>k\geq n(\omega) $ and $ t\in [0,\sigma_k(\omega)). $
\end{Lemma}
\begin{proof}
	Given $ \omega\in\Omega $, choose $ n=n(\omega)$ such that $ \omega\in\Gamma_n $. Since $ \|u_0\|_{E_p} $ is almost surely finite, this can be done for almost all $ \omega\in\Omega. $ We first prove, that we have $ u_k(\omega,t)=u_l(\omega,t)$ for almost all $ \omega\in\Gamma_n $ and all $ t\in [0,\sigma_k(\omega)\wedge\sigma_l(\omega)). $ Clearly, both $ u_k $ and $ u_l $ solve 
	\begin{equation}\label{lp_quasi_stochastic_parabolic_construction_local_lipschitz}
	\begin{cases}
	du(t)&=\left [-A_{l}(u(t))u(t)+F_l(t,u(t))+f(t)\right]\operatorname{dt}+[B_l(t,u(t))+b(t)]dW(t),\\
	u(0)&=u_0\ind_{\Gamma_n}
	\end{cases}
	\end{equation}
	in the strong sense on $ [0,\sigma_k) $ and $ [0,\sigma_l) $ respectively and therefore the uniqueness result from Lemma \ref{lp_quasi_stochastic_parabolic_local_uniqueness_initial_value} directly yields the almost sure coincidence of $ u_l $ and $ u_k $ on $ \Gamma_n\times[0,\sigma_l\wedge\sigma_k). $

To prove the pathwise monotonicity of the stopping times on $ \Gamma_n, $ we distinguish the cases $ \Gamma_n=\Lambda_n\mathbin{\dot{\cup}}\widetilde{\Lambda}_n\mathbin{\dot{\cup}}\widehat{N} $ with a null-set $ \widetilde{N}, $
$$ \Lambda_n=\Gamma_n\cap\{\sup_{s\in [0,\tau_l)}\|u_l(s)\|_{E_p}\leq l\} $$ 
and 
$$ \widetilde{\Lambda}_n=\Gamma_n\cap\{\sup_{s\in [0,\tau_l)}\|u_l(s)\|_{E_p}> l\}. $$ 
In particular, we have $ \sigma_l=\tau_l $ on $ \Lambda_n $ and $ \sigma_l=\inf\{t\in[0,\tau_l):\|u_l(t)\|_{E_p}>l \} $ on $ \widetilde{\Lambda}_n. $ As an immediate consequence, we get $ \sigma_k\leq\tau_l=\sigma_l $ almost surely on $ \Lambda_n, $ since $ \tau_l $ was chosen as the maximal stopping time of a solution of \eqref{lp_quasi_stochastic_parabolic_construction_local_lipschitz_bla} which coincides with the maximal time of existence of \eqref{lp_quasi_stochastic_parabolic_construction_local_lipschitz} on $ \Gamma_n $. On $ \widetilde{\Lambda}_n, $ we argue differently. Here, it suffices to note, that by almost sure coincidence of $ u_l $ and $ u_k $ on 
$ \Gamma_n\times[0,\sigma_l\wedge\sigma_k), $ we have
$$ \sup_{s\in[0,\sigma_k\wedge\sigma_l)}\|u_l(s)\|_{E_p}= \sup_{s\in[0,\sigma_k\wedge\sigma_l)}\|u_k(s)\|_{E_p}\leq k, $$
whereas
$$ \sup_{s\in[0,\sigma_l)}\|u_l(s)\|_{E_p}= l. $$
Thus, we must have $ \sigma_k<\sigma_l $ on $ \widetilde{\Lambda}_n. $ Putting these cases together, we finally proved the claimed result, namely $ \sigma_k\leq\sigma_l $ almost surely on $ \Gamma_n. $	
	Last but not least, we choose $ N $ as the union of all sets of measure zero, we excluded in this proof.
\end{proof}

We proved, that $ (\sigma_n)_n$ is, at least for large natural numbers, pathwise almost surely monoton\-ous\-ly increasing and we know from the definition of $ (\sigma_n)_n$ that the sequence is bounded by $ T. $ Therefore we can define the $ \F $-stopping time
\begin{equation}\label{lp_quasi_stochastic_parabolic_local_lipschitz_maximal_time}
\mu=\lim\limits_{n\to\infty}\sigma_n.
\end{equation}

Now let $ \omega\in\Omega\setminus N $ and $ t\in [0,\mu). $ Choose $ n=n(\omega,t) $ large enough such that both $ \omega\in\Gamma_n $ and $ \sigma_n(\omega)>t $ and define
\begin{equation}\label{lp_quasi_stochastic_parabolic_local_lipschitz_solution_definition}
u(\omega,t):=u_n(\omega,t).
\end{equation}
Note, that $ u $ is well-defined. Indeed, let $ \omega\in\Omega\setminus N $, $ t\in [0,\mu(\omega)) $ and $ k$ another natural number with $ (\omega,t)\in \Gamma_k\times[0,\sigma_k) $. Then, Lemma \ref{lp_quasi_stochastic_parabolic_local_lipschitz_truncated_uniqueness} implies $ u_n(\omega,t)=u_k(\omega,t). $
To complete the definition of $ u, $ we set $ u\equiv 0 $ on $ N\times [0,\mu). $ \\ \\
Since the $ u_n $ are strong solutions of $ (\operatorname{QSEE}) $ on $ \Gamma_n\times[0,\sigma_n), $ $ u $ is a good candidate for a local solution of $ (\operatorname{QSEE}) $ on $ \Omega\times [0,\mu). $ We just have to find a sequence of stopping times $ (\mu_n)_n $, that approximates $ \mu $, such that $ u\in C(0,\mu_n;E_p)\cap L^{p}(0,\mu_n;E^{1}) $ almost surely for all $ n\in\N $ and such that $ u $ is a strong solution of $ (\operatorname{QSEE}) $ on $ [0,\mu_n] $. Note, that $ \sigma_n $ does not need to have this property, since we used the maximal stopping times $ \tau_n $ in the definition of $ \sigma_n $ and therefore, we cannot preclude that $ \sigma_n $ is a blow-up time on some paths.
\begin{Theorem} \label{lp_quasi_stochastic_parabolic_full_result_local_lipschitz}
	We assume that $ [\operatorname{Q1}]-[\operatorname{Q3}],[\operatorname{Q4*}],[\operatorname{Q5*}],[\operatorname{Q6}]-[\operatorname{Q9}] $ are satisfied. Then, there is an increasing sequence of $ \F $-stopping times $(\mu_n)_n$, such that $ \big(u,(\mu_n)_n,\mu\big) $ is a maximal unique local solution of
	\begin{equation*}
	(\operatorname{QSEE})\begin{cases}
	du(t)&=\left [-A(u(t))u(t)+F(t,u(t))+f(t)\right]\operatorname{dt}+[B(t,u(t))+b(t)]dW(t)\\
	u(0)&=u_0
	\end{cases}
	\end{equation*}
	Moreover, we have the blow-up criterion
	\[ \mathbb{P}\big\{\mu<T,\ \|u\|_{ L^{p}(0,\mu;E^{1})}<\infty,\ u:[0,\mu)\to E_p\text{ is uniformly continuous} \big\}=0. \]
\end{Theorem}
\begin{proof}
	First we construct the sequence of stopping times $ (\mu_n)_{n\in\N}. $ Recall the definition of $ \sigma_n $ in \eqref{lp_quasi_stochastic_parabolic_construction_local_lipschitz_approx_stop} and of $ \mu $ in \eqref{lp_quasi_stochastic_parabolic_local_lipschitz_maximal_time}. If we additionally set
	\[ \sigma_{nk}:=\tau_{nk}\wedge\inf\big\{t\in [0,\tau_n):\|u_n\|_{E_p}>n\big \}, \]
	we have the pointwise almost sure convergences $ \mu=\lim_{n\to\infty}\sigma_n $ and $ \sigma_n=\lim_{k\to\infty}\sigma_{nk}. $ 
	Since the stopping times $ \sigma_n,\sigma_{nk} $ are all bounded by $ T, $ the dominated convergence theorem yields $ \sigma_{n}\to \mu $ for $ n\to\infty $ and $ \sigma_{nk}\to\sigma_n $ in $ L^{1}(\Omega) $ for $ k\to\infty. $ If we now choose for given $ n\in\N $ the natural number $ k(n) $ such that $ \|\sigma_{n}-\sigma_{nk(n)}\|_{L^{1}(\Omega)}\leq 1/n $, we obtain
	$ \sigma_{nk(n)}\to \sigma $ in $ L^{1}(\Omega) $ for $ n\to\infty. $ Choosing a suitable subsequence still denoted by $ (\sigma_{nk(n)})_{n\in\N} $ yields $ \sigma_{nk(n)}\to \sigma $ pointwise almost surely for $ n\to\infty. $ Moreover, since $ (\Gamma_n)_n $ is an increasing sequence with $ \Omega=\cup_{n\in\N}\Gamma_n $, we also have $ \sigma_{nk(n)}\ind_{\Gamma_n}\to\sigma $ pointwise almost surely for $ n\to\infty. $
	Unfortunately this sequence is not necessarily increasing anymore. Therefore we define
	\[ \mu_n:=\max_{i\in\{1,\dots,n \}}\sigma_{ik(i)}\ind_{\Gamma_i} \]
	and prove that $ (\mu_n)_n $ is the sequence, we wanted to construct. Clearly, since $ \sigma_{nk(n)} $ is for all $ n\in\N $ an $ \F $-stopping time and $ \Gamma_n\in\mathcal{F}_0 $, $ \mu_n $ is also an $ \F $-stopping time. Furthermore the trivial bounds $ \sigma_{nk(n)}\leq\mu_n\leq\mu $ for every $ n\in\N $ yield $ \mu_n\to\mu $ pointwise almost surely. 
	
	It remains to check that $ u $ is strong solution of $ (\operatorname{QSEE}) $ on $ [0,\mu_n]. $ Obviously, it is sufficient to show that $ u $ is a strong solution of $ (\operatorname{QSEE}) $ on $ \Gamma_n\times[0,\sigma_{nk}] $ for all $ n,k\in\N. $
	 We have $ u(\omega,t)=u_n(\omega,t) $ for almost all $ \omega\in \Gamma_{n} $ and all $ t\in[0,\sigma_n(\omega))\supset [0,\sigma_{nk}(\omega)] $ by definition of $ u. $ Since $ u_n $ is a strong solution of the truncated equation \eqref{lp_quasi_stochastic_parabolic_local_truncated} on $ [0,\tau_{nk}] $ and in particular a strong solution of $ (\operatorname{QSEE}) $ on $ \Gamma_n\times [0,\sigma_{nk}] $, we conclude that $ u $ itself is a strong solution of $ (\operatorname{QSEE}) $ on $ \Gamma_n\times[0,\sigma_{nk}]. $
	
	Next, we prove the blow-up alternative for $ u. $ As in the proof of Lemma \ref{lp_quasi_stochastic_parabolic_blow_up}, it is sufficient to show
	\[ \mathbb{P}\big\{\mu<T,\ \|u\|_{ L^{p}(0,\mu;E^{1})}<\infty,\ u:[0,\mu)\to E_p\text{ is uniformly continuous} \big\}=0. \]
	Since uniformly continuous functions on a bounded interval are always bounded, we only need to prove
	$  \mathbb{P}(\Omega_n)=0  $ for every $ n\in\N $, where $ \Omega_n $ is given by
	\begin{align*}
	\Omega_n:=\big\{&\mu<T,\ \|u\|_{ L^{p}(0,\mu;E^{1})}<\infty,\ u:[0,\mu)\to E_p\text{ is uniformly continuous},\\
	& \|u\|_{C(0,\mu;E_p)}\in [\tfrac{n-1}{2},\tfrac{n}{2}) \big\}.
	\end{align*}
	We first show, that for almost all $ \omega\in \{\|u\|_{C(0,\mu;E_p)}\in [\tfrac{n-1}{2},\tfrac{n}{2}) \}, $ we have $ \mu(\omega)=\tau_n(\omega) $. 
	
	Clearly $ \tau_n=\sigma_n $ on $ \{\|u\|_{C(0,\mu;E_p)}\in [\tfrac{n-1}{2},\tfrac{n}{2}) \} $. Furthermore the sequence $ (\sigma_k)_{k\geq n} $ increases on the even larger set $  \Gamma_n $ by Lemma \ref{lp_quasi_stochastic_parabolic_local_lipschitz_truncated_uniqueness} and converges to $ \mu. $ Thus we have $ \tau_n\leq\mu $ on $ \{\|u\|_{C(0,\mu;E_p)}\in [\tfrac{n-1}{2},\tfrac{n}{2}) \} $.
	
On the other hand, we have $ \tau_n\geq\mu$ on $ \{\|u\|_{C(0,\mu;E_p)}\in [\tfrac{n-1}{2},\tfrac{n}{2}) \} $, since on this subset of $ \Omega $, $ u $ solves the truncated equation
\begin{equation}\label{lp_quasi_stochastic_parabolic_full_result_local_lipschitz_truncated}
\begin{cases}
dw(t)&=\left [-A_n(w(t))u(t)+F_n(t,w(t))+f(t)\right]\operatorname{dt}+[B_n(t,w(t))+b(t)]dW(t),\\
w(0)&=u_0\ind_{\Gamma_n},
\end{cases}
\end{equation}
 where $ \Gamma_n $ was given by $ \{\|u_0\|_{E_p}\leq n/2\} $, and $ (u_n,\tau_n) $ was defined as the the maximal unique solution of this equation. This finally proves $ \tau=\mu $ on $ \{\|u\|_{C(0,\mu;E_p)}\in [\tfrac{n-1}{2},\tfrac{n}{2}) \} $ and the above argument also shows $ u(\omega,t)=u_n(\omega,t) $ for almost all $ \omega\in \{\|u\|_{C(0,\mu;E_p)}\in [\tfrac{n-1}{2},\tfrac{n}{2}) \} $ and all $ t\in [0,\mu(\omega). $ In conclusion, we have
	\begin{align*}
	\mathbb{P}\big\{&\mu<T,\ \|u\|_{ L^{p}(0,\mu;E^{1})}<\infty,\ u:[0,\mu)\to E_p\text{ is uniformly continuous},\\
	&\hspace{7,1cm}\|u\|_{C(0,\mu;E_p)}\in [\tfrac{n-1}{2},\tfrac{n}{2}) \big\}\\
	=&\mathbb{P}\big\{\tau_n<T,\ \|u_n\|_{ L^{p}(0,\tau_n;E^{1})}<\infty,\ u_n:[0,\tau_n)\to E_p\text{ is uniformly continuous}, \\
	&\hspace{7,9cm}\|u_n\|_{C(0,\tau_n;E_p)}\in [\tfrac{n-1}{2},\tfrac{n}{2}) \big\}
	\end{align*}
	and by the blow-up Lemma \ref{lp_quasi_stochastic_parabolic_blow_up} this quantity equals zero. 
	
	It remains to check, that $ \big(u,(\mu_n)_n,\mu \big) $ is a maximal unique solution. Let $ \big(v,(\kappa_n)_n,\kappa\big) $ be another local solution of $ (\operatorname{QSEE}). $ We first prove that $ u $ and $ v $ coincide on $ \Omega\times [0,\mu\wedge\kappa). $ 
	Define the sequence $ (\rho_n)_n $  of $\F $-stopping times by
	\begin{align*}
		\rho_n:=\inf\big\{t\in [0,\mu):\|u\|_{E_p}>n\big\}\wedge \inf\big\{t\in [0,\kappa):\|v\|_{E_p}>n\big\}\wedge\mu\wedge\kappa
	\end{align*}
	for $ n\in\N. $ Then both $ u $ and $ v $ solve the truncated equation \eqref{lp_quasi_stochastic_parabolic_full_result_local_lipschitz_truncated} on $ \Gamma\times [0,\rho_n) $ and this equation is uniquely solvable up to a maximal stopping time, which implies $ u(\omega,t)=v(\omega,t) $ for almost all $ \omega\in\Gamma_n $ and all $ t\in [0,\rho_n). $ Since $ \rho_n\to\mu\wedge\kappa $ almost surely for $ n\to\infty $ and $ \cup_{n=1}^{\infty}\Gamma_n=\Omega\setminus \widetilde{N} $ for some set of measure zero $ \widetilde{N}, $ we conclude that $ u $ and $ v $ coincide on $ \Omega\times [0,\mu\wedge\kappa). $ Maximality is then a consequence of the blow-up alternative we derived above. Indeed, if we had $ \kappa > \mu $ on a set of postive measure $ \Lambda, $ then $ u:\Lambda\times[0,\mu)\to E_p $ would be pathwise uniformly continuous and we had $ \|u\|_{L^{p}(0,\mu;E^{1})}<\infty $ almost surely on $ \Lambda $. But this would imply $ \mu=T $ on $ \Lambda $, which contradicts $ \kappa>\mu $ on $ \Lambda, $ since $ \kappa $ is also bounded by $ T. $
\end{proof}

    \section{Examples}
\subsection{\texorpdfstring{A quasilinear parabolic equation on $ \R^{d} $}{A quasilinear equation on the full space}}
In this section, we give a straightforward example, namely
\begin{equation*}
\begin{cases}
du(t)=\big[\sum_{i,j=1}^{d}a_{ij}(\cdot,u(t),\nabla u(t))\partial_i\partial_ju(t)+ f(t)\big]dt+B(t,u(t))dW(t),\\
u(0)=u_0
\end{cases}
\end{equation*}
on $ \R^d $ and we prove existence and uniqueness of a local strong solution in $ L^{p}(\R^{d}) $ under the following hypothesis.
\begin{itemize}
\item  [\lbrack\text{E1}\rbrack] The coefficient matrix $ a=(a_{ij})_{i,j=1,\dots,d}:\R^{d}\times\C\times\C^{d}\to\C^{d\times d} $ is uniformly elliptic, i.e.
\[ \essinf_{y\in\C,z\in\C^{d},x\in \R^{d}}\inf_{|\xi|=1}\re\xi^{T}a(x,y,z)\overline{\xi}=\delta_0>0. \]
Moreover, $ a $ is $ \beta $-H\"older continuous in the first and locally Lipschitz continuous in the second and the third component, i.e. for every $ n\in\N $, there exists $ C,L(n)>0 $, $ \widetilde{L}(n)>0 $ such that
\[ |a(x,y,\widetilde{y})-a(\widetilde{x},z,\widetilde{z})|\leq C|x-\widetilde{x}|^{\beta}+ L(n)|y-z|+\widetilde{L}(n)|\widetilde{y}-\widetilde{z}| \] 
for all $ x,\widetilde{x}\in\R^{d} $ and all $ |y|,|z|, |\widetilde{y}|,|\widetilde{z}|<n$.
\item  [\lbrack\text{E2}\rbrack] We choose $ p,q\in (2,\infty) $, such that $ 1-2/p>d/q $.
\item [\lbrack\text{E3}\rbrack] The initial value $ u_0:\Omega\to B^{2-2/p}_{q,p}(\R^{d}) $ is a strongly $\mathcal{F}_0 $- measurable random variable.
\item [\lbrack\text{E4}\rbrack]
The driving noise $ W $ is an $ l^{2} $- cylindrical Brownian motion of the form
\[ W(t)=\sum_{k=1}^{\infty}e_k\beta_k(t), \]
where $ (e_k)_k $ is the standard orthonormal basis of $ l^2 $ and $ (\beta_k)_k $ is a sequence of independent real-valued Brownian motions relative to the filtration $ (\mathcal{F}_t)_{t\in [0,T]}. $
\item [\lbrack\text{E5}\rbrack] $ B=(B_k^{(1)})_{k\in\N}+(B_k^{(2)})_{k\in\N}:\Omega\times [0,T]\times \R^{d}\times \C\times\C^{d}\to l^2(\N) $ is strongly measurable and $ \omega\mapsto B(\omega,t,x,y,z) $ is for all $ t\in [0,T] $, $x\in \R^{d}$, $ y\in\C $ and $ z\in\C^d $ strongly $ \mathcal{F}_t $-measurable. Furthermore, $ (B_k^{(1)})_k $ is locally of linear growth, i.e.
\[ \|(B_k^{(1)})_k(\omega,t,\cdot,y,\nabla y)\|_{\gamma(l^2;W^{1,q}(\R^{d}))}\leq C(n)(1+\|y\|_{B^{2-2/p}_{q,p}(\R^{d})}) \]
and Lipschitz continuous, i.e. there is $ L(n)>0 $ such that
\begin{align*}
 \hspace{1cm}\|(B_k^{(1)})_k(\omega,t,\cdot,y,\nabla y)-(B_k^{(1)})_k(\omega,t,\cdot,z,\nabla z)\|_{\gamma(l^2;W^{1,q}(\R^{d}))}\leq L(n)\|y-z\|_{B^{2-2/p}_{q,p}(\R^{d})}
\end{align*}
for all $ y,z\in B^{2-2/p}_{q,p}(\R^{d}) $ with norm at most $ n $, all $ t\in [0,T] $ and almost all $ \omega\in\Omega. $ Furthermore $(B_k^{(2)})_k $ is also of linear growth, i.e.
\[ \|(B_k^{(2)})_k(\omega,t,\cdot,y,\nabla y)\|_{\gamma(l^2;W^{1,q}(\R^{d}))}\leq C(1+\|y\|_{W^{2,q}(\R^{d})}) \]
and Lipschitz continuous, i.e. there is $ L_B,\widetilde{L}_B>0 $ such that
\begin{align*}
\|(B_k^{(2)})_k(\omega,t,\cdot,y,\nabla y)-(B_k^{(2)}&)_k(\omega,t,\cdot,z,\nabla z)\|_{\gamma(l^2;W^{1,q}(\R^{d}))}\\
&\leq L_B\|y-z\|_{W^{2,q}(\R^{d})}+\widetilde{L}_B\|y-z\|_{L^{q}(\R^{d})}
\end{align*}
for all $ y,z\in W^{2,q}(\R^{d}) $ with norm at most $ n $, all $ t\in [0,T] $ and almost all $ \omega\in\Omega. $ Here, $ L_B $ must be small enough in the sense of assumption $ [\operatorname{Q8}] $ in the previous section.
\item [\lbrack\text{E5}\rbrack]  $ f\in L^{p}(\Omega\times [0,T];L^{q}(\R^{d})) $ is strongly measurable and $ \F $-adapted.
\end{itemize}
We want to apply Theorem \ref{lp_quasi_stochastic_parabolic_full_result_local_lipschitz} with $ A(z)=-\sum_{i,j=1}^{d}a_{ij}(\cdot,z,\nabla z)\partial_i\partial_j $ and choose the spaces $ E=L^{q}(\R^{d}) $, $ E^1=W^{2,q}(\R^{d}) $ and $ (E,E^1)_{1-1/p,p}=B^{2-2/p}_{q,p}(\R^{d}) $. Due to our assumptions, $ [\operatorname{Q1}]-[\operatorname{Q3}]  $, $ [\operatorname{Q6*}],[\operatorname{Q7*}]  $ and $ [\operatorname{Q8}],[\operatorname{Q9}]  $ are directly fulfilled.

It remains to discuss $ [\operatorname{Q4}^*],[\operatorname{Q5}^*]. $ By \cite{amann_hieber_simonett_bounded_calculus_for_elliptic_operators}, Theorem $ 9.1 $, elliptic operators of the form $ B=\mu-\sum_{i,j=1}b_{ij}(x)\partial_i\partial_j $ with H\"older continuous coefficients have a bounded $ H^{\infty}(\Sigma_{\theta}) $-calculus with angle $ 0<\theta<\pi/2 $ if $ \mu>0 $ is large enough. Following the proof step by step, one sees that $ \mu, \theta $ and $ C>0 $ in
\[ \|f(B)\|_{B(L^{q}(\R^{d}))}\leq C\|f\|_{H^{\infty}(\Sigma_\theta)} \] 
only depend on the supremum, the ellipticity and the modulus of continuity of $ (b_{ij})_{i,j} $. 

Since, we choose $ 1-2/p-d/q>0, $ the Sobolev embedding $ B^{2-2/p}_{q,p}(\R^{d})\hookrightarrow C^{1,\alpha}(\R^{d}) $ holds true for some $ \alpha>0. $ In particular, the modulus of continuity of $ z $,$ \nabla z $ and the quantities $ \| z\|_{\infty} $, $ \|\nabla z\|_{\infty} $ are controlled by $ \|z\|_{B^{2-2/p}_{q,p}(\R^{d})}. $ Consequently, both the modulus of continuity of $ a(\cdot,z,\nabla z) $ and $ \| a(\cdot,z,\nabla z)\|_{\infty} $ depend on $ \|z\|_{B^{2-2/p}_{q,p}(\R^{d})}. $ Thus, given $ n\in\N $, there exists $ C(n),\mu(n)>0,\theta(n)\in (0,\pi/2) $ such that the operators $ \mu(n)+A(z) $ all have a bounded $ H^{\infty}(\Sigma_{\theta(n)}) $-calculus with
\[ \|f(\mu(n)+A(z))\|_{B(L^{q}(\R^{d}))}\leq C(n)\|f\|_{H^{\infty}(\Sigma_{\theta(n)})} \]
for all $ \|z\|_{B^{2-2/p}_{q,p}(\R^{d})}\leq n $. This shows $ [\operatorname{Q4}^*]. $

For given $ y,z\in B^{2-2/p}_{q,p}(\R^{d}) $ with norm at most $ n, $ we estimate
\begin{align*}
\|A(y)v-A(z)v\|_{L^{q}(\R^{d})}&\leq\sum_{i,j=1}^{d}\|(a_{ij}(\cdot,y,\nabla y)-a_{ij}(\cdot,z,\nabla z))\partial_i\partial_j v\|_{L^{q}(\R^{d})}\\
&\leq \sum_{i,j=1}^{d}\|(a_{ij}(\cdot,y,\nabla y)-a_{ij}(\cdot,z,\nabla z))\|_{\infty}\|v\|_{W^{2,q}(\R^{d})}\\
&\leq \sum_{i,j=1}^{d} \big(L(n)\|y-z\|_{\infty}+\widetilde{L}(n)\|\nabla y-\nabla z\|_{\infty}\big)\|v\|_{W^{2,q}(\R^{d})}\\
&\lesssim (L(n)+\widetilde{L}(n))\|y-z\|_{B^{2-2/p}_{q,p}(\R^{d})}\|v\|_{W^{2,q}(\R^{d})}
\end{align*}
almost surely for all $ v\in W^{2,q}(\R^{d}) $, which is the in $ [\operatorname{Q5}^*] $ demanded Lipschitz estimate.

All in all, Theorem \ref{lp_quasi_stochastic_parabolic_full_result_local_lipschitz} yields a maximal unique local strong solution $ (u,(\tau_n)_n,\tau) $ with 
\[ u\in L^{p}(0,\tau_n;W^{2,q}(\R^{d}))\cap C(0,\tau_n;B^{2-2/p}_{q,p}(\R^{d})) \]
pathwise almost surely for every $ n\in\N $ and $ \tau $ satisfies 
	\[ \mathbb{P}\big\{\tau<T,\ \|u\|_{ L^{p}(0,\tau;W^{2,q}(\R^{d}))}<\infty,\ u:[0,\tau)\to B^{2-2/p}_{q,p}(\R^{d})\text{ is uniformly continuous} \big\}=0. \] 
	
\subsection{The incompressible Navier-Stokes system for generalized-Newtonian fluids}
	We now deal with a stochastic model in fluid dynamics. This example is inspired by Bothe and Pr\"uss, who treated the same model in a deterministic setting (see \cite{bothe_pruss_navier_stokes}).
	
	 Throughout this section the divergence of a $ d\times d $-matrix $ T $ is a vector field defined by $ (\Div T )_i=\sum_{k=1}^{d}\partial_kT_{ik} $ and $ \nabla f $ is the Jacobian of the vector field $ f. $ We start with a universal model for fluids, namely
	\begin{equation*}
	(\operatorname{FM})\begin{cases}
	du(t)=[-(u(t)\cdot\nabla)u(t)+\Div S(t)+f(t)]\operatorname{dt}+[g(u(t),\nabla u(t))-\nabla \widetilde{p}]dW(t),\\
	S(t)=\widetilde{\mu}(t)-p(t)I,\\
	\Div u(t)=0,\\
	u(0)=u_0.
	\end{cases}
	\end{equation*}
	Here, $ u:[0,T]\times\R^{d}\to\C^{d} $ is the macroscopic velocity. Since the density of a perfect fluid is assumed to be constant and can therefore be chosen identically one, the continuity equation implies $ \Div(u)=0. $ Moreover, as in every perfect fluid, the total stress tensor $ S:[0,T]\times\R^{d}\to\C^{d\times d} $ is a sum of the viscous stress $ \widetilde{\mu}:[0,T]\times\R^{d}\to\C^{d\times d} $ and the hydrostatic pressure $ pI $, where $ p $ is scalar-valued.
	
	In the following, we discuss generalized Newtonian fluids, that are characterised by the assumption $ \widetilde{\mu}=2\mu(|\mathcal{E}|_{2}^2) \mathcal{E}$, where $ \mathcal{E}=\tfrac{1}{2}(\nabla u+\nabla u^T) $ is the symmetrized derivative of the velocity, the so called rate-of-strain tensor and $ |\cdot|_2 $ is the Hilbert-Schmidt norm on $ \C^{d\times d}. $ There are many examples for this model, e.g. the Ostwald-de-Waele power-law $ \mu(s)=\mu_0s^{m/2-1} $ for $ m\geq 1 $ and $ \mu_0>0 $, the Carreau model $ \mu(s)=\mu_0(1+s)^{m/2-1} $ or the truncated Spriggs law $ \mu(s)=\mu_0s^{m/2-1}\ind_{[s_0,\infty)}(s) $ for some $ s_0>0. $ For details about generalized Newtonian fluids, we refer to chapter $ 5 $ in the monograph of Armstrong, Bird and Hassager (\cite{armstrong_bird_hassager_liquid}). Last but not least, we would like to mention, that the stochastic perturbation of the classical equation covers a model for turbulent flows introduced by Kraichnan, namely noise of the form $ g(u,\nabla u)=(\sigma\cdot\nabla)u+b(u). $ In the mathematical literature, such a noise perturbation was discussed several times in case of Newtonian fluids with $ \widetilde{\mu}=\mu_0\mathcal{E} $ (see e.g. \cite{brezniak_motyl_existence_martingale_solution_navier_stokes}, \cite{mikulevicius_Rozovskii_global_solution_stochastic_navier_stokes} and \cite{nerven_veraar_weis_maximal_lp_stochastic_evolution}).
	
	As a first step, we derive a quasilinear evolution equation from $ (\operatorname{FM}). $ Using the pro\-duct rule and $ \Div(u)=0 $, we calculate 
	\begin{align*}
	(&\Div S)_i=\Div\big(\mu(|\mathcal{E}|_2^2)2\mathcal{E}-pI \big)_{i}=
	\big(\mu(|\mathcal{E}|_2^2)\Div(2\mathcal{E})+\mu'(|\mathcal{E}|_2^2)\nabla (|\mathcal{E}|_2^2)\cdot2\mathcal{E}\big)_i-\partial_i p\\
	&=\mu(|\mathcal{E}|_2^2)\sum_{k=1}^{d}\big(\partial_k\partial_i u_k+\partial_k^2u_i\big)+\mu'(|\mathcal{E}|_2^2)\sum_{j,k,l=1}^{d}(\partial_lu_i+\partial_iu_l)(\partial_ku_j+\partial_ju_k)\partial_k\partial_lu_j-\partial_ip\\
	&=\mu(|\mathcal{E}|_2^2)\sum_{k=1}^{d}\partial_k^2u_i+\mu'(|\mathcal{E}|_2^2)\sum_{j,k,l=1}^{d}(\partial_lu_i+\partial_iu_l)(\partial_ku_j+\partial_ju_k)\partial_k\partial_lu_j-\partial_ip.
	\end{align*} 
	All in all, we get the quasilinear system
	\begin{equation*}\label{lp_quasilinear_example_fluid}
	\begin{cases}
	du(t)=[-A(u(t))u(t)-\nabla p(t)-(u(t)\cdot\nabla)u(t)+f(t)]\operatorname{dt}+[g(u(t),\nabla u(t))-\nabla \widetilde{p}]dW(t),\\
	\Div u(t)=0,\\
	u(0)=u_0,
	\end{cases}
	\end{equation*}
	with $$ (A(z)u)_i= -\mu(|\tfrac{\nabla z+\nabla z^T}{2}|_2^2)\sum_{k=1}^{d}\partial_k^2u_i-\mu'(|\tfrac{\nabla z+\nabla z^T}{2}|_2^2)\sum_{j,k,l=1}^{d}(\partial_lz_i+\partial_iz_l)(\partial_kz_j+\partial_jz_k)\partial_k\partial_lu_j. $$
	We consider this equation on $ L^{p}(\R^{d})^d, $ $ 2<p<\infty $, and as usual in the context of fluid dynamics, we use the Helmholtz decomposition $$ L^{p}(\R^{d})^d=L^{p}_\sigma(\R^{d})\oplus \nabla H^{1}(\R^{d}), $$ where $ L^{p}_\sigma(\R^{d})=\{f\in L^{p}(\R^{d})^d:\Div(f)=0 \}.  $ Note that this decomposition exists for all $ p\in(1,\infty) $ and induces the bounded Helmholtz projection $ P:L^{p}(\R^{d})^d\to L^{p}_\sigma(\R^{d}).$ 
	Applying $ P $ yields the evolution equation
	\begin{equation*}\label{lp_quasilinear_example_helmholtz_non_newtonian}
	(\operatorname{QNS})\begin{cases}
	du(t)=[-PA(u(t))u(t)-P(u(t)\cdot\nabla)u(t)+Pf(t)]\operatorname{dt}+Pg(u(t),\nabla u(t))dW(t),\\
	u(0)=u_0
	\end{cases}
	\end{equation*}
	in $ L^{q}_{\sigma}(\R^{d}) $ for the velocity $ u. $ 
	
	In the following, we use the abbreviations $ B^{s}_{q,p,\sigma}(\R^{d}):=\{f\in B^{s}_{q,p}(\R^{d})^d:\Div(f)=0 \} $ and $ W^{s,q}_{\sigma}(\R^{d})^d:=\{f\in W^{s,k}(\R^{d}):\Div(f)=0 \} $. We thread $ (\operatorname{QNS}) $ under the following assumptions.
	
	\begin{itemize}
		\item  [\lbrack\text{QN1}\rbrack] Let $ \mu:\R_{\geq 0}\to \R_{>0} $ be continuously differentiable, such that $ \mu' $ is still locally Lipschitz continuous, i.e. for every $ n\in\N $ there exists $ C(n)>1 $ such that
		\[ |\mu'(x)-\mu'(y)|+|\mu(x)-\mu(y)|\leq C(n)|x-y| \]
		for all $ 0\leq x,y\leq n. $ Moreover, we assume $ \mu(s)+2s\mu'(s)> 0 $ for all $ s\geq 0. $
		\item  [\lbrack\text{QN2}\rbrack] We choose $ p,q\in (2,\infty) $, such that $ 1-2/p>d/q $.
		\item [\lbrack\text{QN3}\rbrack] The initial value $ u_0:\Omega\to B^{2-2/p}_{q,p,\sigma}(\R^{d}) $ is a strongly $\mathcal{F}_0 $- measurable random variable.
		\item [\lbrack\text{QN3}\rbrack]
		The driving noise $ W $ is an $ l^{2} $- cylindrical Brownian motion of the form
		\[ W(t)=\sum_{k=1}^{\infty}e_k\beta_k(t), \]
		where $ (e_k)_k $ is the standard orthonormal basis of $ l^2 $ and $ (\beta_k)_k $ is a sequence of independent real-valued Brownian motions relative to the filtration $ (\mathcal{F}_t)_{t\in [0,T]}. $
		\item [\lbrack\text{QN4}\rbrack] $ g=(g_n^{(1)})_n+(g_n^{(2)})_n:\Omega\times [0,T]\times \R^{d}\times \C^d\times \C^{d\times d}\to l^2(\N)^d $ is strongly measurable and $ \omega\mapsto g(\omega,t,x,y,z) $ is for all $ t\in [0,T] $, $x\in \R^{d}$, $ y\in\C^d $ and $ z\in\C^{d\times d} $ strongly $ \mathcal{F}_t $-measurable. Furthermore, $ (g_n^{(1)})_n $ is of linear growth, i.e.
		\[ \|(g_n^{(1)})_n(\omega,t,\cdot,z,\nabla z)\|_{\gamma(l^2;W^{1,q}(\R^{d})^d)}\leq C(1+\|z\|_{W^{2,q}_\sigma(\R^{d})}) \]
		and Lipschitz continuous, i.e. there are constants $ L_B,\widetilde{L}_{B}>0 $ such that
		\begin{align*}
		\|(g_n^{(1)})_n(\omega,t,\cdot,y,\nabla y)-(g_n^{(1)})_n&(\omega,t,\cdot,z,\nabla z)\|_{\gamma(l^2;W^{1,q}(\R^{d})^d)}\\&\leq L_B\|y-z\|_{W^{2,q}_\sigma(\R^{d})}+\widetilde{L}_{B}\|y-z\|_{L^{q}_\sigma(\R^{d})}
		\end{align*}
		for all $ y,z\in W^{2,q}_{\sigma}(\R^{d}) $, $ t\in [0,T] $ and almost all $ \omega\in\Omega. $ Here, $ L_B $ must be small enough in the sense of assumption $ [\operatorname{Q8}] $ in the previous section. Furthermore, $ (g_n^{(2)})_n $ is locally of linear growth, i.e.
		\[ \|(g_n^{(2)})_n(\omega,t,\cdot,y,\nabla y)\|_{\gamma(l^2;W^{1,q}(\R^{d})^d)}\leq C(k)(1+\|y\|_{B^{2-2/p}_{q,p,\sigma}(\R^{d})}) \]
		and locally Lipschitz continuous, i.e. 
		\begin{align*}
		\|(g_n^{(2)})_n(\omega,t,\cdot,y,\nabla y)-(g_n^{(2)})_n(\omega,t,\cdot,z,\nabla z)\|_{\gamma(l^2;W^{1,q}(\R^{d})^d)}\leq L(k)\|y-z\|_{B^{2-2/p}_{q,p,\sigma}(\R^{d})}
		\end{align*}
		for all $ y,z\in B^{2-2/p}_{q,p,\sigma}(\R^{d}) $ with norm at most $ k\in\N $, all $ t\in [0,T] $ and almost all $ \omega\in\Omega. $ 
		\item [\lbrack\text{QN5}\rbrack]  $ f\in L^{p}(\Omega\times [0,T];L^{q}(\R^{d})^d) $ is strongly measurable and $ \F $-adapted.
	\end{itemize}
	We want to apply Theorem \ref{lp_quasi_stochastic_parabolic_full_result_local_lipschitz} in $ E=L^{q}_\sigma(\R^{d}) $, $ E^1=W^{2,q}_\sigma(\R^{d}) $. The trace space is then given by $ (E,E^1)_{1-1/p,p}=B^{2-2/p}_{q,p,\sigma}(\R^{d}) $. Due to our assumptions, $ [\operatorname{Q1}]-[\operatorname{Q3}]  $, $ [\operatorname{Q6*}], [\operatorname{Q7*}]  $ and $ [\operatorname{Q8}],[\operatorname{Q9}]  $ are directly fulfilled. We now check $ [\operatorname{Q4*}] $, i.e. we have to prove that $ PA(z) $ has for every $ z\in B^{2-2/p}_{q,p,\sigma}(\R^{d}) $ a bounded $ H^{\infty} $-calculus. In the following Proposition, we restate a result of Bothe and Pr\"uss (see \cite{bothe_pruss_navier_stokes}, proof of Theorem $ 4.1 $). Unlike Bothe and Pr\"uss, we need the precise dependence of all involved constants from $ z. $ Therefore, we need an additional argument.
	
	\begin{Lemma}\label{quasi_example_fluid_r_bounded}
We assume $ [\operatorname{QN1}] $ and $ [\operatorname{QN2}] $. Then, for every $ z\in B^{2-2/p}_{q,p,\sigma}(\R^{d}) $, there exists $ \gamma>0,\theta\in [0,\pi/2) $, such that the operator $ \gamma+PA(z) $ is R-sectorial in $ L^{q}_\sigma(\R^{d}) $ on the sector $ \Sigma_\theta $. Moreover, $ \gamma,\theta $
and the bound $ C_\nu>0 $ in 
\[ \mathcal{R}\big(\big\{\lambda R(\lambda,\gamma+PA(z))\big\}\subset B(L^{q}_\sigma(\R^{d}))\big) \leq C_\nu \]
for given $ \nu>\theta $ only depends on $ \|z\|_{B^{2-2/p}_{q,p,\sigma}(\R^{d})} $.
	\end{Lemma}
	\begin{proof}
		Bothe and Pr\"uss derive from $ [\operatorname{QN1}] $ the strong ellipticity of $A(z)$ (see \cite{bothe_pruss_navier_stokes}, page $ 385 $). Thus, it is sufficient to show that given $ u\in C^{\alpha}(\R^{d})^{d^2}\cap L^{q}(\R^{d})^{d^2} $ and a strongly elliptic operator $ B(u)=-\sum_{|\beta|=2}b_\beta(u) D^{\beta} $ with locally Lipschitz continuous coefficients $ b_\beta:\C^{d^2}\to \C^{d\times d}, $ the above statement holds true with $ PA(\cdot) $ replaced by $ PB(\cdot) $ and $ \theta, C_\nu $ and $ \mu $ only depend on the H\"older norm $ \|u\|_{\alpha} $ and on $ \|u\|_{L^{q}(\R^{d})^{d^2}}. $ Indeed, the above statement then follows directly by the Sobolev embeddings
\[  B^{1-2/p}_{q,p}(\R^{d})^{d\times d}\hookrightarrow C^{\alpha}(\R^{d})^{d^2}\cap L^{q}(\R^{d})^{d^2}\]
for some $ \alpha\in (0,1). $

Note, that in the original, the authors prove that $ \gamma-PB(u) $ has the maximal regularity property in $ L^{q}_\sigma(\R^{d}). $ But this is well-known to be equivalent to our statement. We can follow their argument step by step, we just have to argue, that the spectral shift and the maximal regularity constant only depend on $ \|u\|_{\alpha} $ and on $ \|u\|_{L^{q}(\R^{d})^{d^2}}. $
In Corollary $ 6.2, $ the authors prove, that one still has maximal regularity, if one perturbs a constant coefficient elliptic operator with functions, whose supremum is smaller than some $ \eta>0. $ This $ \eta $ only depends on the ellipticity and the supremum of the coefficients. For the general case, their idea is to use the uniform continuity of the coefficients and convergence at infinity to choose finitely many balls $ B_i $ with center $ x_i $, such that $ |b_\beta(u(x))-b_\beta(u(x_i))|< \eta $ for all $ x\in B_i $ and $ |b_\beta(u(x))-b_\beta(0)|< \eta $ for $ x\notin \cup_{i}B_i $. Then, they localize the equation with a partition of unity subordinate to these balls, solve locally and put the local solutions together. It turns out, that both $ \gamma $ and the maximal regularity constant only depend on the ellipticity and the supremum of the coefficients and the number of balls needed in this argument. So, we have to estimate the number of balls by a quantity that can be controlled by $ \|u\|_{\alpha} $ and $ \|u\|_{L^{q}(\R^{d})^{d^2}}. $

Fix $ u\in C^{\alpha}(\R^{d})^{d^2}\cap L^{q}(\R^{d})^{d^2} $ and let $ C(\|u\|_{\infty})>1 $ such that we have $$ |b_{\beta}(x)-b_{\beta}(y)|\leq C(\|u\|_{\infty})|x-y| $$ for all $ |x|,|y|\leq \|u\|_{\infty}. $
We divide $ \R^{d} $ in the two disjoint subsets $ \big\{|u|\geq\tfrac{\eta}{2C(\|z\|_{\infty})}\big \} $ and $ \{|u|<\tfrac{\eta}{2C(\|z\|_{\infty})} \} $ and we define $ \delta:=\big(\frac{\eta}{6\|u\|_{\alpha}C(\|u\|_{\infty})} \big)^{1/\alpha}. $ Then, by compactness and Vitali's covering Lemma (see e.g. \cite{Grafakos_classical}, Lemma $ 2.1.5 $), there are disjoint balls $ (B^{(i)}_{\delta})_{i=1,\dots,N} $ with radius $ \delta $ and center $ x_i\in \big\{|u|\geq\tfrac{\eta}{2C(\|z\|_{\infty})}\big \},  $ such that $$ \big\{|u|\geq\tfrac{\eta}{2C(\|z\|_{\infty})}\big \}\subset \bigcup_{i=1}^{N}B^{(i)}_{3\delta}. $$
 $ (B^{(i)}_{3\delta})_{i=1,\dots,N} $ are the balls we are looking for. Indeed, for $ x\notin \cup_{i=1}^{N}B^{(i)}_{3\delta}, $ we have $ |u(x)|\leq\eta/2 $ and for $ x,y\in B^{(i)}_{3\delta},  $ we have
 \[ |b_{\beta}(u(x))-b_{\beta}(u(y))|\leq C(\|u\|_{\infty})\|u\|_{\alpha}(3\delta)^\alpha\leq \tfrac{3^{\alpha}\eta}{6}\leq\tfrac{\eta}{2}. \]
It remains to estimate the size of $ N. $ We have 
$$ \cup_{i=1}^{N}B^{(i)}_{\delta}\subset \Big\{|u|> \tfrac{\eta}{4C(\|u\|_{\infty})} \Big\}. $$
Indeed, given $ y\in B^{(i)}_{\delta} $ for some $ i=1,\dots,N, $ we obtain
\begin{align*}
|u(y)|\geq |u(x_i)|-|u(x_i)-u(y)|\geq \tfrac{\eta}{2C(\|u\|_{\infty})}-\|u\|_{\alpha}\delta^{\alpha}=\tfrac{\eta}{2C(\|u\|_{\infty})}-\tfrac{\eta}{6C(\|u\|_{\infty})}=\tfrac{\eta}{3C(\|u\|_{\infty})}.
\end{align*} 
Consequently, using that the $ B^{(i)}_\delta $ are disjoint, we get
\begin{align*}
C_dN\delta^{d}=\Big|\bigcup_{i=1}^{N}B^{(i)}_\delta\Big|\leq \Big|\Big\{|u|> \tfrac{\eta}{4C(\|u\|_{\infty})} \Big\}\Big|\leq \frac{4^qC(\|u\|_{\infty})^q\|u\|^q_{L^{q}(\R^{d})^{d^2}}}{\eta^q}
\end{align*}
with Chebyshev's inequality, where $ C_d $ is the volume of the unit sphere in $ \R^{d}. $ This yields finally 
\[ N\leq \frac{4^q6^{d/\alpha}\|u\|^q_{L^{q}(\R^{d})^{d^2}}\|u\|_{\alpha}^{d/\alpha}C(\|u\|_{\infty})^{d/\alpha+q}}{C_d\eta^{q+d/\alpha}}. \]
	\end{proof}
Next, we conclude, that the operators $ \gamma+PA(z) $ from the above Lemma also have a bounded $ H^{\infty} $-calculus. Our proof of $ [\operatorname{Q4*}] $ adapts the arguments of  \cite{kalton_kunstmann_weis_pertubation_and_interpolation_for_functional_calculus}, Proposition $ 9.5 $ to our situation. A key ingredient is Sneiberg's Lemma.
	\begin{Lemma}\label{sneibergs_lemma}
		Let $ (X_{\theta})_{\theta\in (0,1)} $ and $ (Y_{\theta})_{\theta\in (0,1)} $ be complex interpolation scales of Banach spaces and let $ S:X_{\theta}\to Y_{\theta} $ for each $ \theta\in (0,1) $ be a bounded linear operator. If $ S $ is for some $ \theta_0\in (0,1) $ an isomorphism between $ X_{\theta_0} $ and $ Y_{\theta_0}, $ then there is a $ \delta\in (0,1) $ such that $ S $ is also an isomorphism between $ X_{\mu} $ and $ Y_\mu $ for $ \mu\in (\theta_0-\delta,\theta_0+\delta). $ In particular, $ \|S^{-1}\|_{B(Y_\mu,X_\mu)} $ depends on $ \|S\|_{B(X_\mu,Y_\mu)} $, $ \|S\|_{B(X_{\theta_0},Y_{\theta_0})} $, $ \|S^{-1}\|_{B(Y_{\theta_0},X_{\theta_0})} $ and $ |\mu-\theta_0|. $
	\end{Lemma}
	A proof can be found in \cite{vignati_spectral_theory_and_complex_interpolation}, Theorem $ 3.6. $ The precise dependence of $ \|S^{-1}\|_{B(Y_\mu,X_\mu)} $ on the other parameters is stated in Theorem $ 2.3 $ in the same article. The original proof is due to Sneiberg (see \cite{sneiberg_spectral_properties_of_linear_operators}) in Russian language. 
	\begin{Proposition}
     Given $ z\in B^{2-2/p}_{q,p,\sigma}(\R^{d}), $ the operator $$ u\mapsto B(
     z)u=\gamma u+PA(z)u+P(z\cdot\nabla)u: W^{2,q}(\R^{d})\cap L^{q}_\sigma(\R^{d})\to L^{q}_\sigma(\R^{d}) $$ 
		has a bounded $ H^{\infty}(\Sigma_\theta) $-calculus and the angle $ \theta\in (0,\pi/2), $ the spectral shift $ \gamma $ and the constant $ C>0 $ in
		\[ \|f(B(z))\|_{B(L^{q}_\sigma(\R^{d}))}\leq C\|f\|_{\infty} \]
		only depend $ \|z\|_{B^{2-2/p}_{q,p,\sigma}(\R^{d})}. $ In particular, $ B(z) $ satisfies $ [\operatorname{Q4*}] $ of the previous section.
	\end{Proposition}
	\begin{proof}
		By Lemma \ref{quasi_example_fluid_r_bounded}, $ \gamma-PA(z) $ is $ R $-bounded on $ L^{q}_\sigma(\R^{d}) $ on a sector $ \Sigma_{\theta}, $ $ \theta\in [0,\pi/2) $. The same holds true for $ B(z) $, since $ u\mapsto P(z\cdot\nabla)u $ is just a lower order perturbation (see e.g. \cite{pruess_simonett_moving_interfaces_and_quasilinear}, Proposition $ 4.4.2 $).
		
		Let $ (r_n)_n $ be a sequence of independent Rademacher random variables, $ \nu\in (\theta,\pi) $ and $ (\lambda_j)_{j\in\N}\subset\Sigma_\nu $ be a dense sequence. For $ \eta\in\R, $ we define the norms
		\begin{align*}
		\|(u_j)_j\|_{X_{\eta}}:&=\E\|\sum_{j=1}^{\infty}r_ju_j\|_{W^{\eta+2,q}_\sigma(\R^{d})}+\E\|\sum_{j=1}^{\infty}r_j\lambda_j u_j\|_{W^{\eta,q}_\sigma(\R^{d})}\\
		\|(u_j)_j\|_{Y_{\eta}}:&=\E\|\sum_{j=1}^{\infty}r_j u_j\|_{W^{\eta,q}_\sigma(\R^{d})}
		\end{align*}
		and the spaces 
		\begin{align*}
		X_\eta&:=\big\{(u_j)_j\subset W^{\eta+2,q}_\sigma(\R^{d}):\ \|(u_j)_j\|_{X_\eta}<\infty \big \}\\
		Y_\eta&:=\big\{(u_j)_j\subset W^{\eta,q}_\sigma(\R^{d}):\ \|(u_j)_j\|_{Y_\eta}<\infty \big \}.
		\end{align*}
		Both $ (X_\eta)_{\eta\in\R} $ and $ (Y_\eta)_{\eta\in\R} $ form complex interpolation scales.  We define the operator
		\[ S_\eta: X_\eta\to Y_\eta, (f_j)_j\mapsto \big (\lambda_j-B(z))f_j\big)_j. \]
		Due to its H\"older continu\-ous coefficients, the operator $ B(z):W^{\eta+2,q}_\sigma(\R^{d})\to W^{\eta,q}_\sigma(\R^{d}) $ is bounded, if $ |\eta|<\delta $ for some $ \delta>0 $ small enough. In particular $ S_\eta $ is bounded for $ |\eta|<\delta. $ $ R $-sectoriality of $ B(z) $ on $ L^{q}_\sigma(\R^{d}) $ implies, that $ S_0 $ is an isomorphism with $ S_0^{-1}(u_j)_j=\big((\lambda_j-B(z))^{-1}u_j\big)_j $. By the previous Lemma, $ \|S_0\|_{B(X_0,Y_0)} $, $ \|S_0^{-1}\|_{B(Y_0,X_0)} $ only depend on the ellipticity and the H\"older norm of the coefficients and hence they are determined by $ \|z\|_{B^{2-2/p}_{q,p,\sigma}(\R^{d})}. $
	 By Sneiberg's Lemma, there exists $ \beta>0 $ such that $ S:X_{-\beta}\to Y_{-\beta} $ is an isomorphism and the size of $ \beta $ and $ \|S^{-1}\|_{B(Y_{-\beta}, X_{-\beta})} $ depend on $ \mu $ and $ \|z\|_{B^{2-2/p}_{q,p}(\R^{d})}. $ Especially, we have
	 \[ \E\|\sum_{j=1}^{\infty}r_j\lambda_j(\lambda_j-B(z))^{-1}u_j\|_{W^{-\beta,q}_\sigma(\R^{d})}\leq \|S_{-\beta}^{-1}\|_{B(Y_{-\beta}, X_{-\beta})} \E\|\sum_{j=1}^{\infty}r_ju_j\|_{W^{-\beta,q}_\sigma(\R^{d})}.\]
	 This proves $ R $-sectoriality of $ B(z) $ on $ W^{-\beta,q}_\sigma(\R^{d}) $ with domain  $ W^{2-\beta,q}_\sigma(\R^{d}) $. Indeed, let $ (\widetilde{\lambda}_j)_{j=1}^{N}\subset \C\setminus \Sigma_{\nu}  $ and $ \lambda_j^{(n)}\in (\lambda_k)_k $, $ n\in\N,\ j=1,\dots, N, $ such that $ \lambda_j^{(n)}\to\widetilde{\lambda}_j $ as $ n\to\infty. $ Then, by Fatou and the holomorphie of the resolvent, we have
		\begin{align*}
		\E\|\sum_{j=1}^{N}&r_j\widetilde{\lambda}_jR(\widetilde{\lambda}_j,B(z))f_j\|_{W^{-\beta,q}_\sigma(\R^{d})}\\
		&\leq\liminf_{n\to\infty}\E\|\sum_{j=1}^{N}r_j\lambda^{(n)}_jR(\lambda^{(n)}_j,B(z))f_j\|_{W^{-\beta,q}_\sigma(\R^{d})}\\
		&\leq \|S^{-1}_{-\beta}\|_{B(Y_{-\beta}, X_{-\beta})}\E\|\sum_{j=1}^{N}r_jf_j\|_{W^{-\beta,q}_\sigma(\R^{d})}.
		\end{align*}
		for every $ (f_j)_{j=1}^{N}\subset W^{-\beta,q}_\sigma(\R^{d}) $.
		
		 If we now apply Corollary $ 7.8 $ in \cite{kalton_kunstmann_weis_pertubation_and_interpolation_theorems}, we get that $ B(z) $ has a bounded $ H^{\infty}(\Sigma_\eta) $ calculus on the space $ \langle W^{-\beta,p}_\sigma(\R^{d}),W^{2-\beta,p}_\sigma(\R^{d})\rangle_{\beta/2}. $ Here $ \langle\cdot,\cdot\rangle_\eta $ denotes Rademacher interpolation. Working through the proof of Corollary $ 7.8 $ one sees, that the bound of the calculus only depends on the size of $ |\beta| $ and on $ \|S^{-1}_{-\beta}\|_{B(Y_{-\beta}, X_{-\beta})} $. It remains to identify the Rademacher interpolation space. Since the Helmholtz projection $ P $ commutes with $ I-\Delta $ and $ I-\Delta $ has a bounded $ H^{\infty} $-calculus on $ W^{\alpha,p}(\R^{d})^d $ for every $ \alpha\in\R $, $ p\in (1,\infty) $, this is also true for $ P(I-\Delta)=I-\Delta $ on $ W^{\alpha,p}_\sigma(\R^{d}). $ In this case, by Lemma $ 7.4 $ in \cite{kalton_kunstmann_weis_pertubation_and_interpolation_for_functional_calculus}, the Rademacher interpolation spaces and the complex interpolation spaces coincide. This finally implies
		\begin{align*}
		\langle W^{-\beta,q}_\sigma(\R^{d}),W^{2-\beta,q}_\sigma(\R^{d})\rangle_{\beta/2} = \big(W^{-\beta,q}_\sigma(\R^{d}),W^{2-\beta,q}_\sigma(\R^{d})\big)_{\beta/2}=L^{q}_\sigma(\R^{d}).
		\end{align*}
	\end{proof}
It remains to check $ [\operatorname{Q5*}]. $ Let $ y,z\in B^{2-2/p}_{q,p,\sigma}(\R^{d})  $ with norm at most $ n $ and $ u\in W^{2,q}_\sigma(\R^{d}). $ Recall that 
	$$ A(z)u= -\mu(|\tfrac{\nabla z+\nabla z^T}{2}|_2^2)\sum_{k=1}^{d}\partial_k^2u_i-\mu'(|\tfrac{\nabla z+\nabla z^T}{2}|_2^2)\sum_{j,k,l=1}^{d}(\partial_lz_i+\partial_iz_l)(\partial_kz_j+\partial_jz_k)\partial_k\partial_lu_j. $$
	With the Sobolev embedding $ B^{2-2/p}_{q,p,\sigma}(\R^{d})\hookrightarrow C^{1}_b(\R^{d})^d, $ we estimate
	\begin{align*}
	\|P&A(y)u-PA(z)u\|_{L^{q}_{\sigma}(\R^{d})}\\
	\leq&\Big(\|\mu(|\tfrac{\nabla y+\nabla y^T}{2}|_2^2 )-\mu(|\tfrac{\nabla z+\nabla z^T}{2}|_2^2)\|_{L^{\infty}(\R^{d})}\\
	&\ \ + \|\mu'(|\tfrac{\nabla y+\nabla y^T}{2}|_2^2 )-\mu'(|\tfrac{\nabla z+\nabla z^T}{2}|_2^2)\|_{L^{\infty}(\R^{d})}\|\nabla y\|^2_{L^{\infty}(\R^{d})^{d\times d}}\\
	&\ \ +\|\mu'(|\tfrac{\nabla y+\nabla y^T}{2}|_2^2)\|_{L^{\infty}(\R^{d})}\|\nabla y\|_{L^{\infty}(\R^{d})^{d\times d}}\|\nabla y-\nabla z\|_{L^{\infty}(\R^{d})^{d\times d}} \Big)\|u\|_{W^{2,q}_\sigma(\R^{d})}\\
	\leq& C\big(\|y\|_{L^{\infty}_\sigma(\R^{d})},\|z\|_{L^{\infty}_\sigma(\R^{d})},\|\nabla y\|_{L^{\infty}(\R^{d})^{d\times d}},\|\nabla z\|_{L^{\infty}(\R^{d})^{d\times d}}\big)\|\nabla y-\nabla z\|_{L^{\infty}(\R^{d})^{d\times d}}\|u\|_{W^{2,q}_\sigma(\R^{d})}\\
	\leq &C(n)\| y- z\|_{B^{2-2/p}_{q,p,\sigma}(\R^{d})}\|u\|_{W^{2,q}_\sigma(\R^{d})}.
	\end{align*}
Again using a Sobolev embedding, we get
	\begin{align*}
\|P(y\cdot\nabla)u&-P(z\cdot\nabla)u\|_{L^{q}_{\sigma}(\R^{d})}\\
&\leq \|y-z\|_{L^{\infty}(\R^{d})^d}\|\nabla u\|_{L^{q}(\R^{d})^{d\times d}}\leq \| y- z\|_{B^{2-2/p}_{q,p,\sigma}(\R^{d})}\|u\|_{W^{2,q}_\sigma(\R^{d})}.
	\end{align*}
	Thus $ u\mapsto PA(z)u+P(z\cdot\nabla)u $ satisfies $ [\operatorname{Q5*}]. $ All in all, we can apply Theorem \ref{lp_quasi_stochastic_parabolic_full_result_local_lipschitz} to the equation
		\begin{equation*}
		(\operatorname{QNS})\begin{cases}
		du(t)=[-PA(u(t))u(t)-P(u(t)\cdot\nabla)u(t)+Pf(t)]\operatorname{dt}+Pg(u,\nabla u)dW(t),\\
		u(0)=u_0.
		\end{cases}
		\end{equation*}
	This yields a maximal unique local strong solution $ (u,(\tau)_n,\tau) $ of \eqref{lp_quasilinear_example_helmholtz_non_newtonian} with 
	\[ u\in L^{p}(0,\tau_n;W_\sigma^{2,q}(\R^{d}))\cap C(0,\tau_n;B^{2-2/p}_{q,p,\sigma}(\R^{d})) \]
	pathwise almost surely for every $ n\in\N $ and $ \tau $ satisfies 
	\[ \mathbb{P}\big\{\tau<T,\ \|u\|_{ L^{p}(0,\tau;W^{2,q}_\sigma(\R^{d}))}<\infty,\ u:[0,\tau)\to B^{2-2/p}_{q,p,\sigma}(\R^{d})\text{ is uniformly continuous} \big\}=0. \] 
\subsection{Weak solution of a quasilinear parabolic stochastic equation in divergence form on a bounded domain}
In this section, we consider a convection-diffusion equation on a bounded domain $ D\subset\R^{d} $, $ d\geq 2 $,
\begin{equation*}
(\operatorname{DIV})\begin{cases}
du(t)=\big[\Div(a(u(t))\nabla u(t))+ F(t,u(t))\big]dt+B(t,u(t))dW(t),\\
u(0)=u_0,
\end{cases}
\end{equation*}
with Dirichlet, Neumann or mixed boundary conditions. To simplify the notation, we take real valued coefficients and we look for a real valued solution. We first introduce the spaces we work with.

Let $ \Gamma\subset\partial D $ be open in the topology of $ \partial D. $ For $ q\in (1,\infty) $, we define $ W^{1,q}_{\Gamma}(D) $ as the completion of
\[ C^{\infty}_{\Gamma}(D):=\big\{\phi|_{D}:\phi\in C_{c}^{\infty}(\R^{d})\text{ and }\operatorname{supp}(\phi)\cap (\partial D\setminus\Gamma)=\emptyset \big\} \]
with respect to the norm $ \|\phi\|_{W^{1,q}_{\Gamma}(D)}:=\|\nabla \phi\|_{L^{q}(D)}+\|\phi\|_{L^{q}(D)}. $ Since smooth functions in $  f\in W^{1,q}_{\Gamma}(D)  $ satisfy $f|_{\partial D\setminus\Gamma}=0, $ $ \partial D\setminus\Gamma $ is understood as the Dirichlet part of the boundary, whereas $ \Gamma $ can be interpreted as the Neumann part of the boundary. The space $ W^{-1,q}_{\Gamma}(D) $ is defined as the dual space of $ W^{1,\frac{q}{q-1}}_{\Gamma}(D) $ with respect to the standart $ L^{2} $-duality, which means that
	\[  \langle u,v\rangle_{\big(W^{-1,q}_{\Gamma}(D),W^{1,\frac{q}{q-1}}_{\Gamma}(D)\big)}=\int_{D}u(x)v(x)\operatorname{dx} \]
	if $ u\in  W^{-1,q}_{\Gamma}(D)\cap L^{q}(D)  $ and $ v\in W^{1,\frac{q}{q-1}}_{\Gamma}(D).  $ 	Analogously, we define the Besov space $ B^{1-2/p}_{q,p,\Gamma}(D) $ as the completion of $ C^{\infty}_{\Gamma}(D) $ with respect to the usual Besov norm $ \|\cdot\|_{B^{1-2/p}_{q,p}(D)}. $ These spaces are extensively studied in the literature about equations with mixed boundary on rough domains, but we won't go into detail here and just quote the results we use.
	
	In our case, we always work with $ 1-2/p>d/q $ and therefore, every $ u\in B^{1-2/p}_{q,p,\Gamma}(D) $ is continuous on $ \overline{D} $ and satisfies $ u|_{\partial D\setminus \Gamma}=0. $
	
	We will consider the quasilinear equation $ (\operatorname{DIV}) $ in the space $ W^{-1,q}_{\Gamma}(D)  $ for $ q\in [2,\infty) $, which means, we try to find a weak solution in the sense of partial differential equations. Indeed, $ \big(u,(\tau_n)_n,\tau \big)$ is a local solution of $ (\operatorname{DIV}) $ in the sense of Definition \ref{lp_quasi_stochastic_parabolic_solution_concept} with the choice $ E=W^{-1,q}_{\Gamma}(D) $ and $ E^{1}=W^{1,q}_{\Gamma}(D) $ if and only if the identity
	\begin{align*}
	\int_{D} (u(t,x)-u_0(x))\phi(x)\operatorname{dx}=&-\int_{0}^{t}\int_{D}a(u(s,x))\nabla u(s,x)\nabla\phi(x)\operatorname{dx}\operatorname{ds}\\
	&+\int_{0}^{t}\langle F(s,x,u(s)),\phi\rangle_{(W^{-1,q}_{\Gamma}(D),W^{1,q}_{\Gamma}(D))} \operatorname{ds}\\
	&+\int_{0}^{t}\int_{D}B(s,x,u(s))\phi(x)\operatorname{dx}dW(s)
	\end{align*}
	holds for almost all $ \omega\in\Omega, $ all $ t\in [0,\tau) $ and for all $ \phi\in C_{\Gamma}^{\infty}(D). $
	

\subsection{Local weak solutions}
In this section, we look at $ (\operatorname{DIV}) $ with a locally Lipschitz diffusion matrix $ a(\cdot). $ However, we have to guarantee that the operators $ \Div(a(u(t))\nabla) $ on $ W^{-1,q}_{\Gamma}(D) $ have for every $ t $ the same domain $ W^{1,q}_{\Gamma}(D) $ and in the last decades, it turned out that this property highly depends on $ D $, its dimension and the regularity of the coefficient function. Therefore, we introduce the following notation. Given a uniformly elliptic and bounded coefficient function $ \mu:D\to\R^{d\times d} $, we denote the set of all $ r\in [1,\infty) $, such that the operator
\[ z\mapsto -\Div(\mu\nabla z)+z:W^{1,r}_{\Gamma}(D)\to W^{-1,r}_{\Gamma}(D) \]
is a topological isomorphism, with $ \mathcal{T}_{\mu}. $ We now specify the assumptions.
\begin{itemize}
\item [\lbrack\text{LD1}\rbrack] For every point $ x\in\partial D $, there exists two open sets $ U,V\subset\R^{d} $ and a bi-Lipschitz transformation $ \Phi $ from $ U $ to $ V $ such that $ x\in U $ and $ \Phi(U\cap(D\cup\Gamma)) $ coincides with one of the sets $\{y\in\R^{d}:|y|<1,y_1<0 \}\cup\{y\in\R^{d}:|y|<1,y_1=0,y_2>0\} $ and $ \{y\in\R^{d}:|y|<1 \} $. 
\item  [\lbrack\text{LD2}\rbrack] $ a:D\times\R\to\R^{d\times d} $ is uniformly positive definite, i.e.
\[ \essinf_{y\in\R,x\in D}\inf_{|\xi|=1}\xi^{T}a(x,y)\xi=\delta_0>0, \]
$ a(\cdot,0)\in L^{\infty}(D) $ and $ a $ is locally Lipschitz continuous in the second component, i.e. for every $ \alpha>0 $, there exists $ L(\alpha)>0 $ such that
\[ |a(x,y)-a(x,z)|\leq L(\alpha)|y-z| \] 
for all $ |y|,|z|<\alpha $ and almost all $ x\in D. $
\item  [\lbrack\text{LD3}\rbrack] We choose $ p,q\in (2,\infty) $, such that $ 1-2/p>d/q $ and $ q\in \mathcal{T}_{a(\cdot,z)} $ for all $ z\in B^{1-\frac{2}{p}}_{q,p,\Gamma}(D) $.
\item [\lbrack\text{LD4}\rbrack] The initial value $ u_0:\Omega\to B^{1-2/p}_{q,p,\Gamma}(D) $ is a strongly $\mathcal{F}_0 $- measurable random variable.
\item  [\lbrack\text{LD5}\rbrack] $ F: \Omega\times [0,T]\times D\times\R\to \R $ is strongly measurable and $ \omega\mapsto F(\omega,t,x,y) $ is for all $ t\in [0,T] $, $x\in D$ and $ y\in\R $ strongly $ \mathcal{F}_t $-measurable. Moreover, $ F $ is of linear growth, i.e.
\[ \|F(\omega,t,\cdot,y)\|_{W^{-1,q}_{\Gamma}(D)}\leq C(1+\|y\|_{W^{1,q}_{\Gamma}(D)}) \]
and Lipschitz continuous, i.e. there are constants $ L_F,\widetilde{L}_{F}>0 $ such that
\begin{align*}
\hspace{1,3cm}\|F(\omega,t,\cdot,y)-F(\omega,t,\cdot,z)\|_{W^{-1,q}_{\Gamma}(D)}\leq L_F\|y-z\|_{W^{1,q}_{\Gamma}(D)}+\widetilde{L}_{F}\|y-z\|_{W^{-1,q}_{\Gamma}(D)}
\end{align*}
for all $ y,z\in W^{1,q}_{\Gamma}(D) $, for all $ t\in [0,T] $ and almost all $ \omega\in\Omega. $ Furthermore $ L_F $ must be small enough in the sense of assumption $ [\operatorname{Q8}] $ of the previous section.
\item [\lbrack\text{LD6}\rbrack]
The driving noise $ W $ is an $ l^{2} $- cylindrical Brownian motion of the form
\[ W(t)=\sum_{k=1}^{\infty}e_k\beta_k(t), \]
where $ (e_k)_k $ is the standard orthonormal basis of $ l^2 $ and $ (\beta_k)_k $ is a sequence of independent real-valued Brownian motions relative to the filtration $ (\mathcal{F}_t)_{t\in [0,T]}. $
\item [\lbrack\text{LD7}\rbrack] $ B=(B_n)_n:\Omega\times [0,T]\times D\times \R\to l^2(\N) $ is strongly measurable and $ \omega\mapsto B_n(\omega,t,x,y) $ is for all $ t\in [0,T] $, $x\in D$ and $ y\in\R $ strongly $ \mathcal{F}_t $-measurable. Furthermore, $ B $ is of linear growth, i.e.
\[ \|B(\omega,t,\cdot,y)\|_{\gamma(l^2;L^{q}(D))}\leq C(1+\|y\|_{W^{1,q}_{\Gamma}(D)}) \]
and Lipschitz continuous, i.e. there are constants $ L_B,\widetilde{L}_{B}>0 $ such that
\begin{align*}
\hspace{1,3cm}\|B(\omega,t,\cdot,y)-B(\omega,t,\cdot,z)\|_{\gamma(l^2;L^{q}(D))}\leq L_B\|y-z\|_{W^{1,q}_{\Gamma}(D)}+\widetilde{L}_{B}\|y-z\|_{W^{-1,q}_{\Gamma}(D)}
\end{align*}
for all $ y,z\in W^{1,q}_{\Gamma}(D) $, for all $ t\in [0,T] $, $ n\in\N $ and almost all $ \omega\in\Omega. $ Furthermore $ L_B $ must be small enough in the sense of assumption $ [\operatorname{Q8}] $ of the previous section.
\end{itemize}
Before we proceed, we comment on our assumptions. We chose the requirement on the domain $ [\operatorname{LD1}] $ in order to guarantee the important interpolation results
\[ [W^{-1,q}_{\Gamma}(D),W^{1,q}_{\Gamma}(D)]_{1-1/p.p}=B^{1-2/p}_{q,p,\Gamma}(D),\ [W^{-1,q}_{\Gamma}(D),W^{1,q}_{\Gamma}(D)]_{1/2}=L^{q}(D). \]
In particular, this representation of the real interpolation space makes sure, that $ u_0 $ is in the usual space for initial values. Moreover, $ \lbrack\text{LD3}\rbrack $ implicitly contains assumptions on the boundary of $ D $ and on the coefficient function $ a $ as well, since it is impossible to ensure, that 
\[ z\mapsto -\Div(a(\cdot,u(t))\nabla z)+z:W^{1,q}_{\Gamma}(D)\to W^{-1,q}_{\Gamma}(D) \]
is an isomorphism for all $ q $, if one just assumes  $\lbrack\text{LD1}\rbrack $ and $\lbrack\text{LD2}\rbrack. $ Even in case of the Dirichlet Laplacian, there are counterexamples (see \cite{jerison_kenig_dirichlet_problem_lipschitz}, Theorem $ A $). In general, one only knows, that a small interval $  (2-\varepsilon,2+\varepsilon) $ with $ \varepsilon>0 $ depending on the geometry of $ D $ and $ \Gamma $ and on the coefficient function $ a$ is contained in $ \mathcal{T}_{a(\cdot,z)} $  (see \cite{haller_rehberg_elliptic_and_parabolic_regularity_divergence}, Theorem $ 5.6 $ and Remark $ 5.7 $). Nevertheless, there are several situations, in which one can fulfil $\lbrack\text{LD3}\rbrack. $ In the following, we mention some of them.

If one assumes $ D $ to be a $ C^{1} $-domain, that has either pure Dirichlet ($ \Lambda=\emptyset $) or pure Neumann boundary ($ \Lambda=\partial D $) and one assumes $ \mu $ to be a uniformly continuous coefficient function, one has $ q\in\mathcal{T}_{\mu} $ for all $ q\in (1,\infty). $ This is a classical result, that can be found in \cite{agmon_douglis_nirenberg_estimates_near_boundary}, section $ 15 $ or \cite{morrey_multiple_integrals_calculus_of_variations}, page $ 156 $-$ 157. $ If there is $ C^{1} $-subdomain $ \widetilde{D} $ with positive distance to $ \partial D $, such that both $ \mu|_{\widetilde{D}} $ and $ \mu|_{D\setminus\widetilde{D}} $ are uniformly continuous and $ \mu $ is symmetric, the same result holds true by \cite{elschner_rehberg_gunther_optimal_regularity_c1_interace}, Theorem $ 1.1.$ Consequently, since we require $ 1-2/p>d/q $ and hence every $ z\in B^{1-2/p}_{q,p,\Gamma}(D) $ is even H\"older continuous, we just need to demand that $ a $ is uniformly continuous in $ \widetilde{D} $ and in $ D\setminus\widetilde{D} $ in the first component to ensure $ q\in \mathcal{T}_{a(\cdot,z)} $. 

If $ D $ is just a Lipschitz domain with Dirichlet boundary ($ \Lambda=\emptyset $) and the coefficient function $ \mu $ is a symmetric, uniformly continuous matrix, then there is a $ q>3 $ with $ q\in \mathcal{T}_{\mu} $. This can only be helpful for us if $ d=2,3 $, since then it is possible to choose $ p $ large enough to ensure $ 1-2/p>d/q. $ The same conclusion is true, if $ D $ is Lipschitz and there is a $ C^{1} $-subdomain $ \widetilde{D} $ with positive distance to $ \partial D $, such that both $ \mu|_{\widetilde{D}} $ and $ \mu|_{D\setminus\widetilde{D}} $ are uniformly continuous. These results are all shown in \cite{elschner_rehberg_gunther_optimal_regularity_c1_interace}, Theorem $ 1.1.$

So far, we only gave examples for Dirichlet or Neumann boundary. In case of mixed boundary, we refer to very detailed work \cite{disser_kaiser_rehberg_optimal_sobolev_regularity_for_divergence_elliptic_operators}. In the case $ d=3 $, the authors provide a wide range of geometries of $ D $ and $ \Gamma $ that permit the existence of a $ q>3 $ such that $ q\in\mathcal{T}_{\mu} $, where $ \mu $ is a symmetric coefficient matrix, that is allowed to be only measurable (see Theorem $ 4.8. $). Moreover, in section $ 3 $, they provide many descriptive examples for the geometries, they allow. 

Note, that we could also add locally Lipschitz nonlinearities $ F $ and $ B $ in the sense of $ [\operatorname{Q6*}] $ and $ [\operatorname{Q7*}] $. We just skipped this for sake of simplicity, since we won't need it when we show global well-posedness in case of Dirichlet boundary. 

Our goal is to apply Theorem \ref{lp_quasi_stochastic_parabolic_full_result_local_lipschitz} to the operators
\[ A(u(t))u(t)=-\Div(a(\cdot,u(t))\nabla u(t)+u(t). \]
In the following Lemma, we prove that $ A(u(t)) $ has the needed mapping properties such as a timely constant domain and a bounded $ H^{\infty} $-calculus.
\begin{Lemma} \label{lp_quasi_example_div_properties_of_the_operators}
Under the assumptions $ [\operatorname{LD1}] $-$ [\operatorname{LD3}] $, the operators
\[ A(z)u:=-\Div(a(\cdot,z)\nabla u)+u:W^{1,q}_{\Gamma}(D)\to W^{-1,q}_{\Gamma}(D) \]
are for all $ z\in B^{1-2/p}_{q,p,\Gamma}(D)  $ densely defined, closed with $ 0\in\rho(A(z)) $ and have a bound $ H^{\infty} $-calculus with bound and angle only depending on $ L $, $ \delta_0 $  and on $ \|z\|_{B^{1-2/p}_{q,p,\Gamma}(D)} $. We also have for every $ n\in\N$ the local Lipschitz estimate
\[ \|A(z)-A(y)\|_{B(W^{1,p}_{\Gamma}(D),W^{-1,p}_{\Gamma}(D))}\leq C(n)\|z-y\|_{B^{1-\frac{2}{p}}_{q,p,\Gamma}(D)} \]
for all $ \|z\|_{B^{1-2/p}_{q,p,\Gamma}(D)},\|y\|_{B^{1-2/p}_{q,p,\Gamma}(D)}\leq n $ and some $ C(n)>0. $ Last but not least, we have
$$ [W^{-1,q}_{\Gamma}(D),W^{1,q}_{\Gamma}(D)]_{1-1/p,p}=B^{1-\frac{2}{p}}_{q,p,\Gamma}(D)$$
and as a consequence, $ A $ satisfies the assumptions $ [\operatorname{Q2}] $, $ [\operatorname{Q3}] $, $ [\operatorname{Q4*}] $ and $ [\operatorname{Q5*}] $ of the previous section.
\end{Lemma}

\begin{proof}
By choice of $ p$ and $q $, the Sobolev embedding $ B^{1-\frac{2}{p}}_{q,p,\Gamma}(D)\hookrightarrow C^{l}(\bar{D}) $ holds true for some $ l>0 $. In the sequel, we write $ C_{J} $ for the constant of this embedding.
Given $ z\in B^{1-\frac{2}{p}}_{q,p,\Gamma}(D) $, we obtain 
\begin{align*}
	\|a(\cdot,z)\|_{L^{\infty}(D)}&\leq \esssup_{x\in D}|a(x,z(x))-a(x,0)|+|a(x,0)|\\
	&\leq L\Big(C_{J}\|z\|_{B^{1-\frac{2}{p}}_{q,p,\Gamma}(D)}\Big)C_J\|z\|_{B^{1-\frac{2}{p}}_{q,p,\Gamma}(D)}+\|a(\cdot,0)\|_{L^{\infty}(D)}.
\end{align*}
In particular, the operator $ A(z):W^{1,q}_{\Gamma}(D)\to W^{-1,q}_{\Gamma}(D) $ is well-defined and bounded. Moreover, since we assumed $ q\in\mathcal{T}(a(\cdot,z)) $, Theorem $ 6.5 $ in \cite{disser_elst__rehberg_maximal_regularity_nonautonomous} implies, that $ A(z) $ with $ D(A(z))=W^{1,q}_{\Gamma}(D) $ is a closed operator.

 By Theorem $ 11.5 $ in \cite{auschzer_badr_haller_rehberg_square_root_divergence_form_lp}, $ A(z) $ has a bounded $ H^{\infty} $-calculus of angle $ \arctan\big(\frac{\|a(\cdot,z)\|_{L^{\infty}(D)}}{\delta_0}\big) $ and the bound also only depends on $ \|a(\cdot,z)\|_{L^{\infty}(D)} $ and $ \delta_0 $ (see also \citep{egert_dintelmann_personal_conversation}.) Note, that the critical assumption for this theorem is, that $ A(z) $ possesses the square root property in $ L^{2}(D), $ i.e. the operator $$ (\Div(a(\cdot,z)\nabla)+I)^{1/2}:W^{1,2}_{\Gamma}(D)\to L^{2}(D) $$ is a topolical isomorphism. This result can be found in \cite{egert_haller_tolksdorf_kato_mixed_boundary}, Theorem $ 4.1 $.

The claimed Lipschitz estimate for $ A $ is an immediate consequence of the Lipschitz continuity of $ a $ and of the Sobolev embedding. Indeed, we have
\begin{align*}
 \|A(z)-A(y)\|_{B(W^{1,q}_{\Gamma}(D),W^{-1,q}_{\Gamma}(D))}&\lesssim \|(a(\cdot,z)-a(\cdot,y))\nabla\|_{B(W^{1,q}_{\Gamma}(D),L^{q}(D))}\\
 &\lesssim\|a(\cdot,z)-a(\cdot,y)\|_{L^{\infty}(D)}\\
 &\leq C_JL(C_Jn)\|z-y\|_{B^{1-\frac{2}{p}}_{q,p,\Gamma}(D)}
\end{align*}
 for all $ \|z\|_{B^{1-2/p}_{q,p,\Gamma}}(D),\|y\|_{B^{1-2/p}_{q,p,\Gamma}}(D)\leq n. $
 
 It remains to check $ [W^{-1,q}_{\Gamma}(D),W^{1,q}_{\Gamma}(D)]_{1-1/p,p}=B^{1-\frac{2}{p}}_{q,p,\Gamma}(D). $ By \cite{griepentrog_groeger_kaiser_rehberg_interpolation_for_function_spaces_mixed_boundary}, Lemma $ 3.4, $ we have the identity $ [W^{-1,q}_{\Gamma}(D),W^{1,q}_{\Gamma}(D)]_{1/2}=L^{q}(D). $ Using the reiteration formula between real and complex interpolation (see e.g. \cite{triebel_interpolation_theorie_function_spaces}, Theorem $ 1.10.3.2 $), it is sufficient to show $$[L^{q}(D),W^{1,p}_{\Gamma}(D)]_{1-2/p,p}=B^{1-2/p}_{q,p,\Gamma}(D).$$ This is done in \cite{griepentrog_groeger_kaiser_rehberg_interpolation_for_function_spaces_mixed_boundary}. Remark $ 3.6. $ 
\end{proof}
Next, we check that the spaces $ W^{-1,q}_{\Gamma}(D) $ and $ W^{1,q}_{\Gamma}(D) $ fit in the setting of stochastic maximal $ L^{p} $-regularity.
\begin{Lemma}\label{lp_quasi_example_div_properties_of_the_spaces}
	The spaces $ W^{1,q}_{\Gamma}(D) $ and $ W^{-1,q}_{\Gamma}(D) $ are UMD Banach spaces with type $ 2 $. Moreover the family of operators $$ \{J_\delta:\delta>0 \}\subset B(L^{p}(\Omega\times(0,\infty);\gamma(H;W^{-1,q}_{\Gamma}(D))),L^{p}(\Omega\times(0,\infty);W^{-1,q}_{\Gamma}(D))) $$
	defined by
	\[ J_{\delta}b(t):=\delta^{-1/2}\int_{(t-\delta)\vee 0}^{t}b(s)dW(s) \]
	is $ R $-bounded. In conclusion, these spaces satisfy assumption $ [\operatorname{Q1}] $ of the previous section.
\end{Lemma}

\begin{proof}
By Lemma \ref{lp_quasi_example_div_properties_of_the_operators} the spaces $ W^{-1,q}_{\Gamma}(D) $ and $ W^{1,q}_{\Gamma}(D) $ are isomorph. In the proof of the same Lemma, we checked $ [W^{-1,q}_{\Gamma}(D),W^{1,q}_{\Gamma}(D)]_{1/2}=L^{q}(D) $ and hence amongst others $ A(0)^{1/2} $ provides an isomorphism between $ L^{q}(D) $ and $ W^{-1,q}_{\Gamma}(D). $ Moreover the type of Banach space, the UMD property and the R-boundedness of $ (J_{\delta})_{\delta>0} $ are stable under isomorphisms and the UMD space $ L^{q}(D) $ has type $ 2 $. Noting that by \cite{nerven_veraar_weis_maximal_lp_regularity_stochastic_convolution}, Theorem $ 3.1, $ the family is $ R $-bounded on $L^{p}(\Omega\times(0,\infty);\gamma(H;L^{q}(D)))  $ completes the proof.
\end{proof}
Now, we are in the position to proof existence and uniqueness of a solution of $ (\operatorname{DIV}) $ by applying Theorem \ref{lp_quasi_stochastic_parabolic_full_result_local_lipschitz} to the operators
$ A(z)y=-\Div(a(\cdot,z)\nabla y)+y. $

\begin{Theorem}\label{lp_quasi_example_local_solution}
Let the assumptions $ [\operatorname{LD1}]-[\operatorname{LD7}] $ be satisfied. Then, there exists a maximal unique local solution $ \big(u,(\tau_n)_n,\tau \big) $ of $ (\operatorname{DIV}) $ in $ W^{-1,q}_{\Gamma}(D) $, such that we have 
$$ u\in L^{p}(0,\tau_n;W^{1,q}_{\Gamma}(D))\cap C(0,\tau_n;B^{1-2/p}_{q,p,\Gamma}(D)) $$ 
pathwise almost surely for every $ n\in\N $. Moreover, $ \tau $ satisfies
	\[ \mathbb{P}\big\{\tau<T,\ \|u\|_{ L^{p}(0,\tau;W^{1,q}_{\Gamma}(D))}<\infty,\ u:[0,\tau)\to B^{1-2/p}_{q,p,\Gamma}(D)\text{ is uniformly continuous} \big\}=0. \]
\end{Theorem}
\begin{proof}
Writing $$ \Div(a(\cdot,z)\nabla z)+F(t,z)= \big(\Div(a(\cdot,z)\nabla z)-z\big)+\big(F(t,z)+z\big), $$ we see, that we can solve the equation
\begin{equation*}
\begin{cases}
du(t)&=\left [-A(u(t))u(t)+\widetilde{F}(t,u(t))\right]\operatorname{dt}+B(t,u(t))dW(t),\\
u(0)&=u_0
\end{cases}
\end{equation*}
with $ \widetilde{F}(t,z):= F(t,z)+z$ for $ z\in W^{1,q}_{\Gamma}(D) $.
By Lemma \ref{lp_quasi_example_div_properties_of_the_operators} the assumptions $ [\operatorname{Q2}] $, $ [\operatorname{Q3}] $, $ [\operatorname{Q4*}] $ and $ [\operatorname{Q5*}] $ fulfilled, whereas Lemma \ref{lp_quasi_example_div_properties_of_the_spaces} guaranties $ [\operatorname{Q1}]. $ $ [\operatorname{LD6}] $, $ [\operatorname{LD7}] $ make sure, that the nonlinearities $ F $ and $ B $ also satisfy $ [\operatorname{Q6}] $-$ [\operatorname{Q8}] $. All in all, Theorem \ref{lp_quasi_stochastic_parabolic_full_result_local_lipschitz} yields the desired result.
\end{proof}

\subsection{Global weak solution with Dirichlet boundary condition}
In this section, we investigate the convection diffusion equation with Dirichlet boundary conditions $ (\Gamma=\emptyset) $ and we therefore restrict us to the space $ W^{1,q}_{\emptyset}(D) $, that will be denoted with $ W^{1,q}_{0}(D) $ in what follows. As usual in the literature, we write  $W^{-1,q}(D) $ for $ W^{-1,q}_{\emptyset}(D). $ We consider the equation
\begin{equation*}
(\operatorname{GDIV})\begin{cases}
du(t)=\big[\Div(a(u(t))\nabla u(t))+ \Div (G(u(t)))\big]dt+B(t,u(t))dW(t) ,\\
u(0)=u_0,\\
\end{cases}
\end{equation*}
and we strengthen the assumptions in order to prove that the local solution from Theorem \ref{lp_quasi_example_local_solution} exists on the whole interval $ [0,T] $. We require:
\begin{itemize}
	\item [\lbrack\text{GD1}\rbrack] $ D\subset\R^{d}$ is a bounded  $ C^{1} $-domain.
	\item  [\lbrack\text{GD2}\rbrack] $ a:\R\to\R^{d\times d} $ is bounded and uniformly positive definite, i.e.
	\[ \inf_{y\in\R}\inf_{|\xi|=1}\xi^{T}a(y)\xi=\delta_0>0 \]
	 and $ a $ is globally Lipschitz continuous, i.e there exists $ L>0 $ such that
	\[ |a(y)-a(z)|\leq L|y-z| \] 
	for all $ y,z\in\R $.
	\item  [\lbrack\text{GD3}\rbrack] We choose $ p,q\in (2,\infty) $, such that $ 1-2/p>d/q $.
	\item [\lbrack\text{GD4}\rbrack] The initial value $ u_0:\Omega\to B^{1-2/p}_{q,p,0}(D) $ is a strongly $ \mathcal{F}_0 $-measurable random variable.
	\item  [\lbrack\text{GD5}\rbrack] $ G:\R\to \R^{d} $ is Lipschitz continuous, i.e.
	there is  $ L_G>0 $ such that
	\[ |G(y)-G(z)|\leq L_G|y-z| \]
	for all $ y,z\in \C $.  
	\item [\lbrack\text{GD6}\rbrack]
	The driving noise $ W $ is an $ l^{2} $- cylindrical Brownian motion with the decomposition
	\[ W(t)=\sum_{k=1}^{\infty}e_k\beta_k(t), \]
	where $ (e_k)_k $ is the standart orthonormal basis of $ l^2 $ and $ (\beta_k)_k $ is a sequence of independent real-valued Brownian motions relative to the filtration $ (\mathcal{F}_t)_{t\in [0,T]}. $
	\item [\lbrack\text{GD7}\rbrack] $ B=(B_n)_n:\Omega\times [0,T]\times D\times \R\to l^2(\N) $ is strongly measurable and $ \omega\mapsto B(\omega,t,x,y) $ is for all $ t\in [0,T] $, $x\in D$ and $ y\in\R $ strongly $ \mathcal{F}_t $-measurable. Furthermore, $ B $ is of linear growth, i.e.
	\[ \Big(\sum_{n=1}^{\infty}|B_n(\omega,t,x,y)|^2\Big)^{1/2}\leq C(1+|y|) \]
	and Lipschitz continuous in the last component, i.e. there is $ L_B>0 $ such that
	\[ \Big(\sum_{n=1}^{\infty}|B_n(\omega,t,x,y)-B_n(\omega,t,x,z)|^2\Big)^{1/2}\leq L_B|y-z| \]
	for all $ y,z\in \C $, $ t\in [0,T] $, $ x\in D $ and almost all $ \omega\in\Omega. $ Moreover, we assume
	\[ \|B_n(\omega,t,\cdot,f)\|_{\gamma(l^2;W^{1,2}_0(D))}\leq C(1+\|f\|_{W^{1,2}_0(D)}) \]
	for all $ f\in W^{1,2}_0(D), $ $ t\in [0,T] $, $ x\in D $ and almost all $ \omega\in\Omega. $
\end{itemize}

These assumptions are strictly stronger than $ [\operatorname{LD1}] $-$ [\operatorname{LD7}]. $ $ a $ is not locally, but globally Lipschitz and the nonlinearites $ \Div(G) $ and $ B $ are only of lower order. As we have already mentioned in our remarks below the assumptions of the previous section, $\lbrack\text{GD1}\rbrack$ and $\lbrack\text{GD2}\rbrack$ also imply $ q\in\mathcal{T}_{a(\cdot,z)} $ for every $ z\in B^{1-2/p}_{q,p,0}(D) $ and $ q\in (1,\infty). $ 

All in all, Theorem \ref{lp_quasi_example_local_solution} yields a local solution $ (u,(\tau_n)_n,\tau) $ of $ (\operatorname{GDIV}), $ i.e an increasing sequence of $ \F $-stopping times $ (\tau_n)_n $ with $ 0\leq\tau_n\leq T $ and $ \lim_{n\to\infty}\tau_n=\tau $ almost surely and a process $ u:\Omega\times [0,\tau)\to W^{1,q}_0(D) $ such that $ u $ solves $ (\operatorname{GDIV}) $ on $ [0,\tau_n] $ and
\begin{equation}\label{lp_quasi_example_global_starting_estimate}
\E\|u\ind_{[0,\tau_n]}\|_{L^{p}(0,T;W^{1,q}_0(D))}^{p}+\E\sup_{t\in[0,\tau_n]}\|u(t)\|_{B^{1-2/p}_{q,p,0}(D)}^{p}<\infty
\end{equation}
for every $ n\in\N. $ 

In this section, we aim to prove, that we actually have $ \tau=T $ almost surely. By the blow-up alternative from Theorem \ref{lp_quasi_example_local_solution}, it is sufficient to show that $ u:[0,\tau)\to B^{1-2/p}_{q,p,0}(D) $ is pathwise almost surely uniformly continuous and satisfies $ \|u\|_{L^{p}(0,\tau;W^{1,q}_{0}(D))}<\infty $. However, this is not too easy, since the estimate \eqref{lp_quasi_example_global_starting_estimate}, that originally comes from the abstract construction of a solution of a quasilinear equation, depends on $ n $ and to find uniform estimates, we have to use the special structure of our equation. 

Our first goal is to derive uniform $ L^{\alpha}(\Omega;L^{\infty}(0,\tau_n;L^{\alpha}(D))) $-estimates for $ u $ with $ \alpha\in [2,\infty) $ and to do this, we need a suitable version of the It\^o formula, that is useful to deal with weak solutions of stochastic partial differential equations. In \cite{debussche_arnoud_hofmanova_degenerate_quasilinear}, the authors developed such a tool for equations on the torus $ \mathbb{T}^{d} $ with periodic boundary conditions. We give an analogous result under assumptions that are adjusted to our situation. The proof follows the same lines.

\begin{Lemma}\label{lp_quasi_example_appropriate_ito_formula}
	Let $ \mu:\Omega\to[0,T] $ be an $ \F $-stopping time and $ \phi\in C^{2}(\R) $ with $ \phi(0)=0 $ and with a bounded second derivative. Assume, that the $ \F $-adapted process $ u:\Omega\times[0,\mu]\to\R $ with
	$ u\ind_{[0,\mu]}\in L^{2}(\Omega\times[0,T];W^{1,2}_{0}(D)) $ is pathwise continuous on $ [0,\mu] $ viewed as function with values in $ L^{2}(D) $ and satisfies $ \E\sup_{0\leq t\leq\mu}\|u(t)\|_{L^{2}(D)}^{2}<\infty $. Furthermore, we assume $ u $ to be of the form
	\begin{align}\label{lp_quasi_example_appropriate_ito_formula_help}
	u(t)-u_0=\int_{0}^{t}\Div F(s)\operatorname{ds}+\int_{0}^{t}H(s)dW(s)
	\end{align}
	almost surely for all $ t\in [0,\mu] $ in $ W^{-1,2}(D) $ with an $ \mathcal{F}_0 $-measurable initial value $ u_0\in L^{2}(D) $ , an $ \F $-adapted $ H\ind_{[0,\mu]}\in L^{2}(\Omega\times[0,T]\times D;l^2(\N)) $ and with $ F\ind_{[0,\mu]}\in L^{2}(\Omega\times[0,T]\times D)^{d}. $ Then, the generalized It\^o formula 
	\begin{align*}
	 \int_{D}\phi(u(t,x))-\phi(u_0(x))\operatorname{dx}=&-\int_{0}^{t}\int_{D}\phi''(u(s,x))\nabla u(s,x)F(s,x)\operatorname{ds}\\
	 &+\int_{0}^{t}\int_{D}\phi'(u(s,x))H(s,x)\operatorname{dx} dW(s)\\
	 &+\frac{1}{2}\sum_{n=1}^{\infty}\int_{0}^{t}\int_{D}\phi''(u(s,x))H^{2}_n(s,x)\operatorname{dx}\operatorname{ds}
	\end{align*}
	holds almost surely for all $ t\in [0,\mu] $.
\end{Lemma}

The following Lemma was used several times in the literature in a comparable situation (see e.g. \cite{hofmanova_zhang_existence_uniqueness_quasilinear}, Theorem $ 3.1 $ or \cite{debussche_arnoud_hofmanova_degenerate_quasilinear}, Proposition $ 5.1 $). The difference is, that we deal with Dirichlet boundary conditions, whereas the references consider periodic boundary conditions on the torus and that we work on a random interval up to a stopping time.

\begin{Lemma}\label{lp_quasi_example_global_uniform_lp_estimate}
	If we assume $ [\operatorname{GD1}] $-$ [\operatorname{GD7}] $ and additionally $ u_0\in L^{\alpha}(\Omega\times D) $ for some $ \alpha\in[2,\infty) $, we have
	\begin{align*}
	\big(\E\sup_{0\leq t< \tau}\|u(t)\|_{L^{\alpha}(D)}^{\alpha}\big)^{1/\alpha}\leq C_\alpha(1+\|u_0\|_{L^{\alpha}(\Omega\times D)})
	\end{align*}
	with a constant $ C_\alpha>0 $ independent of $ u_0 $. Moreover, we have 
	\[ \|u\ind_{[0,\tau)}\|_{L^{2}(\Omega\times [0,T];W^{1,2}_{0}(D))}<\infty. \]
\end{Lemma}
\begin{proof}
	We fix $ m\in\N $ and work on the interval $ [0,\tau_m]. $ By \eqref{lp_quasi_example_global_starting_estimate}, the process $ u$ satisfies $u\ind_{[0,\tau_m]}\in L^{2}(\Omega\times [0,T];W^{1,2}_{0}(D)) $, has almost surely continuous paths viewed as function with values in $ L^{2}(D) $ with $ \E\sup_{0\leq t\leq\tau_m}\|u\|^2_{L^{2}(D)}<\infty $ and can be written in the form
	\[ u(t)-u_0=\int_{0}^{t}\Div(a(u(s))\nabla u(s))+\Div (G(u(s)))\operatorname{ds}+\int_{0}^{t}B(u(s))dW(s) \]
	almost surely for $ t\in [0,\tau_m] $ and is therefore an It\^o  process in $ W^{-1,2}_{0}(D)$. 
	
	We aim to apply the It\^o formula from Lemma \ref{lp_quasi_example_appropriate_ito_formula} to the function $ \phi(\xi)=|\xi|^{\alpha}, $ but unfortunately, its second derivative is not bounded, if we are not in the case $ \alpha=2 $. Therefore, we have to start with an approximation of $ \phi $. We use
	\begin{equation*}
	\phi_n(\xi)=\begin{cases}
	|\xi|^{\alpha},&|\xi|\leq n,\\
	n^{\alpha-2}\big(\frac{\alpha(\alpha-1)}{2}\xi^{2}-\alpha(\alpha-2)n|\xi|+\frac{(\alpha-1)(\alpha-2)}{2}n^{2} \big),&|\xi|> n
	\end{cases}
	\end{equation*}
	for every $ n\in\N $ if $ \alpha\in (2,\infty). $ If $ \alpha=2 $, we set $ \phi_n(\xi):=\xi^{2} $.
	
	One has $ \lim_{n\to\infty}\phi_n(\xi)=|\xi|^{\alpha} $ pointwise for all $ \xi\in\R $, $ \phi_n(0)=0 $ and
	$ \phi_n''(\xi)\geq 0 $ for all $ \xi\in\R. $ Moreover, by calculation, one can derive the estimates
	\begin{align}\label{lp_quasi_example_global_uniform_lp_estimate_bla}
	|\xi\phi_n'(\xi)|&\leq \alpha\phi_n(\xi),\notag\\
	|\phi_n'(\xi)|&\leq \alpha (1+\phi_n(\xi)),\notag\\
	\xi^{2}\phi_n''(\xi)&\leq \alpha(\alpha-1)\phi_n(\xi),\\
	\phi''(\xi)&\leq \alpha(\alpha-1)(1+\phi_n(\xi))\notag
	\end{align}
	for all $ \xi\in\R,n\in\N $. Now we are in the position to apply It\^o's formula and obtain
	\begin{align}\label{lp_quasi_example_global_uniform_lp_estimate_mup}
	\int_{D}\phi_n(u(x,t))\operatorname{dx}=& \int_{D}\phi_n(u_0(x))\operatorname{dx}\notag\\
	&-\int_{0}^{t}\int_{D}\phi_n''(u(s,x))\nabla u(s,x)\big(G(u(s,x))+a(u(s,x))\nabla u(s,x)\big)\operatorname{dx}\operatorname{ds}\notag\\
	&+\int_{0}^{t}\int_{D}\phi_n'(u(s,x))B(s,x)\operatorname{dx}dW(s)\notag\\
	&+\frac{1}{2}\sum_{k=1}^{\infty}\int_{0}^{t}\int_{D}\phi_n''(u(s,x))B_k(s,x,u(s,x))^{2}\operatorname{dx}\operatorname{ds}
	\end{align}
	almost surely for all $ t\in [0,\tau_m]. $ Next, we estimate $ \E\sup_{0\leq s\leq t\wedge\tau_m}\int_{D}\phi_n(u(x,s))\operatorname{dx} $ term by term beginning with the stochastic integral. Applying the scalar valued Burkholder-Davis-Gundy inequality, the assumptions on $ B $ and the estimates \eqref{lp_quasi_example_global_uniform_lp_estimate_bla}, we obtain
	\begin{align*}
	\E\sup_{0\leq s\leq t\wedge\tau_m}&\Big|\int_{0}^{s}\int_{D}\phi_n'(u(r,x))B(r,x,u(r,x))\operatorname{dx}dW(r) \Big|\\
	&\lesssim \E\Big(\int_{0}^{ t\wedge\tau_m}\sum_{k=1}^{\infty}\Big(\int_{D}\phi_n'(u(r,x))B_k(r,x,u(r,x))\operatorname{dx}\Big)^{2}\operatorname{dr} \Big)^{1/2}\\
	&\lesssim \E\Big(\int_{0}^{ t\wedge\tau_m}\|\phi_n'(u(r))^{\tfrac{1}{2}}u(r)^{\tfrac{1}{2}}\|_{L^{2}(D)}^2 \|\phi_n'(u(r))^{\tfrac{1}{2}}u(r)^{-\tfrac{1}{2}}B(r,u(r))\|_{L^{2}(D;l^2(\N)}^{2}\operatorname{dr}\Big)^{1/2}\\
	&\lesssim \E\Big(\int_{0}^{ t\wedge\tau_m}\|\phi_n(u(r))^{\tfrac{1}{2}}\|^2_{L^{2}(D)}\|1+\phi_n(u(r))^{\tfrac{1}{2}}\|_{L^{2}(D)}^{2}\operatorname{dr}\Big)^{1/2}\\
	&\lesssim \E\Big(\int_{0}^{ t\wedge\tau_m}\int_{D}\phi_n(u(r,x))\operatorname{dx}\big(1+\int_{D}\phi_n(u(r,x))\operatorname{dx}\big) \Big)^{1/2}\\
	&\leq \E\Big(\sup_{0\leq s\leq t\wedge\tau_m}\int_{D}\phi_n(u(r,x))\operatorname{dx} \Big)^{1/2}\Big(1+\int_{0}^{ t\wedge\tau_m}\int_{D}\phi_n(u(r,x))\operatorname{dx}\operatorname{dr} \Big)^{1/2}.
	\end{align*}
	Finally, the well-known estimate $ ab\leq \varepsilon a^{2}+C_\varepsilon b^{2} $ for $ a,b> 0 $ yields
	\begin{align*}
	\E\sup_{0\leq s\leq t\wedge\tau_m}&\Big|\int_{0}^{s}\int_{D}\phi_n'(u(r,x))B(r,x,u(r,x))\operatorname{dx}dW(r) \Big|\\
	&\leq\frac{1}{2} \E\sup_{0\leq s\leq t\wedge\tau_m}\int_{D}\phi_n(u(r,x))\operatorname{dx}+C\Big(1+\E\int_{0}^{ t\wedge\tau_m}\int_{D}\phi_n(u(r,x))\operatorname{dx}\operatorname{dr} \Big)
	\end{align*}
	for some $ C>0. $ 
%
	We proceed with the deterministic terms in \eqref{lp_quasi_example_global_uniform_lp_estimate_mup}. Since $ \phi''\geq 0 $ and $ a $ is coercive,
	\begin{align*}
	-\phi_n''(u(s,x))\nabla u(s,x)a(u(s,x)\nabla u(s,x)
	\end{align*}
	is almost surely for all $ s\in[0,\tau_m] $ and all $ x\in D $ non-positive and the corresponding term can be dropped in an upper estimate. Moreover, the divergence theorem of Gauss and $ u(t,x)=0 $ almost surely for $ t\in [0,\tau_m] $ and $ x\in\partial D $ yields
	\begin{align*}
	\int_{D}\phi_n''(u(s,x))\nabla u(s,x)G(u(s,x))\operatorname{dx}&=\int_{D}\Div\Big(\int_{0}^{u(t,x)}\phi_n''(\xi)G(s,\xi)\operatorname{d\xi} \Big)\operatorname{dx}\\
	&=\int_{\partial D}\Big(\int_{0}^{u(t,x)}\phi_n''(\xi)G(s,\xi)\operatorname{d\xi} \Big)\nu\operatorname{d\sigma(x)}=0.
	\end{align*}
	The last remaining term can be estimated with the assumptions on $ B $ and \eqref{lp_quasi_example_global_uniform_lp_estimate_bla}. 
	\begin{align*}
	\E\sup_{0\leq s\leq t\wedge\tau_m}&\Big|\sum_{k=1}^{\infty}\int_{0}^{s}\int_{D}\phi_n''(u(r,x))B_k(r,x,u(r,x))^{2}\operatorname{dx}\operatorname{dr}\Big|\\
	&\lesssim \E\sup_{0\leq s\leq t\wedge\tau_m}\Big|\int_{0}^{s}\int_{D}\phi_n''(u(r,x))(1+u(r,x))^{2}\operatorname{dx}\operatorname{dr}\\
	&\lesssim \E\int^{t\wedge\tau_m}_0\int_{D}1+\phi_n(u(r,x))\operatorname{dx}\operatorname{dr}\\
	&\lesssim 1+\int_{0}^{t}\E\sup_{0\leq s\leq r\wedge\tau_m}\int_{D}\phi_n(s,x)\operatorname{dx}\operatorname{dr}.
	\end{align*}
	All in all, we proved
	\begin{align*}
	E\sup_{0\leq s\leq t\wedge\tau_m}\int_{D}\phi_n(u(x,s))\operatorname{dx}\lesssim\ 1+\E \int_{D}\phi_n(u_0(x))\operatorname{dx}+\int_{0}^{t}\E\sup_{0\leq s\leq r\wedge\tau_m}\int_{D}\phi_n(s,x)\operatorname{dx}\operatorname{dr}
	\end{align*}
	and hence with Gronwall,
	\begin{align*}
	\E\sup_{0\leq s\leq t\wedge\tau_m}\int_{D}\phi_n(u(x,s))\operatorname{dx}\lesssim\ 1+\E \int_{D}\phi_n(u_0(x))\operatorname{dx}
	\end{align*}
	for every $ t\in [0,T] $ and $ n\in\N $ and the estimate is independent of $ n. $ We want to finish the proof by applying Fatou to pass to the limit $ n\to\infty. $ Note that one can interchange $ \sup $ and $ \liminf $ in an upper estimate, since $ \liminf $ can be written in the form $ \sup \inf $ and supremums can be interchanged, whereas $ \sup\inf\leq \inf\sup $. Thus, we have
	\begin{align*}
	\E\sup_{0\leq s\leq t\wedge\tau_m}\int_{D}|u(x,s)|^{\alpha}\operatorname{dx}&\leq\liminf_{n\to\infty}\E\sup_{0\leq s\leq t\wedge\tau_m}\int_{D}\phi_n(u(x,s))\operatorname{dx}\\
	&\lesssim 1+\liminf_{n\to\infty} \E \int_{D}\phi_n(u_0(x))\operatorname{dx}
	\end{align*}
	and the last term equals $ \E\|u_0\|_{L^{\alpha}(D)}, $ which can be proved with Lebesgue's dominated convergences theorem. This proves
		\begin{align*}
		\big(\E\sup_{0\leq t\leq \tau_m}\|u(t)\|_{L^{\alpha}(D)}^{\alpha}\big)^{1/\alpha}\lesssim 1+\|u_0\|_{L^{\alpha}(\Omega\times D)}
		\end{align*}
	for every $ m\in\N $. The first claim now follows from another application of Fatou's Lemma. For the second claim, we have to look at \eqref{lp_quasi_example_global_uniform_lp_estimate_mup} in the special case $ \alpha=2. $ We get
	\begin{align*}
	\|u(t)\|_{L^{2}(D)}^{2}=& \|u_0\|_{L^{2}(D)}^{2}-2\int_{0}^{t}\int_{D}\nabla u(s,x)a(u(s,x))\nabla u(s,x)\operatorname{dx}\operatorname{ds}\notag\\
	&+2\int_{0}^{t}\int_{D}u(s,x)B(s,x,u(s,x))\operatorname{dx}dW(s)+\sum_{k=1}^{\infty}\int_{0}^{t}\int_{D}B_k(s,x,u(s,x))^{2}\operatorname{dx}\operatorname{ds}
	\end{align*}
	almost surely for all $ t\in [0,\tau_m]. $ Coercivity of $ a(u(s,x)) $ then yields
	\begin{align*}
	-\int_{D}\nabla u(s,x)&a(u(s,x))\nabla u(s,x)\operatorname{dx}\leq  -\delta_0\|\nabla u(s)\|_{L^{2}(D)}^{2}
	\end{align*}
	for some $ \varepsilon\in (0,\delta_0) $ and $ C_\varepsilon>0 $. As a consequence, we have
	\begin{align*}
	\delta_0\int_{0}^{t}\|\nabla u(s)\|_{L^{2}(D)}^{2}\operatorname{ds}\leq& \|u_0\|_{L^{2}(D)}^{2}+ 2 \int_{0}^{t}\int_{D}u(s,x)B(s,x,u(s,x))\operatorname{dx}dW(s) \notag\\
	&+\frac{1}{2}\sum_{k=1}^{\infty}\int_{0}^{t}\int_{D}B_k(s,x,u(s,x))^{2}\operatorname{dx}\operatorname{ds}
	\end{align*}
	and with the estimates we already did before and
			\begin{align*}
			\big(\E\sup_{0\leq t< \tau}\|u(t)\|_{L^2(D)}^{2}\big)^{1/2}\leq C_2(1+\|u_0\|_{L^{2}(\Omega\times D)})
			\end{align*}
	we get
			\begin{align*}
			\big(\E\|\nabla u\ind_{[0,\tau)}\|_{L^{2}([0,T]\times D)}^{2}\operatorname{ds}\big)^{1/2}\lesssim (1+\|u_0\|_{L^{2}(\Omega\times D)}).
			\end{align*}
			This finishes the proof.
\end{proof}
As a consequence of these estimates, we can extend $ u $ pathwise to a continuous function with values in $ L^{2}(D) $ on the closed interval $ [0,\tau]. $
\begin{Lemma}\label{lp_quasi_example_global_continuity_up_to_tau}
If we assume $ [\operatorname{GD1}] $-$ [\operatorname{GD7}] $ and additionally $ u_0\in L^{2}(\Omega\times D) $ the function $ u:[0,\tau)\to L^{2}(D) $ is pathwise almost surely uniformly continuous and can be extended to a continuous function on $ [0,\tau]. $
\end{Lemma}
\begin{proof}
We know, that $ u $ is an It\^o process in $ W^{-1,2}(D) $ and that we have
\[ u(t)-u_0=\int_{0}^{t}\big(\Div(a(u(s))\nabla u(s))+\Div (G(u(s)))\big)\operatorname{ds}+\int_{0}^{t}B(u(s))dW(s) \]
for every $ t\in [0,\tau) $ and by Lemma \ref{lp_quasi_example_global_uniform_lp_estimate}, we have $ u\in L^{2}(0,\tau;W^{1,2}_{0}(D)) $ pathwise almost surely. Moreover, by $ [\operatorname{GD7}] $ and It\^o's isometry, we obtain
\begin{align*}
\|t\mapsto\int_{0}^{t}B(u(s))\ind_{[0,\tau)}(s)dW(s)\|_{L^{2}(\Omega\times [0,T];W^{1,2}_0(D))}&=\|B(u)\ind_{[0,\tau)}\|_{L^2(\Omega\times[0,T]\times\N;W^{1,2}_0(D))}\\
&\lesssim 1+\|u\|_{L^{2}(\Omega\times[0,T];W^{1,2}_{0}(D))}<\infty
\end{align*}
and so we also have $ t\mapsto\int_{0}^{t}B(u(s))dW(s)\in L^{2}(0,\tau;W^{1,2}_{0}(D)) $ pathwise almost surely. Consequently, we have
 $$ t\mapsto u_0+ \int_{0}^{t}\Div(a(u(s))\nabla u(s))+\Div (G(u(s)))\operatorname{ds}\in L^{2}(0,\tau;W^{1,2}_{0}(D)) $$
  pathwise almost surely. On the other hand, the fundamental theorem of calculus yields $ t\mapsto u_0+ \int_{0}^{t}(\Div(a(u(s))\nabla u(s))+\Div (G(u(s))))\operatorname{ds}\in W^{1,2}(0,\tau;W^{-1,2}(D)) $ almost surely. Since the embedding
\[ W^{1,2}(0,\tau;W^{-1,2}(D))\cap L^{2}(0,\tau;W^{1,2}_{0}(D))\hookrightarrow C(0,\tau;L^{2}(D)) \]
is bounded, $ t\mapsto u_0+ \int_{0}^{t}\Div(a(u(s))\nabla u(s))+\Div (G(u(s)))\operatorname{ds}$ is uniformly continuous on $ [0,\tau) $ viewed as a function in $ L^{2}(D). $ Clearly, by the Burkholder-Davies-Gundy inequality, the same holds true for the stochastic integral. This closes the proof.
\end{proof} 

In the previous Lemmatas, we extended our local solution $ \big(u,(\tau_n)_n,\tau \big) $ to the closed interval $ [0,\tau] $ and derived estimates for $ u $ on $ [0,\tau]. $ As a consequence, we can apply a regularity result for quasilinear stochastic evolution equations in divergence form, that yields additional regularity properties for $ u. $ It turns out, that $ u $ is even pathwise H\"older continuous in space and time.

\begin{Lemma}\label{lp_quasi_example_global_hoelder_regularity}
If we assume $ (\operatorname{GD1}) $-$ (\operatorname{GD7}) $ and $ u_0\in L^{m}(\Omega\times D) $ for every $ m\in [2,\infty) $, the process $ u:\Omega\times[0,\tau]\times D\to\R $ is pathwise H\"older- continuous in space and time. More precisely there exists $ \eta>0 $, such that 
\begin{align*}
\E\Big( \sup\limits_{t\in[0,\tau],x\in D}|u(t,x)|+\sup\limits_{t,s\in[0,\tau],x,y\in D}\frac{|u(t,x)-u(s,y)|}{\max\{|t-s|^{\eta},|x-y|^{2\eta} \}}\Big)^{m}<\infty
\end{align*}
for every $ m\in [2,\infty). $
\end{Lemma}
\begin{proof}
By Lemma \ref{lp_quasi_example_global_uniform_lp_estimate} and Lemma \ref{lp_quasi_example_global_continuity_up_to_tau}, we have
\begin{align*}
u\ind_{[0,\tau]} \in L^{m}(\Omega;L^{\infty}(0,T;L^{m}(D)))\cap L^{2}(\Omega\times [0,T];W^{1,2}_{0}(D))
\end{align*}
for all $ m\in [2,\infty) $ and $ u:[0,\tau]\to L^{2}(D) $ is pathwise uniformly continuous. Moreover, our initial value $ u_0\in B^{1-2/p}_{q,p,0}(D) $ satisfies $ u_0=0 $ almost surely on $ \partial D $, since required $ 1-2/p>d/q $. Thus a slight variation of \cite{debussche_arnoud_hofmanova_regularity_quasilinear}, Theorem $ 2.6 $ implies the claimed result. The only change we need is that we investigate the equation on the random interval $[0,\tau] $ instead of $ [0,T]. $ However, in the proof of Theorem $ 2.6  $ one can replace $ T $ by $ \tau $ without further difficulties, since they authors argue pathwise with a classical regularity result about deterministic parabolic equations by Ladyzhenskaya, Solonnikov and Uralceva (see \cite{ladyzenskaja_linear_and_quasilinear_equations}, Theorem $ 10.1 $ in Chapter $ \operatorname{III} $). In \cite{debussche_arnoud_hofmanova_regularity_quasilinear}, Theorem $ 2.6 $, $ \partial D $ was assumed to be smooth, but to apply Ladyzhenskaya's result, a piecewise $ C^{1} $-boundary combinded with the so called condition $ A, $ that is explained in \cite{ladyzenskaja_linear_and_quasilinear_equations} on page $ 9,$ is sufficient. However, with a moment of consideration one checks, that our assumption of a $ C^{1} $-boundary implies this condition $ A $.
\end{proof}

Finally, we can prove the main theorem of this section. We show that our local solution $ u $ is indeed a global solution, that exists on the whole interval $ [0,T]. $ 
For this proof, we compare $ u $ with the solution $ z $ of a stochastic heat equation with the noise $ B(u(t))dW(t). $ Then, we investigate the regularity properties of $ u-z $, which solves a non-autonomous deterministic partial differential equation with a random parameter, by applying results on maximal regularity for both the stochastic heat equation and for the arising non-autonomous equation.

\begin{Theorem}\label{lp_quasi_example_global_global existence}
If we assume $ (\operatorname{GD1}) $-$ (\operatorname{GD7}) $, the local solution $ \big(u,(\tau_n)_n,\tau\big) $ of $ (\operatorname{GDIV}) $ is a global solution, i.e. we have $ \tau=T $ almost surely and the solution satisfies
\[ u\in L^{p}(0,T;W^{1,q}_{0}(D))\cap C(0,T;B^{1-2/p}_{q,p,0}(D)) \]
pathwise almost surely.
\end{Theorem}

\begin{proof}
We first check the theorem for $ u_0\in L^{\infty}(\Omega;B^{1-2/p}_{q,p,0}(D)). $ By Theorem \ref{lp_quasi_example_local_solution}, there exists a local solution $\big( u,(\tau_n)_n,\tau\big) $ of $ (\operatorname{GDIV}) $ to the initial value $ u_0. $ Since we chose $ 1-2/p>d/q, $ we have $ u_0\in L^{m}(\Omega\times D) $ for all $ m\in [2,\infty) $ and as a consequence, Lemma \ref{lp_quasi_example_global_hoelder_regularity} implies, that $ u:\Omega\times[0,\tau]\times D\to\R $ is pathwise almost surely uniformly continuous in space and time and $ u\ind_{[0,\tau]}\in L^{m}(\Omega;L^{\infty}(0,T;L^{m}(D))) $.

Next, we consider the equation
\begin{equation*}
\begin{cases}
dz(t)=\Delta z(t)\operatorname{dt}+B(u(t))dW(t) \quad \text{ for } t\in[0,T],\\
z(0)=0.
\end{cases}
\end{equation*}
By $ (\operatorname{GD7}) $, we have $ B(u)\in L^{p}(\Omega\times[0,T];\gamma(l^{2};L^{q}(D))). $ Therefore the maximal $L^{p}$-regularity result for stochastic evolution equations, Theorem \ref{lp_auto_stoch_evolution_equation_global_strong_solution_lp_initial_data}, yields a unique solution 
\[ z\in L^{p}(\Omega\times[0,T];W^{1,q}_{0}(D))\cap L^{p}(\Omega;C(0,T;B^{1-2/p}_{q,p,0}(D))). \]
If we now investigate the difference $ y:=u-z $ on $ [0,\tau], $ we find out, that $ y $ pathwise almost surely solves the deterministic non-autonomous parabolic equation
\begin{equation}\label{lp_quasi_example_global_global existence_1}
\begin{cases}
y'(t)=[\Div(a(u(t))\nabla y(t))+\Div(G(u(t)))+\Div((a(u(t))-I)\nabla z(t)),\\
y(0)=u_0.
\end{cases}
\end{equation}
Note that any solution of this equation in $ L^{2}(0,\tau;W^{1,2}_0(D)) $ is unique by a classical result of Lions for non-autonomous evolution equations governed by forms (see e.g. \cite{showalter_monotone_operators_banach_spaces}, Chapter $ \operatorname{III}, $ Proposition $ 2.3. $)

As a next step, we prove that this equation has deterministic maximal $ L^{p} $-regularity. We estimate
\begin{align*}
\|\Div(a(u(t))\nabla x)-\Div(a(u(s))\nabla x)\|_{W^{-1,q}(D)}&\leq \|(a(u(t))- a(u(s)))\nabla x\|_{L^{q}(D)}\\
&\leq \sup_{x\in D}|a(u(t,x))-a(u(s,x))|\|x\|_{W^{1,q}_{0}(D)}\\
&\lesssim \sup_{x\in D}|u(t,x)-u(s,x)|\|x\|_{W^{1,q}_{0}(D)}
\end{align*}
and since $ u $ is pathwise almost surely uniformly continuous on $ [0,\tau]\times D $ (see Lemma \ref{lp_quasi_example_global_hoelder_regularity}), the mapping $[0,\tau]\ni t\mapsto \Div(a(u(t))\nabla)\in B(W^{1,q}_{0}(D),W^{-1,q}(D) $ is almost surely continuous. Moreover, as we have seen in Lemma \ref{lp_quasi_example_div_properties_of_the_operators}, the operator $ \Div(a(u(t))\nabla)  $ has almost surely for fixed $ t\in[0,\tau] $ a bounded $ H^{\infty} $-calculus on $ W^{-1,q}(D) $ and its domain is given by $ W^{1,q}_{0}(D). $ Therefore we can apply \cite{pruss_schnaubelt_solvability_maximal_regularity_of_parabolic_evolution_equations}, Theorem $ 2.5. $ and obtain the pathwise almost surely maximal $ L^{p} $-regularity of the non-autonomous equation \eqref{lp_quasi_example_global_global existence_1}. Moreover, we both have $ \Div( G(u))\in L^{p}(0,\tau;W^{-1,q}(D)) $ and $ \Div((a(u)-I)\nabla z) \in L^{p}(0,\tau;W^{-1,q}(D))$. Indeed $ [\operatorname{GD6}] $ and the regularity of $ z $ together with $ [\operatorname{GD2}] $ imply
\[ \|\Div(G(u))\|_{ L^{p}(0,\tau;W^{-1,q}(D))}\lesssim \|G(u)\|_{L^{p}(0,\tau;L^{q}(D))}\lesssim 1+\|u\|_{L^{p}(0,\tau;L^{q}(D))}, \]
\[ \|\Div((a(u)-I)\nabla z)\|_{L^{p}(0,\tau;W^{-1,q}(D))}\lesssim \|(a(u)-I)\nabla z\|_{L^{p}(0,\tau;L^{q}(D)}\lesssim \|z\|_{L^{p}(0,\tau;W^{1,q}_{0}(D))}.  \]
As a consequence of maximal regularity, we have
\begin{align*}
\|y\|_{L^{p}(0,\tau;W^{1,q}_{0}(D))}&+\|y\|_{C(0,\tau;B^{1-2/p}_{q,p,0}(D))}\\
&\leq C_{\operatorname{MR}}\big( \|\Div(G(u))\|_{ L^{p}(0,\tau;W^{-1,q}(D))}+\|\Div((a(u)-I)\nabla z)\|_{L^{p}(0,\tau;W^{-1,q}(D))}\big)\\
&\lesssim 1+\|u\|_{L^{p}(0,\tau;L^{q}(D))}+\|z\|_{L^{p}(0,\tau;W^{1,q}_{0}(D))}
\end{align*}
and thus $ y\in L^{p}(0,\tau;W^{1,q}_{0}(D))\cap C(0,\tau;B^{1-2/p}_{q,p,0}(D)) $ pathwise almost surely. With the unique solvability in $ L^{2}(0,\tau;W^{1,2}_0(D)) $ and $ u=y+z $ one sees that $ u $ is also pathwise almost surely in the space $L^{p}(0,\tau;W^{1,q}_{0}(D))\cap C(0,\tau;B^{1-2/p}_{q,p,0}(D)). $ Hence the blow-up alternative from Theorem \ref{lp_quasi_example_local_solution} yields $ \tau=T $ almost surely, which is the desired result. 

Last but not least, we have to deal with arbitrary initial values $ u_0:\Omega\to B^{1-2/p}_{q,p,0}(D). $ Defining $ \Lambda_n:=\{\|u_0\|_{B^{1-2/p}_{q,p,0}(D)}<n \} $ and the truncated initial values $ u_0^{(n)}:=u_0\ind_{\Lambda_n}, $ we can apply the result we derived above and we get global solutions $ u_n $ of $ (\operatorname{GDIV}) $ to the initial value $ u_0^{(n)} $, that pathwise almost surely satisfy 
$u_n\in L^{p}(0,T;W^{1,q}_{0}(D))\cap C(0,T;B^{1-2/p}_{q,p,0}(D)). $ By Corollary \ref{lp_quasi_stochastic_parabolic_local_uniqueness_initial_value}, the solutions $ u_n $ and $ u_m $ coincide on $ \Lambda_{n\wedge m}$ and therefore the pointwise limit $ u=\lim_{n\to\infty}u_n $ is a well-defined adapted process. Moreover, since for almost all $ \omega\in\Omega $ there is an $ n(\omega) $ such that $ u(\omega,\cdot)=u_{n(\omega)}(\omega,\cdot), $ $ u $ solves $ (\operatorname{GDIV}) $ and has pathwise almost surely the claimed regularity.
\end{proof}
The reader may ask, why we could not prove
\[ u\in L^{p}(\Omega\times[0,T];W^{1,q}_0(D))\cap L^{p}(\Omega;C(0,T;B^{1-2/p}_{q,p,0}(D))) \]
under the additional assumption $ u_0\in L^{p}(\Omega\times[0,T];B^{1-2/p}_{q,p,0}(D)) $. This is due to the maximal regularity result for non-autonomous deterministic equations we used. The maximal regularity constant $ C_{\operatorname{MR}} $ highly depends on the modulus of continuity of the coefficient function which is in our case given by $ a(u(\omega,t,x)) $. Therefore, $ C_{\operatorname{MR}} $ depends on the modulus of continuity of $ u $ itself, but this one differs from path to path and cannot be controlled uniformly in $ \omega. $ So, the best estimate, we can achieve is
\begin{align*}
 \|u(\omega,\cdot)\|_{L^{p}(0,T;W^{1,q}_0(D))}&=\|y(\omega,\cdot)+z(\omega,\cdot)\|_{L^{p}(0,T;W^{1,q}_0(D))}\\
 &\leq CC_{\operatorname{MR}}(\omega)\big(1+\|u(\omega,\cdot)\|_{L^{p}(0,\tau;L^{q}(D))}+\|z(\omega,\cdot\|_{L^{p}(0,\tau;W^{1,q}_{0}(D))}\big)
\end{align*} 
for almost $ \omega\in\Omega $, but it is impossible to control $\|u\|_{L^{p}(\Omega\times[0,T];W^{1,p}_0(D))}$ in this way. One would need a significantly stronger result on maximal $ L^{p} $- regularity for non-autonomous deterministic equations with $ C_{MR} $ only depending on the upper bound and the ellipticity constant of the coefficient function $ a(u(\omega,t,x)). $ Unfortunately, such a result is only known for $ p=2 $ by a classical result of Lions and for $ p\in [2-\varepsilon,2+\varepsilon] $ for some small $ \varepsilon>0 $ by a recent result of Disser, ter Elst and Rehberg (see \cite{disser_elst__rehberg_maximal_regularity_nonautonomous}, Proposition $ 6.3 $). This can be used to prove at least  
 $$ u\in L^{p}(\Omega;L^{r_1}(0,T;W^{1,r_2}_0(D))) $$
for $ r_1,r_2\in [2-\varepsilon,2+\varepsilon] $.
    \subsection{Acknowledgement}
I gratefully acknowledge financial support by the Deutsche Forschungs\-gemeinschaft (DFG) through CRC 1173. Moreover, I thank my advisor Lutz Weis and Roland Schnaubelt for many useful discussions and for pointing out references on the subject. I am also grateful, that Moritz Egert and Robert Haller-Dintelmann showed me the precise dependence of the $ H^{\infty} $-calculus for divergence-form operators with mixed boundary conditions on the coefficients in Lemma \ref{lp_quasi_example_div_properties_of_the_operators}. Last but not least, I want to mention the help of Fabian Hornung and Christine Grathwohl. They read the article carefully and gave many useful comments.\newpage
    \bibliographystyle{abbrv} 
    \bibliography{Literaturverzeichnis}

\end{document}